\documentclass[12pt]{amsart}
\usepackage[utf8]{inputenc}
\usepackage{amssymb,url,mathdots,verbatim,mathtools,nccmath}
\usepackage[dvipsnames]{xcolor}
\usepackage[margin = 2.5cm]{geometry}
 
\usepackage{thmtools}

\usepackage{caption}
\usepackage{tikz}
\usetikzlibrary{calc,decorations}
\usepackage[outline]{contour}
\contourlength{2pt}
\usepackage[aligntableaux=center]{ytableau}
\usepackage{MnSymbol}
\usepackage{graphicx}
\usepackage{booktabs}
\usepackage[inline]{enumitem}
\usepackage[normalem]{ulem}
\pgfdeclaredecoration{arrows}{draw}{
	\state{draw}[width=\pgfdecoratedinputsegmentlength]{%
		\path [every arrow subpath/.try] \pgfextra{%
			\pgfpathmoveto{\pgfpointdecoratedinputsegmentfirst}%
			\pgfpathlineto{\pgfpointdecoratedinputsegmentlast}%
		};
}}
\tikzset{every arrow subpath/.style={->, draw, ultra thick}}

\newtheorem{theorem}{Theorem}[section]

\newtheorem{cor}[theorem]{Corollary}
\newtheorem{lemma}[theorem]{Lemma}
\newtheorem{definition}[theorem]{Definition}

\theoremstyle{definition}
\newtheorem{remark}[theorem]{Remark}
\numberwithin{equation}{section}

\usepackage{xspace}
\newcommand{\CSPP}{CStrPP\xspace}
\newcommand{\CSPPs}{CStrPPs\xspace}
\newcommand{\RSPP}{RStrPP\xspace}
\newcommand{\RSPPs}{RStrPPs\xspace}

\let\sp\relax
\DeclareMathOperator{\sp}{sp}

\DeclareMathOperator{\sgn}{sgn}
\DeclareMathOperator{\inv}{inv}

\newcommand{\nenwarrow}{\nwarrow \!\!\!\!\!\;\!\! \nearrow}
\newcommand{\e}{{\operatorname{E}}}
\newcommand{\W}{{\operatorname{w}}}
\newcommand{\asym}{\mathbf{ASym}}

\newcommand{\horizontaltile}{%
	\tikz[scale=.35]{%
		\draw (0,0) -- (0.866,-.5) -- (1.732,0) -- (0.866,.5) -- cycle;
	}%
}

\definecolor{verylight-gray}{gray}{0.95}
\definecolor{light-gray}{gray}{0.9}
  
\colorlet{mygreen}{Green}
\colorlet{myyellow}{yellow}
\colorlet{myblue}{NavyBlue}
\colorlet{myred}{BrickRed}

\usepackage[colorlinks=true]{hyperref}
\title[Vertically symmetric alternating sign matrices]{Alternating sign matrices with reflective symmetry and plane partitions: \boldmath $n+3$ pairs of equivalent statistics and a Cauchy-type identity} 



\author{Ilse Fischer}
\address{University of Vienna, Austria}
\urladdr{https://www.mat.univie.ac.at/~ifischer/}
\author{Hans H\"ongesberg}
\address{University of Ljubljana, Slovenia}
\urladdr{https://hoengesberg.com/}
\thanks{The authors acknowledge support from the Austrian Science Fund (FWF) grant 10.55776/P3493. Additionally, the first author acknowledges support from FWF grant 10.55776/F1002, while the second author acknowledges support from FWF grant 10.55776/J4810. For open access purposes, the authors have applied a CC BY public copyright license to any author accepted manuscript version arising from this submission.}

\begin{document}

\begin{abstract}
Vertically symmetric alternating sign matrices (VSASMs) of order $2n+1$ are known to be equinumerous with 
lozenge tilings of a hexagon with side lengths $2n+2,2n,2n+2,2n,2n+2,2n$ and a central triangular hole of size $2$ that exhibit a cyclical as well as a vertical symmetry, but finding an explicit bijection proving this belongs to the most difficult problems in bijective combinatorics. Towards constructing such a bijection, we generalize 
the result by introducing certain natural extensions for both objects along with $n+3$ parameters and show that the multivariate generating functions with respect to these parameters coincide. 
This is a significant step from a constant number of equidistributed statistics to a linear number of statistics in $n$.
The equinumeracy of VSASMs and the lozenge tilings is then an easy consequence of this result, which is obtained by specializing the generating functions to signed enumerations for both types of objects and then applying certain sign-reversing involutions.
Another main result concerns the expansion of the 
multivariate generating function into symplectic characters as a sum over totally symmetric self-complementary plane partitions, which is 
in perfect analogy to the situation for ordinary ASMs where the Schur expansion can be written as a sum over totally symmetric plane partitions.
This is exciting as it is reminiscent of the well-known Cauchy identity, and the Cauchy identity does have a bijective proof using the Robinson--Schensted--Knuth correspondence, and thus the result raises the question of whether there is a variation of the Robinson--Schensted--Knuth correspondence that does eventually lead to a bijective proof.
\end{abstract}

\maketitle

\section{Introduction}

An \emph{alternating sign matrix} (ASM) is a square matrix with entries in $\{0,\pm 1\}$ such that, in each row and column, the nonzero entries alternate and add up to $1$. An example is displayed next:
\[
\begin{pmatrix}
0 & 0 & 0 & 1 & 0 & 0 & 0 \\
0 & 1 & 0 &-1 & 0 & 1 & 0 \\
1 & -1& 0 & 1 & 0 &-1 & 1 \\
0 & 0 & 1 &-1 & 1 & 0 & 0 \\
0 & 1 & -1& 1&-1 & 1 & 0 \\
0 & 0 & 1 &-1& 1 & 0 & 0 \\
0 & 0 & 0 & 1 & 0 & 0 & 0
 \end{pmatrix} 
\]
It is well-known, see 
\cite{And79,Zei96a,Kra06}, that there is the same number of $n \times n$ ASMs as there is of \emph{cyclically symmetric lozenge tilings} of a hexagon with side lengths $n+1,n-1,n+1,n-1,n+1,n-1$ with a central triangular hole of size $2$, but no explicit bijection has been constructed so far. These lozenge tilings are in easy bijective correspondence with 
\emph{descending plane partitions} (DPPs) with parts no greater than $n$. 
In \cite{nASMDPP}, a certain refinement of this result was provided that involves $n+3$ statistics on both sides. 
Prior to this result, four statistics were known on ASMs and on DPPs that have the same joint distribution. 
Note that the discovery of the (right) quadratic number of statistics  would resolve the problem of establishing a bijection as the two families of objects are assembled of about a quadratic number  of integers.\footnote{For instance, a monotone triangle of order $n$ is comprised of $1+2+\ldots+n=\frac{(n+1)n}{2}$ integers, which we could also view as statistics (for example, the top entry of a monotone triangle is a well-studied classical statistic with an equivalent counterpart on DPPs, see \cite{MilRobRum83,BehDifZin12}). 
If we were to find a set of $\frac{(n+1)n}{2}$ statistics on the DPP-side that has the same joint distribution, then the bijection would have been constructed.}
The  refinement in \cite{nASMDPP}
made it necessary to work with extended objects, which evolved naturally through the $n+3$ pairs of statistics that were included, and it was shown that the joint distributions of the two sets of $n+3$ statistics coincide. A general introduction into this fascinating area of missing bijections is provided in \cite{Bre99} and a recent update is given in \cite{cube}. 

The purpose of the current paper also is to shed new light on the mysterious relation between ASMs and plane partitions. More specifically, it is known that, in perfect analogy to the result mentioned above, vertically symmetric $(2n+1) \times (2n+1)$ alternating sign matrices (VSASMs) are equinumerous with cyclically symmetric lozenge tilings of a hexagon with side lengths $2n+2,2n,2n+2,2n,2n+2,2n$ and with a central triangular hole of size $2$ that exhibit an additional axial symmetry, see \cite{MRR87,Kup94}, however, a conceptual explanation of this fact is still missing. Such lozenge tilings are illustrated in Figure~\ref{fig:ExampleRhombusTiling}. In the present paper, we provide  several $(n+3)$-parameter refinements of the relation between VSASMs and certain families of nonintersecting lattice paths or plane partition objects.

For one of these results (Theorem~\ref{interpret1}), we show that it extends the relation between VSASMs and the lozenge tilings mentioned above. Another of these results (Theorem~\ref{interpret4}) implies that the expansion of the generating function of VSASMs with respect to the $n+3$ parameters when specializing two of the parameters to $1$ into \emph{symplectic characters} can be written as a sum over \emph{totally symmetric self-complementary plane partitions} (Corollary~\ref{cor:expansion}). This result was conjectured independently by Florian Schreier-Aigner, and it is in perfect analogy to another very recent result on ordinary 
ASMs, which expresses the \emph{Schur function expansion} of their multivariate generating function as a sum over \emph{totally symmetric plane partitions}, see \cite{TSPP}.
Further information on the origin and history of the enumeration of VSASMs and other symmetry classes of ASMs can be found in \cite{Rob91,Rob00,Kup02}, as well as in the overview provided in \cite{BehFisKon17}.

Before we give a detailed description of our main results in the next section, we explain briefly how we obtain them. We choose to work with \emph{monotone triangles} instead of ASMs, but there is 
an easy bijective correspondence between $n \times n$ ASMs and monotone triangles with bottom row $1,2,\ldots,n$ as indicated at the beginning of the next section. In \cite{nASMDPP}, an $(n+3)$-parameter refinement of the \emph{operator formula} \cite{Fis06} 
for the number of monotone triangles with prescribed bottom row has been established. In view of the various objects that are known to be equinumerous with ASMs (and thus monotone triangles with bottom row 
$1,2,\ldots,n$), an obvious question is whether we can carry over the $n+3$ parameters to these other objects. The starting point for the study presented in this paper 
was to derive such a result for VSASMs and their symmetric counterparts related to DPPs. However, it also resulted in the discovery of the expansion of the specialized generating function into symplectic characters.

The above mentioned $(n+3)$-parameter refinement of the monotone triangle count can be expressed in terms of an antisymmetrizer formula, see \eqref{eq:antisymformula}. In a sense, this generating function can be seen as a variation 
of \emph{Schur polynomials} when Schur polynomials are thought of as the generating function of 
\emph{Gelfand-Tsetlin patterns} and the variation concerns certain \emph{decorated Gelfand--Tsetlin patterns} defined in Definition~\ref{AMT}. Up to a
factor, this generating function is in fact a generalization of the \emph{Hall--Littlewood polynomials}.
The case of VSASMs corresponds to having $0,2,\ldots,2n-2$ as bottom row of the decorated Gelfand--Tsetlin patterns. In this particular case, we can use Lemma~\ref{general} to transform the \emph{antisymmetrizer formula} into a \emph{bialternant-type formula} (in the spirit of the bialternant formula for the Schur polynomial, but more complicated), which is a quotient of a determinant and a Vandermonde-type product. The same lemma was also applicable in the case with no symmetry, however, in the symmetric case, a key observation was the necessity to multiply with the product in \eqref{Bfactor}. Comparing again with the Schur polynomial case, we transform the bialternant-type formula into a \emph{Jacobi--Trudi-type formula} in the next step. Here, Lemma~\ref{lem:det} is an important tool. There are several possibilities 
to accomplish this, which leads to our different results. The situation is unexpectedly different compared to the situation for ASMs with no symmetry imposed.
The various versions of Jacobi--Trudi-type determinants are then interpreted combinatorially using the \emph{Lindstr\"om--Gessel--Viennot lemma} (Lemma~\ref{LGV}) as families of lattice paths. Such interpretations are still signed enumerations in our cases, but we use further algebraic manipulations or combinatorial reasoning (we have two different proofs) to obtain a signless version, which we express in terms of pairs of plane partitions. Combinatorial reasoning is also used to relate one of our $(n+3)$-parameter refinements to the holey cyclically and vertically symmetric lozenge tilings that are known to be equinumerous with $(2n+1) \times (2n+1)$ VSASMs. Establishing this relation is considerably more complicated compared to the situation for ASMs with no symmetry imposed.

\section{Main results}
\label{sec:MainResults}

A \emph{monotone triangle} is a triangular array of integers of the following form 
\[
\begin{array}{ccccccccccc}
    &   &   &   &   & m_{1,1} &   &   &   &   & \\
    &   &   &   & m_{2,1} &   & m_{2,2} &   &   &   & \\
    &   &   & \dots &   & \dots &   & \dots &   &   & \\
    &   & m_{n-2,1} &   & \dots &   & \dots &   & m_{n-2,n-2} &   & \\
    & m_{n-1,1} &   & m_{n-1,2} &  &   \dots &   & \dots   &  & m_{n-1,n-1}  & \\
    m_{n,1} &   & m_{n,2} &   & m_{n,3} &   & \dots &   & \dots &   & m_{n,n}
\end{array}
\]
with weak increase along $\nearrow$- and $\searrow$-diagonals, and strict increase along rows, that is, $m_{i+1,j} \le m_{i,j} \le m_{i+1,j+1}$ and 
$m_{i,j} < m_{i,j+1}$. There is the following simple bijection between monotone triangles with bottom row $1,2,\ldots,n$ and $n \times n$ ASMs: Let $A=(a_{i,j})_{1 \le i,j \le n}$ be an $n \times n$ ASM and consider the matrix obtained by adding to each entry all the entries that are in the same column above, that is, $S=(\sum_{i'=1}^{i} a_{i',j})_{1 \le i, j \le n}$. It is not hard to see that this always results in a $\{0,1\}$-matrix with $i$ occurrences of $1$'s in the $i$-th row. The corresponding {monotone triangle} is obtained by recording row by row the columns of the $1$'s in $S$ in $n$ centered rows. To give an example, we apply this to the following matrix, which is obtained from the matrix in the introduction by rotating. (Note that the matrix in the introduction exhibits a vertical symmetry.) 

\[\resizebox{\linewidth}{!}{$%
\begin{pmatrix}
0 & 0 & 0 & 0 & 1 & 0 & 0 \\
0 & 0 & 1 &0 & -1 & 1 & 0 \\
0 & 1& -1 & 1 & 0 &0 & 0 \\
1 &-1 & 1 &-1 & 1 &-1 & 1 \\
0 & 1 & -1& 1&0 & 0 & 0 \\
0 & 0 & 1 &0&-1 & 1 & 0 \\
0 & 0 & 0 & 0 & 1 & 0 & 0
 \end{pmatrix} 
\longleftrightarrow 
\begin{pmatrix}
0 & 0 & 0 & 0 & 1 & 0 & 0 \\
0 & 0 & 1 &0 & 0 & 1 & 0 \\
0 & 1& 0 & 1 & 0 &1 & 0 \\
1 &0 & 1 &0 & 1 &0 & 1 \\
1 & 1 & 0& 1&1 & 0 & 1 \\
1 & 1 & 1 &1&0 & 1 & 1 \\
1 & 1 & 1 & 1 & 1 & 1 & 1
 \end{pmatrix} 
 \longleftrightarrow 
 \begin{array}{ccccccccccccc}
   &   &   &   &   &   & 5 &   &   &   &   &   &  \\
   &   &   &   &   & 3 &   & 6 &   &   &   &   &  \\
   &   &   &   & 2 &   & 4 &   & 6 &   &   &   &  \\
   &   &   & 1 &   & 3 &   & 5 &   & 7 &   &   &  \\
   &   & 1 &   & 2 &   & 4 &   & 5 &   & 7 &   &  \\
   & 1 &   & 2 &   & 3 &   & 4 &   & 6 &   & 7 &   \\
1  &   & 2 &   & 3 &   & 4 &   & 5 &   & 6 &   & 7 
\end{array}%
$}\]
An observation that will be useful later is the following: By rotating and considering only the top $n$ rows of the monotone triangles (which 
take care of the ``fundamental domain''), $(2n+1) \times (2n+1)$ VSASMs correspond to monotone triangles with bottom row 
$2,4,6,\ldots,2n$, or, equivalently, when subtracting $2$ from each entry, with bottom row 
$0,2,\ldots,2n-2$. In order to see this, it is crucial to observe that the central 
column of a VSASM is always $(1,-1,1,-1,\ldots,-1,1)^T$, which is not very difficult.

The following decorated versions of monotone triangles have first appeared in \cite{nASMDPP}:
\begin{definition} 
\label{AMT}
An \emph{arrowed monotone triangle} (AMT) is a monotone triangle where each entry $e$ carries a decoration from $\{\nwarrow, \nearrow, \nenwarrow \}$ such that the following two conditions are satisfied: 
\begin{itemize} 
\item If $e$ has a $\nwarrow$-neighbor and is equal to it, then $e$ must carry  $\nearrow$.
\item If $e$ has a $\nearrow$-neighbor and is equal to it, then $e$ must carry $\nwarrow$. 
\end{itemize}
In summary, an arrow indicates that the entry is different from the next entry in the specified direction.

We assign the following weight to an arrowed monotone triangle with $n$ rows 
\begin{equation*}
u^{\# \nearrow} v^{\# \nwarrow} w^{\# \nenwarrow} \prod_{i=1}^n X_i^{\text{(sum of entries in row $i$)}-\text{(sum of entries in row $i-1$)}+(\# \nearrow \text{ in row $i$}) - (\# \nwarrow \text{ in row $i$})},
\end{equation*}
where the sum of entries in row $0$ is defined to be $0$.
\end{definition}

In our examples, we write $^{\nwarrow} e, e^{\nearrow}, ^{\nwarrow}e^{\nearrow}$ if entry $e$ is decorated with 
$\nwarrow, \nearrow, \nenwarrow$, respectively.
Such an example is provided next:
\[
\begin{array}{ccccccccccccccccc}
  &   &   &   &   &   &   &   & ^{\nwarrow}4^{\nearrow} &   &   &   &   &   &   &   & \\
  &   &   &   &   &   &   & ^{\nwarrow}4 &   & ^{\nwarrow}5 &   &   &   &   &   &   & \\
  &   &   &   &   &   & ^{\nwarrow}3^{\nearrow} &   & 4^{\nearrow} &   & {6}^{\nearrow} &   &   &   &   &   & \\
  &   &   &   &   & 2^{\nearrow} &   & 3^{\nearrow} &   & ^{\nwarrow}{6} &   & ^{\nwarrow}8^{\nearrow} &   &   &   &   & \\
  &   &   &   & ^{\nwarrow}1 &   & 2^{\nearrow} &  &   ^{\nwarrow}5^{\nearrow} &   & 6^{\nearrow}   &  & ^{\nwarrow}9  &   &   &   & \\
  &   &   & ^{\nwarrow}0 &   & ^{\nwarrow}2 &   & ^{\nwarrow}4^{\nearrow} &   & ^{\nwarrow}6  &   & ^{\nwarrow}8^{\nearrow} &   & 10^{\nearrow} &   &   &
\end{array}
\]
Its weight is 
\begin{equation*}
u^7 v^8 w^6 X_1^{4} X_2^{3} X_3^{6} X_4^{7} X_5^{4} X_6^{5}.
\end{equation*}

In the case $n=2$, there are three undecorated monotone triangles with bottom row $(0,2)$, which amount to a total of $45$ such monotone triangles with decorations. We list these arrowed monotone triangles, where $^{*} e^{*}$ means that the entry~$e$ can be decorated with any of the arrows $\nwarrow, \nearrow, \nenwarrow$: \begin{equation*}
	\begin{array}{ccc}
		&   ^{*}0^{*}   & \\
		^{\nwarrow}0 &   & ^{*}2^{*} \\
	\end{array}\qquad\quad\begin{array}{ccc}
	&   ^{*}1^{*}   & \\
	^{*}0^{*} &   & ^{*}2^{*} \\
	\end{array}\qquad\quad\begin{array}{ccc}
	&   ^{*}2^{*}   & \\
	^{*}0^{*} &   & 2^{\nearrow} \\
	\end{array}
\end{equation*}
The generating function is 
\begin{multline} 
\label{case2}
v X_2 (u X_1 + v X_1^{-1} + w) (u X_2 + v X_2^{-1} + w) + X_1 X_2 (u X_1 + v X_1^{-1} + w) (u X_2 + v X_2^{-1} + w) ^2 \\ + u X_1^2  X_2 (u X_1 + v X_1^{-1} + w) (u X_2 + v X_2^{-1} + w).
\end{multline} 

The significance of arrowed monotone triangles originally stems from the following specialization: Setting $u=v=1$, $w=-1$ and $X_i=1$ for $i=1,\ldots,n$, the generating function of arrowed monotone triangles with bottom row $k_1,\ldots,k_n$ is the number of monotone triangles with bottom row $k_1,\ldots,k_n$. Indeed, if we want to decorate a given monotone triangle with elements from 
$\{\nwarrow, \nearrow, \nenwarrow \}$ in an eligible way, the arrows are actually prescribed for all entries except for those that 
are different from their $\nwarrow$- and $\nearrow$-neighbors (if they exist). In this case, we are free to choose any decoration. We say that such an entry is \emph{free}. Thus, suppose $f$ is the number of free entries 
for a given monotone triangle, then there are $3^f$ associated arrowed monotone triangles. Setting $X_i=1$ for $i=1,\ldots,n$ in the generating function, the contribution of these $3^f$ arrowed monotone triangles in the generating function is up to a monomial in $u$ and $v$ (coming from the entries that are not free) equal to 
$(u+v+w)^f$, which evaluates to $1$ if $u=v=1$ and $w=-1$. 

We now examine how the generating function of AMTs relates to three well-known statistics on VSASMs: the number of $-1$s, the inversion number and the complementary inversion number. Let $A=(a_{i,j})_{1 \le i,j \le 2n+1}$ be a $(2n+1) \times (2n+1)$ VSASM. We denote by $\mathcal{N}(A)$ the number of $-1$s in the $n$ leftmost columns of $A$ and define
\begin{equation*}
	\inv(A) \coloneq \sum_{\substack{1 \le i < k \le 2n+1\\1 \le l \le j \le n}} a_{i,j} a_{k,l} \quad\text{and}\quad \inv'(A) \coloneq \sum_{\substack{1 \le i < k \le 2n+1\\1 \le j \le l \le n}} a_{i,j} a_{k,l}
\end{equation*}
to be the \emph{inversion number} and the \emph{complementary inversion number}, respectively.
The inversion number of ordinary ASMs was first introduced in \cite{lambda} as a generalization of the inversion number of permutation matrices. We define the weight of $A$ to be
\begin{equation*}
	(u+v+w)^{\mathcal{N}(A)} u^{\inv(A)} v^{\inv'(A)}.
\end{equation*}
It follows from an argument in \cite{TSPP} that the generating function of arrowed monotone triangles with bottom row $0,2,\ldots,2n-2$ after setting $X_i=1$ for $i=1,\dots,n$ equals
\begin{equation*}
(u+v+w)^n \sum_{A} (u+v+w)^{\mathcal{N}(A)} u^{\inv(A)} v^{\inv'(A)},
\end{equation*}
where we sum over all $(2n+1) \times (2n+1)$ VSASMs.

\bigskip

Before we present our first main result, we need two definitions.
A \emph{column-strict plane partition} (\CSPP) is a filling of a Young diagram with positive integers that weakly decrease along rows and strictly decrease down columns, whereas a  \emph{row-strict plane partition} (\RSPP) is a filling of a Young diagram with positive integers that weakly decrease down columns and strictly along rows.  Next we display a \CSPP (left) and an \RSPP (right):

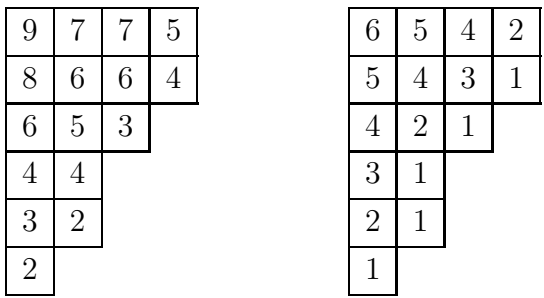
\[
\begin{ytableau}
9 & 7 & 7 & 5 \\
8 & 6 & 6 & 4 \\
6 & 5 & 3 \\
4 & 4  \\
3 & 2 \\
2
\end{ytableau} 
\hspace{2cm}
\begin{ytableau}
 6 & 5 & 4 & 2 \\
 5  & 4   & 3   &  1 \\
 4  & 2   & 1 \\
 3  & 1 \\
 2 & 1 \\
 1
\end{ytableau} 
\]

\begin{theorem} 
\label{interpret4} 
For $n \ge 1$, the generating function of arrowed monotone triangles with bottom row 
$0,2,\ldots,2n-2$ is equal to the generating function of pairs $(P,Q)$ of plane partitions of the same shape with $n$ rows (allowing also rows of length zero) 
such that $P$ is a 
\CSPP and $Q$ is an \RSPP, and
the entries of $P$ in the $i$-th row from the bottom are no greater than $2i$, while the entries of $Q$ in the $i$-th row from the bottom are no greater than $i$. The weight of such a pair is given by the following monomial:
\begin{equation*}
w^{\binom{n+1}{2}-\# \text{ of entries in $P$}} \prod_{i=1}^{n} X_i^{n-1} (u X_i)^{\# \text{ of $2i-1$ in $P$}} (v X_i^{-1})^{\# \text{ of $2i$ in $P$}}.
\end{equation*}
\end{theorem}

Letting $n=6$, the example above is a pair that has the properties described in Theorem~\ref{interpret4}. The weight is 
\begin{multline*}
w^{\binom{7}{2} - 16} X_1^5 X_2^5 X_3^5 X_4^5 X_5^5 X_6^5 (u X_1)^0 (u X_2)^2 (u X_3)^2 (u X_4)^2 (u X_5)^1 (u X_6)^0 \\
\times (v X_1^{-1})^2 (v X_2^{-1})^3 (v X_3^{-1})^3 (v X_4^{-1})^1 (v X_5^{-1})^0 
(v X_6^{-1})^0,
\end{multline*}
which is equal to $w^{5} u^{7} v^{9} X_1^3 X_2^4 X_3^4 X_4^6 X_5^6 X_6^5$. Next we list all pairs for the case $n=2$ along with their weights, which indeed add up to \eqref{case2}.

\ytableausetup{smalltableaux}

\begin{center}
	$\begin{array}{ccccccccc}
			\left(\emptyset,\emptyset\right) & \left(\begin{ytableau}1 \end{ytableau},\begin{ytableau}1 \end{ytableau}\right) &
			\left(\begin{ytableau}2 \end{ytableau},\begin{ytableau}1 \end{ytableau}\right)&
			\left(\begin{ytableau}3 \end{ytableau},\begin{ytableau}1 \end{ytableau}\right)&
			\left(\begin{ytableau}4 \end{ytableau},\begin{ytableau}1 \end{ytableau}\right)&
			\left(\begin{ytableau}1 \end{ytableau},\begin{ytableau}2 \end{ytableau}\right)& 
			\left(\begin{ytableau}2 \end{ytableau},\begin{ytableau}2 \end{ytableau}\right)&
			\left(\begin{ytableau}3 \end{ytableau},\begin{ytableau}2 \end{ytableau}\right)&
			\left(\begin{ytableau}4 \end{ytableau},\begin{ytableau}2 \end{ytableau}\right)\\
		\scriptstyle w^3 X_1 X_2&  
		\scriptstyle w^2 u X_1^{2} X_2&   
		\scriptstyle w^2 v  X_2&   
		\scriptstyle w^2 u X_1 X_2^{2}&   
		\scriptstyle w^2 v X_1&   
		\scriptstyle w^2 u X_1^{2} X_2&   
		\scriptstyle w^2 v X_2&   
		\scriptstyle w^2 u X_1 X_2^{2}& 
		\scriptstyle w^2 v X_1
		\end{array}$
\end{center}

\begin{center}
	$\begin{array}{ccccccc}
			\left(\begin{ytableau} 1& 1 \end{ytableau},\begin{ytableau}2 & 1 \end{ytableau}\right)&
			\left(\begin{ytableau} 2& 1 \end{ytableau},\begin{ytableau}2 & 1 \end{ytableau}\right)&
			\left(\begin{ytableau} 3& 1 \end{ytableau},\begin{ytableau}2 & 1 \end{ytableau}\right)&
			\left(\begin{ytableau} 4& 1 \end{ytableau},\begin{ytableau}2 & 1 \end{ytableau}\right)&
			\left(\begin{ytableau} 2& 2 \end{ytableau},\begin{ytableau}2 & 1 \end{ytableau}\right)&
			\left(\begin{ytableau} 3& 2 \end{ytableau},\begin{ytableau}2 & 1 \end{ytableau}\right)&
			\left(\begin{ytableau} 4& 2 \end{ytableau},\begin{ytableau}2 & 1 \end{ytableau}\right)\\
		\scriptstyle w u^2 X_1^{3} X_2&
		\scriptstyle w u v X_1 X_2&   
		\scriptstyle w u^2 X_1^{2} X_2^{2}&   
		\scriptstyle w u v X_1^{2}&   
		\scriptstyle w v^2 X_1^{-1} X_2&   
		\scriptstyle w u v  X_2^{2}&
		\scriptstyle w v^2
		\end{array}$
\end{center}

\begin{center}
	$\begin{array}{cccccccc}
			\left(\begin{ytableau} 3& 3 \end{ytableau},\begin{ytableau}2 & 1 \end{ytableau}\right)&
			\left(\begin{ytableau} 4& 3 \end{ytableau},\begin{ytableau}2 & 1 \end{ytableau}\right)&
			\left(\begin{ytableau} 4& 4 \end{ytableau},\begin{ytableau}2 & 1 \end{ytableau}\right)&
			\left(\begin{ytableau} 2 \\ 1 \end{ytableau},\begin{ytableau}1 \\1 \end{ytableau}\right)&
			\left(\begin{ytableau} 3 \\ 1 \end{ytableau},\begin{ytableau}1 \\1 \end{ytableau}\right)&
			\left(\begin{ytableau} 4 \\ 1 \end{ytableau},\begin{ytableau}1 \\1 \end{ytableau}\right)&
			\left(\begin{ytableau} 3 \\ 2 \end{ytableau},\begin{ytableau}1 \\1 \end{ytableau}\right)&
			\left(\begin{ytableau} 4 \\ 2 \end{ytableau},\begin{ytableau}1 \\1 \end{ytableau}\right)\\
			\scriptstyle w u^2 X_1 X_2^{3}&
			\scriptstyle w u v X_1 X_2&   
			\scriptstyle w v^2 X_1 X_2^{-1}& 
			\scriptstyle w u v X_1 X_2&   
			\scriptstyle w u^2 X_1^{2} X_2^{2}&   
			\scriptstyle w u v X_1^{2}&
			\scriptstyle w u v X_2^{2}&   
			\scriptstyle w v^2
		\end{array}$
\end{center}

\begin{center}
	$\begin{array}{cccccccc}
			\left(\begin{ytableau} 2 \\ 1 \end{ytableau},\begin{ytableau}2 \\1 \end{ytableau}\right)&
			\left(\begin{ytableau} 3 \\ 1 \end{ytableau},\begin{ytableau}2 \\1 \end{ytableau}\right)&
			\left(\begin{ytableau} 4 \\ 1 \end{ytableau},\begin{ytableau}2 \\1 \end{ytableau}\right)&
			\left(\begin{ytableau} 3 \\ 2 \end{ytableau},\begin{ytableau}2 \\1 \end{ytableau}\right)&
			\left(\begin{ytableau} 4 \\ 2 \end{ytableau},\begin{ytableau}2 \\1 \end{ytableau}\right)& 
			\left(\begin{ytableau} 2&1 \\ 1 \end{ytableau},\begin{ytableau}2&1 \\1 \end{ytableau}\right)&
			\left(\begin{ytableau} 2&2 \\ 1 \end{ytableau},\begin{ytableau}2&1 \\1 \end{ytableau}\right)&
			\left(\begin{ytableau} 3&1 \\ 1 \end{ytableau},\begin{ytableau}2&1 \\1 \end{ytableau}\right)\\
			\scriptstyle w u v X_1 X_2&
			\scriptstyle w u^2 X_1^{2} X_2^{2}&
			\scriptstyle w u v X_1^{2}&   
			\scriptstyle w u v X_2^{2}&  
			\scriptstyle w v^2&
			\scriptstyle u^2 v X_1^{2} X_2&
			\scriptstyle u v^2  X_2&
			\scriptstyle u^3  X_1^{3} X_2^{2}
		\end{array}$
\end{center}

\begin{center}
	$\begin{array}{ccccccc}
			\left(\begin{ytableau} 3&2 \\ 1 \end{ytableau},\begin{ytableau}2&1 \\1 \end{ytableau}\right)&
			\left(\begin{ytableau} 3&3 \\ 1 \end{ytableau},\begin{ytableau}2&1 \\1 \end{ytableau}\right)&
			\left(\begin{ytableau} 3&1 \\ 2 \end{ytableau},\begin{ytableau}2&1 \\1 \end{ytableau}\right)&
			\left(\begin{ytableau} 3&2 \\ 2 \end{ytableau},\begin{ytableau}2&1 \\1 \end{ytableau}\right)&
			\left(\begin{ytableau} 3&3 \\ 2 \end{ytableau},\begin{ytableau}2&1 \\1 \end{ytableau}\right)&
			\left(\begin{ytableau} 4&1 \\ 1 \end{ytableau},\begin{ytableau}2&1 \\1 \end{ytableau}\right)&
			\left(\begin{ytableau} 4&2 \\ 1 \end{ytableau},\begin{ytableau}2&1 \\1 \end{ytableau}\right)\\
			\scriptstyle u^2 v X_1 X_2^{2}&
			\scriptstyle u^3  X_1^{2} X_2^{3}&
			\scriptstyle u^2 v X_1 X_2^{2}&
			\scriptstyle u v^2 X_1^{-1} X_2^{2}&
			\scriptstyle u^2 v X_2^{3}& 
			\scriptstyle u^2 v X_1^{3}&
			\scriptstyle u v^2 X_1
		\end{array}$
\end{center}

\begin{center}
	 $\begin{array}{cccccc}
	 		\left(\begin{ytableau} 4&3 \\ 1 \end{ytableau},\begin{ytableau}2&1 \\1 \end{ytableau}\right)&
	 		\left(\begin{ytableau} 4&4 \\ 1 \end{ytableau},\begin{ytableau}2&1 \\1 \end{ytableau}\right)&
			\left(\begin{ytableau} 4&1 \\ 2 \end{ytableau},\begin{ytableau}2&1 \\1 \end{ytableau}\right)&
			\left(\begin{ytableau} 4&2 \\ 2 \end{ytableau},\begin{ytableau}2&1 \\1 \end{ytableau}\right)&
			\left(\begin{ytableau} 4&3 \\ 2 \end{ytableau},\begin{ytableau}2&1 \\1 \end{ytableau}\right)&
			\left(\begin{ytableau} 4&4 \\ 2 \end{ytableau},\begin{ytableau}2&1 \\1 \end{ytableau}\right)\\
			\scriptstyle u^2 v X_1^{2} X_2&
			\scriptstyle u v^2 X_1^{2} X_2^{-1}&
			\scriptstyle u v^2 X_1&
			\scriptstyle v^3 X_1^{-1}& 
			\scriptstyle u v^2  X_2&
			\scriptstyle v^3 X_2^{-1}
		\end{array}$
\end{center}

\ytableausetup{nosmalltableaux}

It is an open problem to find a bijective proof of this theorem. In fact, there is some similarity with the \emph{Cauchy identity}, which may give a hint that there is a \emph{Robinson--Schensted--Knuth-type correspondence} behind this theorem. In the case of the classical Cauchy identity, one side of the identity has a combinatorial interpretation 
in terms of the generating function of pairs of semistandard tableaux of the same shape, which clearly corresponds  to the pairs of plane partitions 
$(P,Q)$ in Theorem~\ref{interpret4} in this analogy. The other side of the Cauchy identity is interpreted as the generating function of lexicographically ordered two-line arrays, and they generalize 
permutations. In Theorem~\ref{interpret4}, this ``other'' side is the generating function of symmetric ASMs, and ASMs themselves are also generalizations 
of permutations. Moreover, in view of this relation with the Cauchy identity, another open problem is the following: in Theorem~\ref{interpret4}, only the plane partition $P$ carries a multivariate weight. It is an open problem to find a generalization of Theorem~\ref{interpret4} that involves a multivariate 
weight also on $Q$, possibly on a second set of variables, say, $Y_1,\ldots,Y_n$. 

Theorem~\ref{interpret4} is analogous to \cite[Theorem 1.1]{TSPP}, where the latter is about ASMs with no symmetry imposed. Also for that theorem, 
the corresponding open problems just mentioned are unsolved. Next we phrase Theorem~\ref{interpret4} in a way that makes the relation to \cite[Theorem 1.1]{TSPP} more 
apparent. For this purpose, we show that
\RSPPs with at most $n$ rows (allowing also rows of length zero) such that the entries in the $i$-th row from the bottom are no greater than $i$
are in easy bijective correspondence with $(2n+2) \times (2n+2) \times (2n+2)$ {totally symmetric self-complementary plane partitions}: 
After replacing each part $p$ of the \RSPP
by  $n+1-p$ and conjugating, we obtain a semistandard Young tableau with entries in $\{1,2,\ldots,n\}$ 
such that the entries in column $i$ (counted from the left) are no smaller than $i$. Phrased differently, entries $i$ can only occur in the first $i$ columns. For the example preceding Theorem~\ref{interpret4}, we obtain 
\[
\begin{ytableau}
 1 & 2 & 3 & 4 & 5 & 6 \\
 2  & 3  & 5   &  6 & 6  \\
 3  & 4   & 6   \\
 5  & 6  
 \end{ytableau} 
\]
when choosing $n=6$.

Now we use the standard procedure to transform semistandard Young tableaux with entries in $\{1,2,\ldots,n\}$ into Gelfand--Tsetlin patterns with $n$ rows: 
To obtain the $i$-th row of the Gelfand--Tsetlin pattern, consider the shape in the semistandard Young tableaux constituted by the entries less than or 
equal to $i$, add $0$'s to the corresponding partition so that it is of length $i$ and write it in reverse order. In our running example, we obtain the following Gelfand--Tsetlin pattern:
\[
\begin{array}{ccccccccccccccccc}
  &   &   &   &   &   &   &   & 1 &   &   &   &   &   &   &   & \\
  &   &   &   &   &   &   & 1 &   & 2 &   &   &   &   &   &   & \\
  &   &   &   &   &   &1 &   & 2 &   & 3 &   &   &   &   &   & \\
  &   &   &   &   & 0 &   & 2 &   & 2 &   & 4 &   &   &   &   & \\
  &   &   &   & 0 &   & 1 &  &   2 &   & 3   &  & 5  &   &   &   & \\
  &   &   & 0 &   & 0 &   & 2 &   & 3 &   & 5 &   & 6 &   &   &
\end{array}
\]
As the entries less than or equal to $i$ are confined to the first $i$ columns, it follows that all parts in the partition that constitute the $i$-th row of the Gelfand--Tsetlin pattern are less than or equal to $i$. We obtain a Gelfand--Tsetlin pattern with $n$ rows, 
parts in $\{0,1,\ldots,n\}$ such that the entries in $i$-th $\nearrow$-diagonal are no greater $i$. Adding $1$ to each entry and prepending a 
$\nearrow$-diagonal of $1$'s on the left results in a \emph{Magog triangle} of order $n+1$ as defined in \cite{Zei96a,FischerGog}. In our example, we obtain 
\[
\begin{array}{ccccccccccccccccc}
  &   &   &   &   &   &   & 1  &  &   &   &   &   &   &   &   & \\
  &   &   &   &   &   & 1  &   & 2 &   &   &   &   &   &   &   & \\
  &   &   &   &   & 1  &   & 2 &   & 3 &   &   &   &   &   &   & \\
  &   &   &   & 1   &   &2 &   & 3 &   & 4 &   &   &   &   &   & \\
  &   &   & 1  &   & 1 &   & 3 &   & 3 &   & 5 &   &   &   &   & \\
  &   &1   &   & 1 &   & 2 &  &   3 &   & 4   &  & 6  &   &   &   & \\
  &1   &   & 1 &   & 1 &   & 3 &   & 4 &   & 6 &   & 7 &   &   &
\end{array}.
\]
Such Magog triangles are known to be in 
bijective correspondence with totally symmetric self-complementary plane partitions \cite{MilRobRum86} inside a $(2n+2) \times (2n+2) \times (2n+2)$ box. For a given totally symmetric self-complementary plane partitions~$T$, let $\lambda(T)$ be the shape of the corresponding \RSPP under the aforementioned bijection.

On the other hand, the \CSPPs with the constraints as they appear in the theorem are in easy bijective correspondence with {symplectic tableaux} as defined in \cite[Section 4]{KoiTer90}. The generating functions of symplectic tableaux of shape~$\lambda$ are known as the \emph{symplectic Schur polynomials $\sp_\lambda$}, which admit the following bialternant formula
	\begin{equation*}
		\sp_{\lambda} (X_1,\dots,X_n,X_1^{-1},\dots,X_n^{-1}) = \frac{\det_{1 \le i,j \le n} \left(X_i^{\lambda_j + n - j + 1} - (X_i^{-1})^{\lambda_j + n - j + 1} \right)}{\det_{1 \le i,j \le n} \left(X_i^{n - j + 1} - (X_i^{-1})^{n - j + 1} \right)}.
	\end{equation*}

Thus, Theorem~\ref{interpret4} 
can also be seen as to provide the expansion of the generating function of arrowed monotone triangles with bottom row $0,2,\ldots,2n-2$ into symplectic characters when setting $u=v=1$. This was conjectured by Florian Schreier-Aigner \cite{FloPrivate} independently.

\begin{cor}\label{cor:expansion}
	The generating function of AMTs with bottom row $(0,2,\dots,2n-2)$ is
	\[
	\prod_{i=1}^{n} X_i^{n-1} \sum_{T} w^{\binom{n+1}{2}-|\lambda(T)|} \sp_{\lambda(T)}(u X_1,\ldots, u X_n, v X_1^{-1},\ldots,v X_n^{-1}),
	\]
	where we sum over totally symmetric self-complementary plane partitions~$T$  inside a $(2n+2) \times (2n+2) \times (2n+2)$ box.
\end{cor}

\begin{figure}[htb]
	\centering
	\begin{tikzpicture}[scale=.45,baseline=(current bounding box.center)]
	
	\foreach \x in {-8,...,6}
	\foreach \y in {0,...,14}
		\coordinate (\x/\y) at ($ (60:\y) + (\x,0) $);
		
	\coordinate (Z) at (3,6.35);
	
	
	\draw (0/0) -- (6/0) -- (6/8) -- (0/14) -- (-8/14) -- (-8/8) -- cycle;
	
	\draw[fill=black] (0/6) -- (0/8) -- (-2/8) -- cycle;
	
	
	\draw[fill=gray] (3/0) -- (3/1) -- (2/2) -- (2/1) -- cycle;
	\draw[fill=gray] (2/2) -- (2/3) -- (1/4) -- (1/3) -- cycle;
	\draw[fill=gray] (1/4) -- (1/5) -- (0/6) -- (0/5) -- cycle;
	\draw[fill=gray] (-1/8) -- (-1/9) -- (-2/10) -- (-2/9) -- cycle;
	\draw[fill=gray] (-2/10) -- (-2/11) -- (-3/12) -- (-3/11) -- cycle;
	\draw[fill=gray] (-3/12) -- (-3/13) -- (-4/14) -- (-4/13) -- cycle;
	
	\draw[fill=gray] (-4/4) -- (-3/4) -- (-3/5) -- (-4/5) -- cycle;
	\draw[fill=gray] (-3/5) -- (-2/5) -- (-2/6) -- (-3/6) -- cycle;
	\draw[fill=gray] (-2/6) -- (-1/6) -- (-1/7) -- (-2/7) -- cycle;
	\draw[fill=gray] (0/8) -- (1/8) -- (1/9) -- (0/9) -- cycle;
	\draw[fill=gray] (1/9) -- (2/9) -- (2/10) -- (1/10) -- cycle;
	\draw[fill=gray] (2/10) -- (3/10) -- (3/11) -- (2/11) -- cycle;
	
	\draw[fill=gray] (-8/11) -- (-7/11) -- (-6/10) -- (-7/10) -- cycle;
	\draw[fill=gray] (-6/10) -- (-5/10) -- (-4/9) -- (-5/9) -- cycle;
	\draw[fill=gray] (-4/9) -- (-3/9) -- (-2/8) -- (-3/8) -- cycle;
	\draw[fill=gray] (0/7) -- (1/7) -- (2/6) -- (1/6) -- cycle;
	\draw[fill=gray] (2/6) -- (3/6) -- (4/5) -- (3/5) -- cycle;
	\draw[fill=gray] (4/5) -- (5/5) -- (6/4) -- (5/4) -- cycle;
	
	
	\draw[fill=gray] (2/0) -- (3/0) -- (2/1) -- (1/1) -- cycle;
	\draw[fill=gray] (1/0) -- (2/0) -- (1/1) -- (0/1) -- cycle;
	\draw[fill=gray] (0/0) -- (1/0) -- (0/1) -- (-1/1) -- cycle;
	\draw[fill=gray] (-1/1) -- (0/1) -- (-1/2) -- (-2/2) -- cycle;
	\draw[fill=gray] (-2/2) -- (-1/2) -- (-2/3) -- (-3/3) -- cycle;
	\draw[fill=gray] (-3/3) -- (-2/3) -- (-3/4) -- (-4/4) -- cycle;
	
	\draw[fill=gray] (-4/4) -- (-4/5) -- (-5/6) -- (-5/5) -- cycle;
	\draw[fill=gray] (-5/5) -- (-5/6) -- (-6/7) -- (-6/6) -- cycle;
	\draw[fill=gray] (-6/6) -- (-6/7) -- (-7/8) -- (-7/7) -- cycle;
	\draw[fill=gray] (-7/7) -- (-7/8) -- (-8/9) -- (-8/8) -- cycle;
	\draw[fill=gray] (-7/8) -- (-7/9) -- (-8/10) -- (-8/9) -- cycle;
	\draw[fill=gray] (-7/9) -- (-7/10) -- (-8/11) -- (-8/10) -- cycle;
	
	\draw[fill=gray] (-8/11) -- (-7/11) -- (-7/12) -- (-8/12) -- cycle;
	\draw[fill=gray] (-8/12) -- (-7/12) -- (-7/13) -- (-8/13) -- cycle;
	\draw[fill=gray] (-8/13) -- (-7/13) -- (-7/14) -- (-8/14) -- cycle;
	\draw[fill=gray] (-7/13) -- (-6/13) -- (-6/14) -- (-7/14) -- cycle;
	\draw[fill=gray] (-6/13) -- (-5/13) -- (-5/14) -- (-6/14) -- cycle;
	\draw[fill=gray] (-5/13) -- (-4/13) -- (-4/14) -- (-5/14) -- cycle;
	
	\draw[fill=gray] (-3/13) -- (-2/13) -- (-3/14) -- (-4/14) -- cycle;
	\draw[fill=gray] (-2/13) -- (-1/13) -- (-2/14) -- (-3/14) -- cycle;
	\draw[fill=gray] (-1/13) -- (0/13) -- (-1/14) -- (-2/14) -- cycle;
	\draw[fill=gray] (0/13) -- (1/13) -- (0/14) -- (-1/14) -- cycle;
	\draw[fill=gray] (1/12) -- (2/12) -- (1/13) -- (0/13) -- cycle;
	\draw[fill=gray] (2/11) -- (3/11) -- (2/12) -- (1/12) -- cycle;
	
	\draw[fill=gray] (4/9) -- (4/10) -- (3/11) -- (3/10) -- cycle;
	\draw[fill=gray] (5/8) -- (5/9) -- (4/10) -- (4/9) -- cycle;
	\draw[fill=gray] (6/7) -- (6/8) -- (5/9) -- (5/8) -- cycle;
	\draw[fill=gray] (6/6) -- (6/7) -- (5/8) -- (5/7) -- cycle;
	\draw[fill=gray] (6/5) -- (6/6) -- (5/7) -- (5/6) -- cycle;
	\draw[fill=gray] (6/4) -- (6/5) -- (5/6) -- (5/5) -- cycle;
	
	\draw[fill=gray] (5/3) -- (6/3) -- (6/4) -- (5/4) -- cycle;
	\draw[fill=gray] (5/2) -- (6/2) -- (6/3) -- (5/3) -- cycle;
	\draw[fill=gray] (5/1) -- (6/1) -- (6/2) -- (5/2) -- cycle;
	\draw[fill=gray] (5/0) -- (6/0) -- (6/1) -- (5/1) -- cycle;
	\draw[fill=gray] (4/0) -- (5/0) -- (5/1) -- (4/1) -- cycle;
	\draw[fill=gray] (3/0) -- (4/0) -- (4/1) -- (3/1) -- cycle;
	
	
	
	\draw ($ (60:8) + (-1,0) $) -- ($ (60:8) + (0,0) $);
	\draw ($ (60:9) + (-1,0) $) -- ($ (60:9) + (1,0) $);
	\draw ($ (60:10) + (-2,0) $) -- ($ (60:10) + (-1,0) $);
	\draw ($ (60:10) + (0,0) $) -- ($ (60:10) + (2,0) $);
	\draw ($ (60:11) + (-2,0) $) -- ($ (60:11) + (1,0) $);
	\draw ($ (60:12) + (-2,0) $) -- ($ (60:12) + (1,0) $);
	\draw ($ (60:13) + (-3,0) $) -- ($ (60:13) + (0,0) $);
	
	
	\draw ($ (60:10) + (-2,0) $) -- ($ (60:9) + (-1,0) $);
	\draw ($ (60:12) + (-3,0) $) -- ($ (60:11) + (-2,0) $);
	\draw ($ (60:10) + (-1,0) $) -- ($ (60:9) + (0,0) $);
	\draw ($ (60:13) + (-3,0) $) -- ($ (60:12) + (-2,0) $);
	\draw ($ (60:11) + (-1,0) $) -- ($ (60:10) + (0,0) $);
	\draw ($ (60:13) + (-2,0) $) -- ($ (60:12) + (-1,0) $);
	\draw ($ (60:11) + (0,0) $) -- ($ (60:10) + (1,0) $);
	\draw ($ (60:13) + (-1,0) $) -- ($ (60:12) + (0,0) $);
	\draw ($ (60:11) + (1,0) $) -- ($ (60:10) + (2,0) $);
	\draw ($ (60:13) + (0,0) $) -- ($ (60:11) + (2,0) $);
	
	
	\draw ($ (60:12) + (-3,0) $) -- ($ (60:13) + (-3,0) $);
	\draw ($ (60:10) + (-2,0) $) -- ($ (60:12) + (-2,0) $);
	\draw ($ (60:8) + (-1,0) $) -- ($ (60:9) + (-1,0) $);
	\draw ($ (60:10) + (-1,0) $) -- ($ (60:12) + (-1,0) $);
	\draw ($ (60:8) + (0,0) $) -- ($ (60:10) + (0,0) $);
	\draw ($ (60:11) + (0,0) $) -- ($ (60:12) + (0,0) $);
	\draw ($ (60:9) + (1,0) $) -- ($ (60:10) + (1,0) $);
	\draw ($ (60:11) + (1,0) $) -- ($ (60:12) + (1,0) $);
	\draw ($ (60:10) + (2,0) $) -- ($ (60:11) + (2,0) $);

	\begin{scope}[rotate around={120:(Z.center)}]
		
	
	\draw ($ (60:8) + (-1,0) $) -- ($ (60:8) + (0,0) $);
	\draw ($ (60:9) + (-1,0) $) -- ($ (60:9) + (1,0) $);
	\draw ($ (60:10) + (-2,0) $) -- ($ (60:10) + (-1,0) $);
	\draw ($ (60:10) + (0,0) $) -- ($ (60:10) + (2,0) $);
	\draw ($ (60:11) + (-2,0) $) -- ($ (60:11) + (1,0) $);
	\draw ($ (60:12) + (-2,0) $) -- ($ (60:12) + (1,0) $);
	\draw ($ (60:13) + (-3,0) $) -- ($ (60:13) + (0,0) $);
	
	
	\draw ($ (60:10) + (-2,0) $) -- ($ (60:9) + (-1,0) $);
	\draw ($ (60:12) + (-3,0) $) -- ($ (60:11) + (-2,0) $);
	\draw ($ (60:10) + (-1,0) $) -- ($ (60:9) + (0,0) $);
	\draw ($ (60:13) + (-3,0) $) -- ($ (60:12) + (-2,0) $);
	\draw ($ (60:11) + (-1,0) $) -- ($ (60:10) + (0,0) $);
	\draw ($ (60:13) + (-2,0) $) -- ($ (60:12) + (-1,0) $);
	\draw ($ (60:11) + (0,0) $) -- ($ (60:10) + (1,0) $);
	\draw ($ (60:13) + (-1,0) $) -- ($ (60:12) + (0,0) $);
	\draw ($ (60:11) + (1,0) $) -- ($ (60:10) + (2,0) $);
	\draw ($ (60:13) + (0,0) $) -- ($ (60:11) + (2,0) $);
	
	
	\draw ($ (60:12) + (-3,0) $) -- ($ (60:13) + (-3,0) $);
	\draw ($ (60:10) + (-2,0) $) -- ($ (60:12) + (-2,0) $);
	\draw ($ (60:8) + (-1,0) $) -- ($ (60:9) + (-1,0) $);
	\draw ($ (60:10) + (-1,0) $) -- ($ (60:12) + (-1,0) $);
	\draw ($ (60:8) + (0,0) $) -- ($ (60:10) + (0,0) $);
	\draw ($ (60:11) + (0,0) $) -- ($ (60:12) + (0,0) $);
	\draw ($ (60:9) + (1,0) $) -- ($ (60:10) + (1,0) $);
	\draw ($ (60:11) + (1,0) $) -- ($ (60:12) + (1,0) $);
	\draw ($ (60:10) + (2,0) $) -- ($ (60:11) + (2,0) $);
		
	\end{scope}

	\begin{scope}[rotate around={-120:(Z.center)}]
	
	
	\draw ($ (60:8) + (-1,0) $) -- ($ (60:8) + (0,0) $);
	\draw ($ (60:9) + (-1,0) $) -- ($ (60:9) + (1,0) $);
	\draw ($ (60:10) + (-2,0) $) -- ($ (60:10) + (-1,0) $);
	\draw ($ (60:10) + (0,0) $) -- ($ (60:10) + (2,0) $);
	\draw ($ (60:11) + (-2,0) $) -- ($ (60:11) + (1,0) $);
	\draw ($ (60:12) + (-2,0) $) -- ($ (60:12) + (1,0) $);
	\draw ($ (60:13) + (-3,0) $) -- ($ (60:13) + (0,0) $);
	
	
	\draw ($ (60:10) + (-2,0) $) -- ($ (60:9) + (-1,0) $);
	\draw ($ (60:12) + (-3,0) $) -- ($ (60:11) + (-2,0) $);
	\draw ($ (60:10) + (-1,0) $) -- ($ (60:9) + (0,0) $);
	\draw ($ (60:13) + (-3,0) $) -- ($ (60:12) + (-2,0) $);
	\draw ($ (60:11) + (-1,0) $) -- ($ (60:10) + (0,0) $);
	\draw ($ (60:13) + (-2,0) $) -- ($ (60:12) + (-1,0) $);
	\draw ($ (60:11) + (0,0) $) -- ($ (60:10) + (1,0) $);
	\draw ($ (60:13) + (-1,0) $) -- ($ (60:12) + (0,0) $);
	\draw ($ (60:11) + (1,0) $) -- ($ (60:10) + (2,0) $);
	\draw ($ (60:13) + (0,0) $) -- ($ (60:11) + (2,0) $);
	
	
	\draw ($ (60:12) + (-3,0) $) -- ($ (60:13) + (-3,0) $);
	\draw ($ (60:10) + (-2,0) $) -- ($ (60:12) + (-2,0) $);
	\draw ($ (60:8) + (-1,0) $) -- ($ (60:9) + (-1,0) $);
	\draw ($ (60:10) + (-1,0) $) -- ($ (60:12) + (-1,0) $);
	\draw ($ (60:8) + (0,0) $) -- ($ (60:10) + (0,0) $);
	\draw ($ (60:11) + (0,0) $) -- ($ (60:12) + (0,0) $);
	\draw ($ (60:9) + (1,0) $) -- ($ (60:10) + (1,0) $);
	\draw ($ (60:11) + (1,0) $) -- ($ (60:12) + (1,0) $);
	\draw ($ (60:10) + (2,0) $) -- ($ (60:11) + (2,0) $);
	
\end{scope}

	\begin{scope}[xscale=-1,shift={(-6,0)}]
	
	
	\draw ($ (60:8) + (-1,0) $) -- ($ (60:8) + (0,0) $);
	\draw ($ (60:9) + (-1,0) $) -- ($ (60:9) + (1,0) $);
	\draw ($ (60:10) + (-2,0) $) -- ($ (60:10) + (-1,0) $);
	\draw ($ (60:10) + (0,0) $) -- ($ (60:10) + (2,0) $);
	\draw ($ (60:11) + (-2,0) $) -- ($ (60:11) + (1,0) $);
	\draw ($ (60:12) + (-2,0) $) -- ($ (60:12) + (1,0) $);
	\draw ($ (60:13) + (-3,0) $) -- ($ (60:13) + (0,0) $);
	
	
	\draw ($ (60:10) + (-2,0) $) -- ($ (60:9) + (-1,0) $);
	\draw ($ (60:12) + (-3,0) $) -- ($ (60:11) + (-2,0) $);
	\draw ($ (60:10) + (-1,0) $) -- ($ (60:9) + (0,0) $);
	\draw ($ (60:13) + (-3,0) $) -- ($ (60:12) + (-2,0) $);
	\draw ($ (60:11) + (-1,0) $) -- ($ (60:10) + (0,0) $);
	\draw ($ (60:13) + (-2,0) $) -- ($ (60:12) + (-1,0) $);
	\draw ($ (60:11) + (0,0) $) -- ($ (60:10) + (1,0) $);
	\draw ($ (60:13) + (-1,0) $) -- ($ (60:12) + (0,0) $);
	\draw ($ (60:11) + (1,0) $) -- ($ (60:10) + (2,0) $);
	\draw ($ (60:13) + (0,0) $) -- ($ (60:11) + (2,0) $);
	
	
	\draw ($ (60:12) + (-3,0) $) -- ($ (60:13) + (-3,0) $);
	\draw ($ (60:10) + (-2,0) $) -- ($ (60:12) + (-2,0) $);
	\draw ($ (60:8) + (-1,0) $) -- ($ (60:9) + (-1,0) $);
	\draw ($ (60:10) + (-1,0) $) -- ($ (60:12) + (-1,0) $);
	\draw ($ (60:8) + (0,0) $) -- ($ (60:10) + (0,0) $);
	\draw ($ (60:11) + (0,0) $) -- ($ (60:12) + (0,0) $);
	\draw ($ (60:9) + (1,0) $) -- ($ (60:10) + (1,0) $);
	\draw ($ (60:11) + (1,0) $) -- ($ (60:12) + (1,0) $);
	\draw ($ (60:10) + (2,0) $) -- ($ (60:11) + (2,0) $);
	
\end{scope}

	\begin{scope}[xscale=-1,shift={(-6,0)},rotate around={120:(Z.center)}]
	
	
	\draw ($ (60:8) + (-1,0) $) -- ($ (60:8) + (0,0) $);
	\draw ($ (60:9) + (-1,0) $) -- ($ (60:9) + (1,0) $);
	\draw ($ (60:10) + (-2,0) $) -- ($ (60:10) + (-1,0) $);
	\draw ($ (60:10) + (0,0) $) -- ($ (60:10) + (2,0) $);
	\draw ($ (60:11) + (-2,0) $) -- ($ (60:11) + (1,0) $);
	\draw ($ (60:12) + (-2,0) $) -- ($ (60:12) + (1,0) $);
	\draw ($ (60:13) + (-3,0) $) -- ($ (60:13) + (0,0) $);
	
	
	\draw ($ (60:10) + (-2,0) $) -- ($ (60:9) + (-1,0) $);
	\draw ($ (60:12) + (-3,0) $) -- ($ (60:11) + (-2,0) $);
	\draw ($ (60:10) + (-1,0) $) -- ($ (60:9) + (0,0) $);
	\draw ($ (60:13) + (-3,0) $) -- ($ (60:12) + (-2,0) $);
	\draw ($ (60:11) + (-1,0) $) -- ($ (60:10) + (0,0) $);
	\draw ($ (60:13) + (-2,0) $) -- ($ (60:12) + (-1,0) $);
	\draw ($ (60:11) + (0,0) $) -- ($ (60:10) + (1,0) $);
	\draw ($ (60:13) + (-1,0) $) -- ($ (60:12) + (0,0) $);
	\draw ($ (60:11) + (1,0) $) -- ($ (60:10) + (2,0) $);
	\draw ($ (60:13) + (0,0) $) -- ($ (60:11) + (2,0) $);
	
	
	\draw ($ (60:12) + (-3,0) $) -- ($ (60:13) + (-3,0) $);
	\draw ($ (60:10) + (-2,0) $) -- ($ (60:12) + (-2,0) $);
	\draw ($ (60:8) + (-1,0) $) -- ($ (60:9) + (-1,0) $);
	\draw ($ (60:10) + (-1,0) $) -- ($ (60:12) + (-1,0) $);
	\draw ($ (60:8) + (0,0) $) -- ($ (60:10) + (0,0) $);
	\draw ($ (60:11) + (0,0) $) -- ($ (60:12) + (0,0) $);
	\draw ($ (60:9) + (1,0) $) -- ($ (60:10) + (1,0) $);
	\draw ($ (60:11) + (1,0) $) -- ($ (60:12) + (1,0) $);
	\draw ($ (60:10) + (2,0) $) -- ($ (60:11) + (2,0) $);
	
\end{scope}

	\begin{scope}[xscale=-1,shift={(-6,0)},rotate around={-120:(Z.center)}]
	
	
	\draw ($ (60:8) + (-1,0) $) -- ($ (60:8) + (0,0) $);
	\draw ($ (60:9) + (-1,0) $) -- ($ (60:9) + (1,0) $);
	\draw ($ (60:10) + (-2,0) $) -- ($ (60:10) + (-1,0) $);
	\draw ($ (60:10) + (0,0) $) -- ($ (60:10) + (2,0) $);
	\draw ($ (60:11) + (-2,0) $) -- ($ (60:11) + (1,0) $);
	\draw ($ (60:12) + (-2,0) $) -- ($ (60:12) + (1,0) $);
	\draw ($ (60:13) + (-3,0) $) -- ($ (60:13) + (0,0) $);
	
	
	\draw ($ (60:10) + (-2,0) $) -- ($ (60:9) + (-1,0) $);
	\draw ($ (60:12) + (-3,0) $) -- ($ (60:11) + (-2,0) $);
	\draw ($ (60:10) + (-1,0) $) -- ($ (60:9) + (0,0) $);
	\draw ($ (60:13) + (-3,0) $) -- ($ (60:12) + (-2,0) $);
	\draw ($ (60:11) + (-1,0) $) -- ($ (60:10) + (0,0) $);
	\draw ($ (60:13) + (-2,0) $) -- ($ (60:12) + (-1,0) $);
	\draw ($ (60:11) + (0,0) $) -- ($ (60:10) + (1,0) $);
	\draw ($ (60:13) + (-1,0) $) -- ($ (60:12) + (0,0) $);
	\draw ($ (60:11) + (1,0) $) -- ($ (60:10) + (2,0) $);
	\draw ($ (60:13) + (0,0) $) -- ($ (60:11) + (2,0) $);
	
	
	\draw ($ (60:12) + (-3,0) $) -- ($ (60:13) + (-3,0) $);
	\draw ($ (60:10) + (-2,0) $) -- ($ (60:12) + (-2,0) $);
	\draw ($ (60:8) + (-1,0) $) -- ($ (60:9) + (-1,0) $);
	\draw ($ (60:10) + (-1,0) $) -- ($ (60:12) + (-1,0) $);
	\draw ($ (60:8) + (0,0) $) -- ($ (60:10) + (0,0) $);
	\draw ($ (60:11) + (0,0) $) -- ($ (60:12) + (0,0) $);
	\draw ($ (60:9) + (1,0) $) -- ($ (60:10) + (1,0) $);
	\draw ($ (60:11) + (1,0) $) -- ($ (60:12) + (1,0) $);
	\draw ($ (60:10) + (2,0) $) -- ($ (60:11) + (2,0) $);
	
\end{scope}

\begin{scope}[bullet/.style={circle, fill, minimum size=4pt,               inner sep=0pt, outer sep=0pt},nodes=bullet]

	\node[bullet] at ($ (60:8.5) + (-1,0) $) {};
	\node[bullet] at ($ (60:8.5) + (0,0) $) {};
	
	\node[bullet] at ($ (60:10.5) + (-2,0) $) {};
	\node[bullet] at ($ (60:9.5) + (1,0) $) {};
	
	\node[bullet] at ($ (60:12.5) + (-3,0) $) {};
	\node[bullet] at ($ (60:10.5) + (2,0) $) {};
	
	\draw[ultra thick] ($ (60:8.5) + (-1,0) $) -- ($ (60:8.5) + (0,0) $);
	
	\draw[ultra thick] ($ (60:10.5) + (-2,0) $) -- ($ (60:10.5) + (-1,0) $) -- ($ (60:9.5) + (0,0) $) -- ($ (60:9.5) + (1,0) $);
	
	\draw[ultra thick] ($ (60:12.5) + (-3,0) $) -- ($ (60:11.5) + (-2,0) $) -- ($ (60:11.5) + (1,0) $) -- ($ (60:10.5) + (2,0) $);
	
\end{scope}
	
	\end{tikzpicture}
	$\longleftrightarrow$
	\begin{tikzpicture}[scale=.5,baseline=(current bounding box.center)]
	\draw [help lines,step=1cm,dashed] (-.75,-.75) grid (5.75,4.75);

	\fill (1,0) circle (5pt);
	\fill (3,1) circle (5pt);
	\fill (5,2) circle (5pt);
	
	\fill (0,0) circle (5pt);
	\fill (1,2) circle (5pt);
	\fill (2,4) circle (5pt);

	\node[below right] at (1,0) {\contour{white}{$A_1$}};
	\node[below right] at (3,1) {\contour{white}{$A_2$}};
	\node[below right] at (5,2) {\contour{white}{$A_3$}};
	
	\node[above left] at (0,0) {\contour{white}{$E_1$}};
	\node[above left] at (1,2) {\contour{white}{$E_2$}};
	\node[above left] at (2,4) {\contour{white}{$E_3$}};
	
	\draw[ultra thick] (0,0) -- (1,0);
	\draw[ultra thick] (1,2) -- (2,2) -- (2,1) -- (3,1);
	\draw[ultra thick] (2,4) -- (2,3) -- (5,3) -- (5,2);
	\end{tikzpicture}
	\caption{Cyclically and vertically symmetric lozenge tiling of a hexagon with a central triangular hole and the corresponding family of nonintersecting lattice paths. The gray tilings are forced due to the symmetry.}
	\label{fig:ExampleRhombusTiling}
\end{figure}

\bigskip

Our other main result concerns nonintersecting lattice paths objects that have the same generating function as arrowed monotone triangles with bottom row $0,2,\ldots,2n-2$ that can be related combinatorially to the lozenge tilings when specializing. Before we state the results, recall that  cyclically and vertically symmetric lozenge tilings of 
a hexagon with side lengths $2n+2,2n,2n+2,2n,2n+2,2n$ and a central triangular hole of size $2$ are in easy bijective correspondence with families of $n$ nonintersecting lattice paths. An example of such a lozenge tiling is provided in Figure~\ref{fig:ExampleRhombusTiling} (left). Due to the symmetries of the tilings, the holey hexagon can be decomposed into six equal parts, up to rotation and reflection. Figure~\ref{fig:ExampleRhombusTiling} (right) shows the family of nonintersecting lattice paths that corresponds to the fundamental area. At the end of Section~\ref{sec:first}, we provide more details and show how these families of nonintersecting lattice paths relate to the following theorem.

The theorem involves lattice paths in $\mathbb{Z}^2$, see Figure~\ref{fig:ExampleFirstInterpretation} for an example. We divide the lattice points of $\mathbb{Z}^2$ into \emph{even} and \emph{odd} points depending on whether the sum of the coordinates is even or odd, respectively.

\begin{figure}[htb]
	\centering
	\begin{tikzpicture}[scale=.45,baseline=(current bounding box.center)]
		\fill [verylight-gray] (-6.75,-4.75) rectangle (0,10.75);
		\fill [light-gray] (0,-4.75) rectangle (6.75,1);
		
		\draw [help lines,step=1cm,dashed] (-6.75,-4.75) grid (6.75,10.75);
		
		\draw[->,thick] (-6.75,0)--(6.75,0) node[right]{$x$};
		\draw[->,thick] (0,-4.75)--(0,10.75) node[above]{$y$};
		
		\fill (-1,1) circle (5pt);
		\fill (-2,2) circle (5pt);
		\fill (-3,3) circle (5pt);
		\fill (-4,4) circle (5pt);
		\fill (-5,5) circle (5pt);
		\fill (-6,6) circle (5pt);
		
		\fill (0,1) circle (5pt);
		\fill (1,0) circle (5pt);
		\fill (2,-1) circle (5pt);
		\fill (3,-2) circle (5pt);
		\fill (4,-3) circle (5pt);
		\fill (5,-4) circle (5pt);
		
		\path[decoration=arrows, decorate] (-1,1) --++ (1,0);
		
		\path[decoration=arrows, decorate] (-2,2) --++ (1,0) --++ (1,0) --++ (1,-1) --++ (0,-1);
		
		\draw[ultra thick,dashed,myred] (0,2) --++ (1,-1);
		
		\path[decoration=arrows, decorate] (-3,3) --++ (1,1) --++ (1,0) --++ (1,1) --++ (0,-2) --++ (1,-1) --++ (1,-1) --++ (0,-1) --++ (0,-1);
		
		\draw[ultra thick,dotted,mygreen] (0,5) --++ (0,-2) --++ (1,-1) --++ (1,-1);
		
		\path[decoration=arrows, decorate] (-4,4) --++ (1,1) --++ (1,1) --++ (1,0) --++ (1,1) --++ (1,-1) --++ (1,-1) --++ (1,-1) --++ (1,-1) --++ (1,-1) --++ (1,-1) --++ (0,-1) --++ (0,-1) --++ (0,-1) --++ (0,-1) --++ (0,-1) --++ (-1,0);
		
		\draw[ultra thick,dotted,mygreen] (0,7) --++ (1,-1) --++ (1,-1) --++ (1,-1) --++ (1,-1) --++ (1,-1) --++ (1,-1);
		
		\path[decoration=arrows, decorate] (-5,5) --++ (1,0) --++ (1,1) --++ (1,1) --++ (1,1) --++ (1,0) --++ (1,-1) --++ (0,-2) --++ (0,-2) --++ (1,-1) --++ (1,-1) --++ (0,-1) --++ (0,-1) --++ (0,-1);
		
		\draw[ultra thick,dashed,myred] (0,8) --++ (1,-1) --++ (0,-2) --++ (0,-2) --++ (1,-1) --++ (1,-1);
		
		\path[decoration=arrows, decorate] (-6,6) --++ (1,1) --++ (1,1) --++ (1,1) --++ (1,0) --++ (1,1) --++ (1,0) --++ (1,-1) --++ (1,-1) --++ (1,-1) --++ (0,-2) --++ (1,-1) --++ (0,-2) --++ (1,-1) --++ (0,-1) --++ (0,-1) --++ (-1,0) --++ (0,-1) --++ (0,-1);
		
		\draw[ultra thick,dashed,myred] (0,10) --++ (1,-1) --++ (1,-1) --++ (1,-1) --++ (0,-2) --++ (1,-1) --++ (0,-2) --++ (1,-1);
		
	\end{tikzpicture}
	\caption{Example of the lattice paths in Theorem~\ref{interpret1} for $n=6$. The associated permutation is $\sigma = (1\;2\;3\;6\;4\;5)$ and the weight is $(-u v)^5 w^9 \allowbreak X_1^5 X_2^5 X_3^5 X_4^5 X_5^5 X_6^5(u X_3+v X_3^{-1}) (u X_6 + v X_6^{-1})$. In the second region, we draw the even and odd paths in different colors.}
	\label{fig:ExampleFirstInterpretation}
\end{figure}

\begin{theorem} 
\label{interpret1} 
For $n \ge 1$, the generating function of arrowed monotone triangles with bottom row 
$0,2,\ldots,2n-2$ is equal to the signed generating function of $n$ lattice paths with starting points 
$(-1,1),(-2,2),\ldots,(-n,n)$ and end points $(0,1),(1,0),\ldots,(n-1,-n+2)$ and the following properties:
\begin{itemize} 
\item In the region $\{(x,y) \mid x \le 0\}$, the step set is $\{(1,1),(1,0)\}$, and steps of type $(1,0)$ are equipped with the weight $w$. 
\item In the region $\{(x,y) \mid  x \ge 0, y \ge 1\}$, the step set is $\{(1,-1),(0,-2)\}$, and steps of type $(0,-2)$  are equipped with the weight $-uv$. 
\item 
In the region $\{(x,y) \mid  x \ge 0, y \le 1\}$, the step set is $\{(-1,0),(0,-1)\}$. Horizontal steps with distance $d$ from the line $y=2$ are equipped with the weight $u X_d+v X_d^{-1}$. 
\end{itemize} 
The paths are nonintersecting in the first and in the third region. In the second region, we have two types of paths, 
even and odd, depending on whether they contain only even lattice points or only odd lattice points, respectively. Lattice paths of the same type are not intersecting each other, but an odd path may have an intersection with an even path. 

The weight of a family of lattice paths is $\prod_{i=1}^{n} X_i^{n-1}$ multiplied by the product of the weights of all its steps where the weight of a step is $1$ if it has not been specified.
Let $\sigma$ be the permutation so that the $i$-th starting point is connected to the $\sigma(i)$-th end point, then the sign of the family is $\sgn \sigma$. 
\end{theorem}

Next we list all families of lattice paths from Theorem~\ref{interpret1} for the case $n=2$ along with their weights up to the overall factor of $X_1 X_2$.
        
        \begin{center} 
        \begin{tikzpicture}[scale=.45,baseline=(current bounding box.center)]
		\fill [verylight-gray] (-2.75,-.75) rectangle (0,3.75);
		\fill [light-gray] (0,-.75) rectangle (1.75,1);
		
		\draw [help lines,step=1cm,dashed] (-2.75,-.75) grid (1.75,3.75);
		
		\draw[->,thick] (-2.75,0)--(1.75,0) node[right]{\tiny $x$};
		\draw[->,thick] (0,-.75)--(0,3.75) node[above]{\tiny$y$};
		
		\fill (-1,1) circle (5pt);
		\fill (-2,2) circle (5pt);
				
		\fill (0,1) circle (5pt);
		\fill (1,0) circle (5pt);

		\path[decoration=arrows, decorate] (-1,1) --++ (1,0);
		
		\path[decoration=arrows, decorate] (-2,2) --++ (1,0) --++ (1,0) --++ (1,-1) --++ (0,-1);
		
		\draw[ultra thick,dashed,myred] (0,2) --++ (1,-1);
		
		\draw (0,-.75) node[below]{\tiny $w^3$};
		
	\end{tikzpicture}
        \begin{tikzpicture}[scale=.45,baseline=(current bounding box.center)]
		\fill [verylight-gray] (-2.75,-.75) rectangle (0,3.75);
		\fill [light-gray] (0,-.75) rectangle (2.75,1);
		
		\draw [help lines,step=1cm,dashed] (-2.75,-.75) grid (2.75,3.75);
		
		\draw[->,thick] (-2.75,0)--(2.75,0) node[right]{\tiny$x$};
		\draw[->,thick] (0,-.75)--(0,3.75) node[above]{\tiny$y$};
		
		\fill (-1,1) circle (5pt);
		\fill (-2,2) circle (5pt);
				
		\fill (0,1) circle (5pt);
		\fill (1,0) circle (5pt);

		\path[decoration=arrows, decorate] (-1,1) --++ (1,0);
		
		\path[decoration=arrows, decorate] (-2,2) --++ (1,0) --++ (1,1) --++ (1,-1) --++  (1,-1) --++ (-1,0) --++ (0,-1);
		
		\draw[ultra thick,dotted,mygreen] (0,3) --++ (1,-1) --++ (1,-1);
		
		\draw (0,-.75) node[below]{\tiny$w^2(u X_1+v X_1^{-1})$};
		
	\end{tikzpicture}
	 \begin{tikzpicture}[scale=.45,baseline=(current bounding box.center)]
		\fill [verylight-gray] (-2.75,-.75) rectangle (0,3.75);
		\fill [light-gray] (0,-.75) rectangle (2.75,1);
		
		\draw [help lines,step=1cm,dashed] (-2.75,-.75) grid (2.75,3.75);
		
		\draw[->,thick] (-2.75,0)--(2.75,0) node[right]{\tiny$x$};
		\draw[->,thick] (0,-.75)--(0,3.75) node[above]{\tiny$y$};
		
		\fill (-1,1) circle (5pt);
		\fill (-2,2) circle (5pt);
				
		\fill (0,1) circle (5pt);
		\fill (1,0) circle (5pt);

		\path[decoration=arrows, decorate] (-1,1) --++ (1,0);
		
		\path[decoration=arrows, decorate] (-2,2) --++ (1,0) --++ (1,1) --++ (1,-1) --++  (1,-1) --++ (0,-1) --++ (-1,0);
		
		\draw[ultra thick,dotted,mygreen] (0,3) --++ (1,-1) --++ (1,-1);
		
		\draw (0,-.75) node[below]{\tiny$w^2(u X_2+v X_2^{-1})$};
		
	\end{tikzpicture}
	 \begin{tikzpicture}[scale=.45,baseline=(current bounding box.center)]
		\fill [verylight-gray] (-2.75,-.75) rectangle (0,3.75);
		\fill [light-gray] (0,-.75) rectangle (2.75,1);
		
		\draw [help lines,step=1cm,dashed] (-2.75,-.75) grid (2.75,3.75);
		
		\draw[->,thick] (-2.75,0)--(2.75,0) node[right]{\tiny$x$};
		\draw[->,thick] (0,-.75)--(0,3.75) node[above]{\tiny$y$};
		
		\fill (-1,1) circle (5pt);
		\fill (-2,2) circle (5pt);
				
		\fill (0,1) circle (5pt);
		\fill (1,0) circle (5pt);

		\path[decoration=arrows, decorate] (-1,1) --++ (1,0);
		
		\path[decoration=arrows, decorate] (-2,2) --++ (1,1) --++ (1,0) --++ (1,-1) --++  (1,-1) --++ (-1,0) --++ (0,-1);
		
		\draw[ultra thick,dotted,mygreen] (0,3) --++ (1,-1) --++ (1,-1);
		
		\draw (0,-.75) node[below]{\tiny$w^2(u X_1+v X_1^{-1})$};
		
	\end{tikzpicture}
	 \begin{tikzpicture}[scale=.45,baseline=(current bounding box.center)]
		\fill [verylight-gray] (-2.75,-.75) rectangle (0,3.75);
		\fill [light-gray] (0,-.75) rectangle (2.75,1);
		
		\draw [help lines,step=1cm,dashed] (-2.75,-.75) grid (2.75,3.75);
		
		\draw[->,thick] (-2.75,0)--(2.75,0) node[right]{\tiny$x$};
		\draw[->,thick] (0,-.75)--(0,3.75) node[above]{\tiny$y$};
		
		\fill (-1,1) circle (5pt);
		\fill (-2,2) circle (5pt);
				
		\fill (0,1) circle (5pt);
		\fill (1,0) circle (5pt);

		\path[decoration=arrows, decorate] (-1,1) --++ (1,0);
		
		\path[decoration=arrows, decorate] (-2,2) --++ (1,1) --++ (1,0) --++ (1,-1) --++  (1,-1) --++ (0,-1) --++ (-1,0);
		
		\draw[ultra thick,dotted,mygreen] (0,3) --++ (1,-1) --++ (1,-1);
		
		\draw (0,-.75) node[below]{\tiny$w^2(u X_2+v X_2^{-1})$};
		
	\end{tikzpicture}
	\end{center}
	
	 \begin{center} 
	 \begin{tikzpicture}[scale=.45,baseline=(current bounding box.center)]
		\fill [verylight-gray] (-2.75,-.75) rectangle (0,4.75);
		\fill [light-gray] (0,-.75) rectangle (1.75,1);
		
		\draw [help lines,step=1cm,dashed] (-2.75,-.75) grid (1.75,4.75);
		
		\draw[->,thick] (-2.75,0)--(1.75,0) node[right]{\tiny$x$};
		\draw[->,thick] (0,-.75)--(0,4.75) node[above]{\tiny$y$};
		
		\fill (-1,1) circle (5pt);
		\fill (-2,2) circle (5pt);
				
		\fill (0,1) circle (5pt);
		\fill (1,0) circle (5pt);

		\path[decoration=arrows, decorate] (-1,1) --++ (1,0);
		
		\path[decoration=arrows, decorate] (-2,2) --++ (1,1) --++ (1,1) --++ (0,-2) --++  (1,-1) --++ (0,-1);
		
		\draw[ultra thick,dashed,myred] (0,4) --++ (0,-2) --++ (1,-1);
		
		\draw (0,-.75) node[below]{\tiny$w (- uv)$};
		
	\end{tikzpicture}
	\begin{tikzpicture}[scale=.45,baseline=(current bounding box.center)]
		\fill [verylight-gray] (-2.75,-.75) rectangle (0,4.75);
		\fill [light-gray] (0,-.75) rectangle (1.75,1);
		
		\draw [help lines,step=1cm,dashed] (-2.75,-.75) grid (1.75,4.75);
		
		\draw[->,thick] (-2.75,0)--(1.75,0) node[right]{\tiny$x$};
		\draw[->,thick] (0,-.75)--(0,4.75) node[above]{\tiny$y$};
		
		\fill (-1,1) circle (5pt);
		\fill (-2,2) circle (5pt);
				
		\fill (0,1) circle (5pt);
		\fill (1,0) circle (5pt);

		\path[decoration=arrows, decorate] (-1,1) --++ (1,0);
		
		\path[decoration=arrows, decorate] (-2,2) --++ (1,1) --++ (1,1) --++ (1,-1) --++  (0,-2) --++ (0,-1);
		
		\draw[ultra thick,dashed,myred] (0,4) --++ (1,-1) --++ (0,-2);
		
		\draw (0,-.75) node[below]{\tiny$w (- uv)$};
		
	\end{tikzpicture}
	\begin{tikzpicture}[scale=.45,baseline=(current bounding box.center)]
		\fill [verylight-gray] (-2.75,-.75) rectangle (0,4.75);
		\fill [light-gray] (0,-.75) rectangle (3.75,1);
		
		\draw [help lines,step=1cm,dashed] (-2.75,-.75) grid (3.75,4.75);
		
		\draw[->,thick] (-2.75,0)--(3.75,0) node[right]{\tiny$x$};
		\draw[->,thick] (0,-.75)--(0,4.75) node[above]{\tiny$y$};
		
		\fill (-1,1) circle (5pt);
		\fill (-2,2) circle (5pt);
				
		\fill (0,1) circle (5pt);
		\fill (1,0) circle (5pt);

		\path[decoration=arrows, decorate] (-1,1) --++ (1,0);
		
		\path[decoration=arrows, decorate] (-2,2) --++ (1,1) --++ (1,1) --++ (1,-1) --++  (1,-1) --++  (1,-1) --++ (-1,0) --++ (-1,0) --++ (0,-1);
		
		\draw[ultra thick,dashed,myred] (0,4) --++ (1,-1) --++ (1,-1) --++ (1,-1);
		
		\draw (0.5,-.75) node[below]{\tiny$w (u X_1 + v X_1^{-1})^2$};
		
	\end{tikzpicture}
	\begin{tikzpicture}[scale=.45,baseline=(current bounding box.center)]
		\fill [verylight-gray] (-2.75,-.75) rectangle (0,4.75);
		\fill [light-gray] (0,-.75) rectangle (3.75,1);
		
		\draw [help lines,step=1cm,dashed] (-2.75,-.75) grid (3.75,4.75);
		
		\draw[->,thick] (-2.75,0)--(3.75,0) node[right]{\tiny$x$};
		\draw[->,thick] (0,-.75)--(0,4.75) node[above]{\tiny$y$};
		
		\fill (-1,1) circle (5pt);
		\fill (-2,2) circle (5pt);
				
		\fill (0,1) circle (5pt);
		\fill (1,0) circle (5pt);

		\path[decoration=arrows, decorate] (-1,1) --++ (1,0);
		
		\path[decoration=arrows, decorate] (-2,2) --++ (1,1) --++ (1,1) --++ (1,-1) --++  (1,-1) --++  (1,-1) --++ (-1,0) --++ (0,-1) --++ (-1,0);
		
		\draw[ultra thick,dashed,myred] (0,4) --++ (1,-1) --++ (1,-1) --++ (1,-1);
		
		\draw (0.5,-.75) node[below]{\tiny$w (u X_1 + v X_1^{-1})(u X_2 + v X_2^{-1})$};
		
	\end{tikzpicture}
	\begin{tikzpicture}[scale=.45,baseline=(current bounding box.center)]
		\fill [verylight-gray] (-2.75,-.75) rectangle (0,4.75);
		\fill [light-gray] (0,-.75) rectangle (3.75,1);
		
		\draw [help lines,step=1cm,dashed] (-2.75,-.75) grid (3.75,4.75);
		
		\draw[->,thick] (-2.75,0)--(3.75,0) node[right]{\tiny$x$};
		\draw[->,thick] (0,-.75)--(0,4.75) node[above]{\tiny$y$};
		
		\fill (-1,1) circle (5pt);
		\fill (-2,2) circle (5pt);
				
		\fill (0,1) circle (5pt);
		\fill (1,0) circle (5pt);

		\path[decoration=arrows, decorate] (-1,1) --++ (1,0);
		
		\path[decoration=arrows, decorate] (-2,2) --++ (1,1) --++ (1,1) --++ (1,-1) --++  (1,-1) --++  (1,-1) --++ (0,-1) --++ (-1,0) --++ (-1,0);
		
		\draw[ultra thick,dashed,myred] (0,4) --++ (1,-1) --++ (1,-1) --++ (1,-1);
		
		\draw (0.5,-.75) node[below]{\tiny$w(u X_2 + v X_2^{-1})^2$};
		
	\end{tikzpicture}
	\end{center} 
	
	\begin{center}
	\begin{tikzpicture}[scale=.45,baseline=(current bounding box.center)]
		\fill [verylight-gray] (-2.75,-.75) rectangle (0,4.75);
		\fill [light-gray] (0,-.75) rectangle (2.75,1);
		
		\draw [help lines,step=1cm,dashed] (-2.75,-.75) grid (2.75,4.75);
		
		\draw[->,thick] (-2.75,0)--(2.75,0) node[right]{\tiny$x$};
		\draw[->,thick] (0,-.75)--(0,4.75) node[above]{\tiny$y$};
		
		\fill (-1,1) circle (5pt);
		\fill (-2,2) circle (5pt);
				
		\fill (0,1) circle (5pt);
		\fill (1,0) circle (5pt);

		\path[decoration=arrows, decorate] (-1,1) --++ (1,1) --++ (1,-1) --++ (-1,0);
		
		\path[decoration=arrows, decorate] (-2,2) --++ (1,0) --++ (1,1) --++ (1,-1) --++ (1,-1) --++ (0,-1) --++ (-1,0);
		
		\draw[ultra thick,dotted,mygreen] (0,2) --++ (1,-1);
		
		\draw[ultra thick,dashed,myred] (0,3) --++ (1,-1) --++ (1,-1);
		
		\draw (0,-.75) node[below]{\tiny$w(u X_1 + v X_1^{-1})(u X_2 + v X_2^{-1})$};
		
	\end{tikzpicture}
	\begin{tikzpicture}[scale=.45,baseline=(current bounding box.center)]
		\fill [verylight-gray] (-2.75,-.75) rectangle (0,4.75);
		\fill [light-gray] (0,-.75) rectangle (2.75,1);
		
		\draw [help lines,step=1cm,dashed] (-2.75,-.75) grid (2.75,4.75);
		
		\draw[->,thick] (-2.75,0)--(2.75,0) node[right]{\tiny$x$};
		\draw[->,thick] (0,-.75)--(0,4.75) node[above]{\tiny$y$};
		
		\fill (-1,1) circle (5pt);
		\fill (-2,2) circle (5pt);
				
		\fill (0,1) circle (5pt);
		\fill (1,0) circle (5pt);

		\path[decoration=arrows, decorate] (-1,1) --++ (1,1) --++ (1,-1) --++ (-1,0);
		
		\path[decoration=arrows, decorate] (-2,2) --++ (1,1) --++ (1,0) --++ (1,-1) --++ (1,-1) --++ (0,-1) --++ (-1,0);
		
		\draw[ultra thick,dotted,mygreen] (0,2) --++ (1,-1);
		
		\draw[ultra thick,dashed,myred] (0,3) --++ (1,-1) --++ (1,-1);
		
		\draw (0,-.75) node[below]{\tiny$w(u X_1 + v X_1^{-1})(u X_2 + v X_2^{-1})$};
		
	\end{tikzpicture}
	\begin{tikzpicture}[scale=.45,baseline=(current bounding box.center)]
		\fill [verylight-gray] (-2.75,-.75) rectangle (0,4.75);
		\fill [light-gray] (0,-.75) rectangle (3.75,1);
		
		\draw [help lines,step=1cm,dashed] (-2.75,-.75) grid (3.75,4.75);
		
		\draw[->,thick] (-2.75,0)--(3.75,0) node[right]{\tiny$x$};
		\draw[->,thick] (0,-.75)--(0,4.75) node[above]{\tiny$y$};
		
		\fill (-1,1) circle (5pt);
		\fill (-2,2) circle (5pt);
				
		\fill (0,1) circle (5pt);
		\fill (1,0) circle (5pt);

		\path[decoration=arrows, decorate] (-1,1) --++ (1,1) --++ (1,-1) --++ (-1,0);
		
		\path[decoration=arrows, decorate] (-2,2) --++ (1,1) --++ (1,1) --++ (1,-1) --++ (1,-1) --++ (1,-1) --++ (-1,0) --++ (0,-1) --++ (-1,0);
		
		\draw[ultra thick,dotted,mygreen] (0,2) --++ (1,-1);
		
		\draw[ultra thick,dotted,mygreen] (0,4) --++ (1,-1) --++ (1,-1) --++ (1,-1);
		
		\draw (0.5,-.75) node[below]{\tiny$(u X_1 + v X_1^{-1})^2(u X_2 + v X_2^{-1})$};
		
	\end{tikzpicture}
	\end{center}

	\begin{center}
	\begin{tikzpicture}[scale=.45,baseline=(current bounding box.center)]
		\fill [verylight-gray] (-2.75,-.75) rectangle (0,4.75);
		\fill [light-gray] (0,-.75) rectangle (3.75,1);
		
		\draw [help lines,step=1cm,dashed] (-2.75,-.75) grid (3.75,4.75);
		
		\draw[->,thick] (-2.75,0)--(3.75,0) node[right]{\tiny$x$};
		\draw[->,thick] (0,-.75)--(0,4.75) node[above]{\tiny$y$};
		
		\fill (-1,1) circle (5pt);
		\fill (-2,2) circle (5pt);
		
		\fill (0,1) circle (5pt);
		\fill (1,0) circle (5pt);

		\path[decoration=arrows, decorate] (-1,1) --++ (1,1) --++ (1,-1) --++ (-1,0);
		
		\path[decoration=arrows, decorate] (-2,2) --++ (1,1) --++ (1,1) --++ (1,-1) --++ (1,-1) --++ (1,-1) --++ (0,-1) --++ (-1,0) --++ (-1,0);
		
		\draw[ultra thick,dotted,mygreen] (0,2) --++ (1,-1);
		
		\draw[ultra thick,dotted,mygreen] (0,4) --++ (1,-1) --++ (1,-1) --++ (1,-1);
		
		\draw (0.5,-.75) node[below]{\tiny$(u X_1 + v X_1^{-1})(u X_2 + v X_2^{-1})^2$};
		
	\end{tikzpicture}
	\begin{tikzpicture}[scale=.45,baseline=(current bounding box.center)]
		\fill [verylight-gray] (-2.75,-.75) rectangle (0,3.75);
		\fill [light-gray] (0,-.75) rectangle (1.75,1);
		
		\draw [help lines,step=1cm,dashed] (-2.75,-.75) grid (1.75,3.75);
		
		\draw[->,thick] (-2.75,0)--(1.75,0) node[right]{\tiny $x$};
		\draw[->,thick] (0,-.75)--(0,3.75) node[above]{\tiny$y$};
		
		\fill (-1,1) circle (5pt);
		\fill (-2,2) circle (5pt);
				
		\fill (0,1) circle (5pt);
		\fill (1,0) circle (5pt);

		\path[decoration=arrows, decorate] (-1,1) --++ (1,1) --++ (1,-1) --++ (0,-1);
		
		\path[decoration=arrows, decorate] (-2,2) --++ (1,1) --++ (1,0) --++ (0,-2);
		
		\draw[ultra thick,dashed,myred] (0,2) --++ (1,-1);
		
		\draw[ultra thick,dotted,mygreen] (0,3) --++ (0,-2);

		\draw (0,-.75) node[below]{\tiny $- w (-uv)$};
		
	\end{tikzpicture}
	\begin{tikzpicture}[scale=.45,baseline=(current bounding box.center)]
		\fill [verylight-gray] (-2.75,-.75) rectangle (0,3.75);
		\fill [light-gray] (0,-.75) rectangle (1.75,1);
		
		\draw [help lines,step=1cm,dashed] (-2.75,-.75) grid (1.75,3.75);
		
		\draw[->,thick] (-2.75,0)--(1.75,0) node[right]{\tiny $x$};
		\draw[->,thick] (0,-.75)--(0,3.75) node[above]{\tiny$y$};
		
		\fill (-1,1) circle (5pt);
		\fill (-2,2) circle (5pt);
				
		\fill (0,1) circle (5pt);
		\fill (1,0) circle (5pt);

		\path[decoration=arrows, decorate] (-1,1) --++ (1,1) --++ (1,-1) --++ (0,-1);
		
		\path[decoration=arrows, decorate] (-2,2) --++ (1,0) --++ (1,1) --++ (0,-2);
		
		\draw[ultra thick,dashed,myred] (0,2) --++ (1,-1);
		
		\draw[ultra thick,dotted,mygreen] (0,3) --++ (0,-2);

		\draw (0,-.75) node[below]{\tiny $- w (-uv)$};
		
	\end{tikzpicture}
        \end{center}
It can be checked that the weights indeed add up to the generating function in \eqref{case2}.

At the end of Section~\ref{sec:first}, we show that, when setting $u=v=1$, $w=-1$ and $X_i=1$ for $i=1,\ldots,n$, the number of families of lattice paths as described in Theorem~\ref{interpret1} are equinumerous with the families of nonintersecting lattice paths that correspond to cyclically and vertically symmetric lozenge tilings of a hexagon with alternating side lengths $2n+2$ and $2n$ and a central triangular hole of size $2$. We provide a purely combinatorial proof of this result in Section~\ref{ap:bijective_proofs}.

\begin{remark}
	Building on the work of Lalonde \cite{Lal03}, Krattenthaler \cite{Kra06} showed that descending plane partitions 
	with entries less than or equal to $2n+1$ can be realized as cyclically symmetric lozenge tilings of a hexagon with alternating side length~$2n$ and $2n+2$ and central triangular hole with side length~$2$.
	Let us recall the definition of a descending plane partition. A \emph{descending plane partition} is the filling of a shifted Young diagram with positive integers such that
	\begin{itemize}
		\item entries are weakly decreasing along rows and strictly decreasing along columns and
		\item the first entry of each row does not exceed the length of the preceding row (if it exists) but is at least the length of its own row. 
	\end{itemize}
	Figure~\ref{fig:DPP} (right) depicts an example of a descending plane partition.
	
	To demonstrate the bijective correspondence between descending plane partitions and lozenge tilings, we take the lozenge tiling from Figure~\ref{fig:ExampleRhombusTiling}, rotate it by $30^\circ$ and draw lines as illustrated in Figure~\ref{fig:DPP}. We then read off the heights of the tiles~\horizontaltile, omitting tiles of height~$0$.	
	
	\begin{figure}[htb]
		\centering
		\begin{tikzpicture}[scale=.45,baseline=(current bounding box.center),rotate=30
			]
			
			\foreach \x in {-8,...,6}
			\foreach \y in {0,...,14}
			\coordinate (\x/\y) at ($ (60:\y) + (\x,0) $);
			
			\coordinate (Z) at (3,6.35);
			
			
			\draw (0/0) -- (6/0) -- (6/8) -- (0/14) -- (-8/14) -- (-8/8) -- cycle;
			
			\draw[fill=black] (0/6) -- (0/8) -- (-2/8) -- cycle;
			
			
			\draw (3/0) -- (3/1) -- (2/2) -- (2/1) -- cycle;
			\draw (2/2) -- (2/3) -- (1/4) -- (1/3) -- cycle;
			\draw (1/4) -- (1/5) -- (0/6) -- (0/5) -- cycle;
			\draw (-1/8) -- (-1/9) -- (-2/10) -- (-2/9) -- cycle;
			\draw (-2/10) -- (-2/11) -- (-3/12) -- (-3/11) -- cycle;
			\draw (-3/12) -- (-3/13) -- (-4/14) -- (-4/13) -- cycle;
			
			\draw (-4/4) -- (-3/4) -- (-3/5) -- (-4/5) -- cycle;
			\draw (-3/5) -- (-2/5) -- (-2/6) -- (-3/6) -- cycle;
			\draw (-2/6) -- (-1/6) -- (-1/7) -- (-2/7) -- cycle;
			\draw (0/8) -- (1/8) -- (1/9) -- (0/9) -- cycle;
			\draw (1/9) -- (2/9) -- (2/10) -- (1/10) -- cycle;
			\draw (2/10) -- (3/10) -- (3/11) -- (2/11) -- cycle;
			
			\draw (-8/11) -- (-7/11) -- (-6/10) -- (-7/10) -- cycle;
			\draw (-6/10) -- (-5/10) -- (-4/9) -- (-5/9) -- cycle;
			\draw (-4/9) -- (-3/9) -- (-2/8) -- (-3/8) -- cycle;
			\draw (0/7) -- (1/7) -- (2/6) -- (1/6) -- cycle;
			\draw (2/6) -- (3/6) -- (4/5) -- (3/5) -- cycle;
			\draw (4/5) -- (5/5) -- (6/4) -- (5/4) -- cycle;
			
			
			\draw (2/0) -- (3/0) -- (2/1) -- (1/1) -- cycle;
			\draw (1/0) -- (2/0) -- (1/1) -- (0/1) -- cycle;
			\draw (0/0) -- (1/0) -- (0/1) -- (-1/1) -- cycle;
			\draw (-1/1) -- (0/1) -- (-1/2) -- (-2/2) -- cycle;
			\draw (-2/2) -- (-1/2) -- (-2/3) -- (-3/3) -- cycle;
			\draw (-3/3) -- (-2/3) -- (-3/4) -- (-4/4) -- cycle;
			
			\draw (-4/4) -- (-4/5) -- (-5/6) -- (-5/5) -- cycle;
			\draw (-5/5) -- (-5/6) -- (-6/7) -- (-6/6) -- cycle;
			\draw (-6/6) -- (-6/7) -- (-7/8) -- (-7/7) -- cycle;
			\draw (-7/7) -- (-7/8) -- (-8/9) -- (-8/8) -- cycle;
			\draw (-7/9) -- (-7/10) -- (-8/11) -- (-8/10) -- cycle;
			
			\draw (-8/11) -- (-7/11) -- (-7/12) -- (-8/12) -- cycle;
			\draw (-8/12) -- (-7/12) -- (-7/13) -- (-8/13) -- cycle;
			\draw (-8/13) -- (-7/13) -- (-7/14) -- (-8/14) -- cycle;
			\draw (-7/13) -- (-6/13) -- (-6/14) -- (-7/14) -- cycle;
			\draw (-6/13) -- (-5/13) -- (-5/14) -- (-6/14) -- cycle;
			\draw (-5/13) -- (-4/13) -- (-4/14) -- (-5/14) -- cycle;
			
			\draw (-3/13) -- (-2/13) -- (-3/14) -- (-4/14) -- cycle;
			\draw (-2/13) -- (-1/13) -- (-2/14) -- (-3/14) -- cycle;
			\draw (-1/13) -- (0/13) -- (-1/14) -- (-2/14) -- cycle;
			\draw (0/13) -- (1/13) -- (0/14) -- (-1/14) -- cycle;
			\draw (1/12) -- (2/12) -- (1/13) -- (0/13) -- cycle;
			\draw (2/11) -- (3/11) -- (2/12) -- (1/12) -- cycle;
			
			\draw (4/9) -- (4/10) -- (3/11) -- (3/10) -- cycle;
			\draw (5/8) -- (5/9) -- (4/10) -- (4/9) -- cycle;
			\draw (6/7) -- (6/8) -- (5/9) -- (5/8) -- cycle;
			\draw (6/6) -- (6/7) -- (5/8) -- (5/7) -- cycle;
			\draw (6/5) -- (6/6) -- (5/7) -- (5/6) -- cycle;
			\draw (6/4) -- (6/5) -- (5/6) -- (5/5) -- cycle;
			
			\draw (5/3) -- (6/3) -- (6/4) -- (5/4) -- cycle;
			\draw (5/2) -- (6/2) -- (6/3) -- (5/3) -- cycle;
			\draw (5/1) -- (6/1) -- (6/2) -- (5/2) -- cycle;
			\draw (4/0) -- (5/0) -- (5/1) -- (4/1) -- cycle;
			\draw (3/0) -- (4/0) -- (4/1) -- (3/1) -- cycle;
			
			
			
			\draw ($ (60:8) + (-1,0) $) -- ($ (60:8) + (0,0) $);
			\draw ($ (60:9) + (-1,0) $) -- ($ (60:9) + (1,0) $);
			\draw ($ (60:10) + (-2,0) $) -- ($ (60:10) + (-1,0) $);
			\draw ($ (60:10) + (0,0) $) -- ($ (60:10) + (2,0) $);
			\draw ($ (60:11) + (-2,0) $) -- ($ (60:11) + (1,0) $);
			\draw ($ (60:12) + (-2,0) $) -- ($ (60:12) + (1,0) $);
			\draw ($ (60:13) + (-3,0) $) -- ($ (60:13) + (0,0) $);
			
			
			\draw ($ (60:10) + (-2,0) $) -- ($ (60:9) + (-1,0) $);
			\draw ($ (60:12) + (-3,0) $) -- ($ (60:11) + (-2,0) $);
			\draw ($ (60:10) + (-1,0) $) -- ($ (60:9) + (0,0) $);
			\draw ($ (60:13) + (-3,0) $) -- ($ (60:12) + (-2,0) $);
			\draw ($ (60:11) + (-1,0) $) -- ($ (60:10) + (0,0) $);
			\draw ($ (60:13) + (-2,0) $) -- ($ (60:12) + (-1,0) $);
			\draw ($ (60:11) + (0,0) $) -- ($ (60:10) + (1,0) $);
			\draw ($ (60:13) + (-1,0) $) -- ($ (60:12) + (0,0) $);
			\draw ($ (60:11) + (1,0) $) -- ($ (60:10) + (2,0) $);
			\draw ($ (60:13) + (0,0) $) -- ($ (60:11) + (2,0) $);
			
			
			\draw ($ (60:12) + (-3,0) $) -- ($ (60:13) + (-3,0) $);
			\draw ($ (60:10) + (-2,0) $) -- ($ (60:12) + (-2,0) $);
			\draw ($ (60:8) + (-1,0) $) -- ($ (60:9) + (-1,0) $);
			\draw ($ (60:10) + (-1,0) $) -- ($ (60:12) + (-1,0) $);
			\draw ($ (60:8) + (0,0) $) -- ($ (60:10) + (0,0) $);
			\draw ($ (60:11) + (0,0) $) -- ($ (60:12) + (0,0) $);
			\draw ($ (60:9) + (1,0) $) -- ($ (60:10) + (1,0) $);
			\draw ($ (60:11) + (1,0) $) -- ($ (60:12) + (1,0) $);
			\draw ($ (60:10) + (2,0) $) -- ($ (60:11) + (2,0) $);
			
			\begin{scope}[rotate around={120:(Z.center)}]
				
				
				\draw ($ (60:8) + (-1,0) $) -- ($ (60:8) + (0,0) $);
				\draw ($ (60:9) + (-1,0) $) -- ($ (60:9) + (1,0) $);
				\draw ($ (60:10) + (-2,0) $) -- ($ (60:10) + (-1,0) $);
				\draw ($ (60:10) + (0,0) $) -- ($ (60:10) + (2,0) $);
				\draw ($ (60:11) + (-2,0) $) -- ($ (60:11) + (1,0) $);
				\draw ($ (60:12) + (-2,0) $) -- ($ (60:12) + (1,0) $);
				\draw ($ (60:13) + (-3,0) $) -- ($ (60:13) + (0,0) $);
				
				
				\draw ($ (60:10) + (-2,0) $) -- ($ (60:9) + (-1,0) $);
				\draw ($ (60:12) + (-3,0) $) -- ($ (60:11) + (-2,0) $);
				\draw ($ (60:10) + (-1,0) $) -- ($ (60:9) + (0,0) $);
				\draw ($ (60:13) + (-3,0) $) -- ($ (60:12) + (-2,0) $);
				\draw ($ (60:11) + (-1,0) $) -- ($ (60:10) + (0,0) $);
				\draw ($ (60:13) + (-2,0) $) -- ($ (60:12) + (-1,0) $);
				\draw ($ (60:11) + (0,0) $) -- ($ (60:10) + (1,0) $);
				\draw ($ (60:13) + (-1,0) $) -- ($ (60:12) + (0,0) $);
				\draw ($ (60:11) + (1,0) $) -- ($ (60:10) + (2,0) $);
				\draw ($ (60:13) + (0,0) $) -- ($ (60:11) + (2,0) $);
				
				
				\draw ($ (60:12) + (-3,0) $) -- ($ (60:13) + (-3,0) $);
				\draw ($ (60:10) + (-2,0) $) -- ($ (60:12) + (-2,0) $);
				\draw ($ (60:8) + (-1,0) $) -- ($ (60:9) + (-1,0) $);
				\draw ($ (60:10) + (-1,0) $) -- ($ (60:12) + (-1,0) $);
				\draw ($ (60:8) + (0,0) $) -- ($ (60:10) + (0,0) $);
				\draw ($ (60:11) + (0,0) $) -- ($ (60:12) + (0,0) $);
				\draw ($ (60:9) + (1,0) $) -- ($ (60:10) + (1,0) $);
				\draw ($ (60:11) + (1,0) $) -- ($ (60:12) + (1,0) $);
				\draw ($ (60:10) + (2,0) $) -- ($ (60:11) + (2,0) $);
				
			\end{scope}
			
			\begin{scope}[rotate around={-120:(Z.center)}]
				
				
				\draw ($ (60:8) + (-1,0) $) -- ($ (60:8) + (0,0) $);
				\draw ($ (60:9) + (-1,0) $) -- ($ (60:9) + (1,0) $);
				\draw ($ (60:10) + (-2,0) $) -- ($ (60:10) + (-1,0) $);
				\draw ($ (60:10) + (0,0) $) -- ($ (60:10) + (2,0) $);
				\draw ($ (60:11) + (-2,0) $) -- ($ (60:11) + (1,0) $);
				\draw ($ (60:12) + (-2,0) $) -- ($ (60:12) + (1,0) $);
				\draw ($ (60:13) + (-3,0) $) -- ($ (60:13) + (0,0) $);
				
				
				\draw ($ (60:10) + (-2,0) $) -- ($ (60:9) + (-1,0) $);
				\draw ($ (60:12) + (-3,0) $) -- ($ (60:11) + (-2,0) $);
				\draw ($ (60:10) + (-1,0) $) -- ($ (60:9) + (0,0) $);
				\draw ($ (60:13) + (-3,0) $) -- ($ (60:12) + (-2,0) $);
				\draw ($ (60:11) + (-1,0) $) -- ($ (60:10) + (0,0) $);
				\draw ($ (60:13) + (-2,0) $) -- ($ (60:12) + (-1,0) $);
				\draw ($ (60:11) + (0,0) $) -- ($ (60:10) + (1,0) $);
				\draw ($ (60:13) + (-1,0) $) -- ($ (60:12) + (0,0) $);
				\draw ($ (60:11) + (1,0) $) -- ($ (60:10) + (2,0) $);
				\draw ($ (60:13) + (0,0) $) -- ($ (60:11) + (2,0) $);
				
				
				\draw ($ (60:12) + (-3,0) $) -- ($ (60:13) + (-3,0) $);
				\draw ($ (60:10) + (-2,0) $) -- ($ (60:12) + (-2,0) $);
				\draw ($ (60:8) + (-1,0) $) -- ($ (60:9) + (-1,0) $);
				\draw ($ (60:10) + (-1,0) $) -- ($ (60:12) + (-1,0) $);
				\draw ($ (60:8) + (0,0) $) -- ($ (60:10) + (0,0) $);
				\draw ($ (60:11) + (0,0) $) -- ($ (60:12) + (0,0) $);
				\draw ($ (60:9) + (1,0) $) -- ($ (60:10) + (1,0) $);
				\draw ($ (60:11) + (1,0) $) -- ($ (60:12) + (1,0) $);
				\draw ($ (60:10) + (2,0) $) -- ($ (60:11) + (2,0) $);
				
			\end{scope}
			
			\begin{scope}[xscale=-1,shift={(-6,0)}]
				
				
				\draw ($ (60:8) + (-1,0) $) -- ($ (60:8) + (0,0) $);
				\draw ($ (60:9) + (-1,0) $) -- ($ (60:9) + (1,0) $);
				\draw ($ (60:10) + (-2,0) $) -- ($ (60:10) + (-1,0) $);
				\draw ($ (60:10) + (0,0) $) -- ($ (60:10) + (2,0) $);
				\draw ($ (60:11) + (-2,0) $) -- ($ (60:11) + (1,0) $);
				\draw ($ (60:12) + (-2,0) $) -- ($ (60:12) + (1,0) $);
				\draw ($ (60:13) + (-3,0) $) -- ($ (60:13) + (0,0) $);
				
				
				\draw ($ (60:10) + (-2,0) $) -- ($ (60:9) + (-1,0) $);
				\draw ($ (60:12) + (-3,0) $) -- ($ (60:11) + (-2,0) $);
				\draw ($ (60:10) + (-1,0) $) -- ($ (60:9) + (0,0) $);
				\draw ($ (60:13) + (-3,0) $) -- ($ (60:12) + (-2,0) $);
				\draw ($ (60:11) + (-1,0) $) -- ($ (60:10) + (0,0) $);
				\draw ($ (60:13) + (-2,0) $) -- ($ (60:12) + (-1,0) $);
				\draw ($ (60:11) + (0,0) $) -- ($ (60:10) + (1,0) $);
				\draw ($ (60:13) + (-1,0) $) -- ($ (60:12) + (0,0) $);
				\draw ($ (60:11) + (1,0) $) -- ($ (60:10) + (2,0) $);
				\draw ($ (60:13) + (0,0) $) -- ($ (60:11) + (2,0) $);
				
				
				\draw ($ (60:12) + (-3,0) $) -- ($ (60:13) + (-3,0) $);
				\draw ($ (60:10) + (-2,0) $) -- ($ (60:12) + (-2,0) $);
				\draw ($ (60:8) + (-1,0) $) -- ($ (60:9) + (-1,0) $);
				\draw ($ (60:10) + (-1,0) $) -- ($ (60:12) + (-1,0) $);
				\draw ($ (60:8) + (0,0) $) -- ($ (60:10) + (0,0) $);
				\draw ($ (60:11) + (0,0) $) -- ($ (60:12) + (0,0) $);
				\draw ($ (60:9) + (1,0) $) -- ($ (60:10) + (1,0) $);
				\draw ($ (60:11) + (1,0) $) -- ($ (60:12) + (1,0) $);
				\draw ($ (60:10) + (2,0) $) -- ($ (60:11) + (2,0) $);
				
			\end{scope}
			
			\begin{scope}[xscale=-1,shift={(-6,0)},rotate around={120:(Z.center)}]
				
				
				\draw ($ (60:8) + (-1,0) $) -- ($ (60:8) + (0,0) $);
				\draw ($ (60:9) + (-1,0) $) -- ($ (60:9) + (1,0) $);
				\draw ($ (60:10) + (-2,0) $) -- ($ (60:10) + (-1,0) $);
				\draw ($ (60:10) + (0,0) $) -- ($ (60:10) + (2,0) $);
				\draw ($ (60:11) + (-2,0) $) -- ($ (60:11) + (1,0) $);
				\draw ($ (60:12) + (-2,0) $) -- ($ (60:12) + (1,0) $);
				\draw ($ (60:13) + (-3,0) $) -- ($ (60:13) + (0,0) $);
				
				
				\draw ($ (60:10) + (-2,0) $) -- ($ (60:9) + (-1,0) $);
				\draw ($ (60:12) + (-3,0) $) -- ($ (60:11) + (-2,0) $);
				\draw ($ (60:10) + (-1,0) $) -- ($ (60:9) + (0,0) $);
				\draw ($ (60:13) + (-3,0) $) -- ($ (60:12) + (-2,0) $);
				\draw ($ (60:11) + (-1,0) $) -- ($ (60:10) + (0,0) $);
				\draw ($ (60:13) + (-2,0) $) -- ($ (60:12) + (-1,0) $);
				\draw ($ (60:11) + (0,0) $) -- ($ (60:10) + (1,0) $);
				\draw ($ (60:13) + (-1,0) $) -- ($ (60:12) + (0,0) $);
				\draw ($ (60:11) + (1,0) $) -- ($ (60:10) + (2,0) $);
				\draw ($ (60:13) + (0,0) $) -- ($ (60:11) + (2,0) $);
				
				
				\draw ($ (60:12) + (-3,0) $) -- ($ (60:13) + (-3,0) $);
				\draw ($ (60:10) + (-2,0) $) -- ($ (60:12) + (-2,0) $);
				\draw ($ (60:8) + (-1,0) $) -- ($ (60:9) + (-1,0) $);
				\draw ($ (60:10) + (-1,0) $) -- ($ (60:12) + (-1,0) $);
				\draw ($ (60:8) + (0,0) $) -- ($ (60:10) + (0,0) $);
				\draw ($ (60:11) + (0,0) $) -- ($ (60:12) + (0,0) $);
				\draw ($ (60:9) + (1,0) $) -- ($ (60:10) + (1,0) $);
				\draw ($ (60:11) + (1,0) $) -- ($ (60:12) + (1,0) $);
				\draw ($ (60:10) + (2,0) $) -- ($ (60:11) + (2,0) $);
				
			\end{scope}
			
			\begin{scope}[xscale=-1,shift={(-6,0)},rotate around={-120:(Z.center)}]
				
				
				\draw ($ (60:8) + (-1,0) $) -- ($ (60:8) + (0,0) $);
				\draw ($ (60:9) + (-1,0) $) -- ($ (60:9) + (1,0) $);
				\draw ($ (60:10) + (-2,0) $) -- ($ (60:10) + (-1,0) $);
				\draw ($ (60:10) + (0,0) $) -- ($ (60:10) + (2,0) $);
				\draw ($ (60:11) + (-2,0) $) -- ($ (60:11) + (1,0) $);
				\draw ($ (60:12) + (-2,0) $) -- ($ (60:12) + (1,0) $);
				\draw ($ (60:13) + (-3,0) $) -- ($ (60:13) + (0,0) $);
				
				
				\draw ($ (60:10) + (-2,0) $) -- ($ (60:9) + (-1,0) $);
				\draw ($ (60:12) + (-3,0) $) -- ($ (60:11) + (-2,0) $);
				\draw ($ (60:10) + (-1,0) $) -- ($ (60:9) + (0,0) $);
				\draw ($ (60:13) + (-3,0) $) -- ($ (60:12) + (-2,0) $);
				\draw ($ (60:11) + (-1,0) $) -- ($ (60:10) + (0,0) $);
				\draw ($ (60:13) + (-2,0) $) -- ($ (60:12) + (-1,0) $);
				\draw ($ (60:11) + (0,0) $) -- ($ (60:10) + (1,0) $);
				\draw ($ (60:13) + (-1,0) $) -- ($ (60:12) + (0,0) $);
				\draw ($ (60:11) + (1,0) $) -- ($ (60:10) + (2,0) $);
				\draw ($ (60:13) + (0,0) $) -- ($ (60:11) + (2,0) $);
				
				
				\draw ($ (60:12) + (-3,0) $) -- ($ (60:13) + (-3,0) $);
				\draw ($ (60:10) + (-2,0) $) -- ($ (60:12) + (-2,0) $);
				\draw ($ (60:8) + (-1,0) $) -- ($ (60:9) + (-1,0) $);
				\draw ($ (60:10) + (-1,0) $) -- ($ (60:12) + (-1,0) $);
				\draw ($ (60:8) + (0,0) $) -- ($ (60:10) + (0,0) $);
				\draw ($ (60:11) + (0,0) $) -- ($ (60:12) + (0,0) $);
				\draw ($ (60:9) + (1,0) $) -- ($ (60:10) + (1,0) $);
				\draw ($ (60:11) + (1,0) $) -- ($ (60:12) + (1,0) $);
				\draw ($ (60:10) + (2,0) $) -- ($ (60:11) + (2,0) $);
				
			\end{scope}
			
			\begin{scope}[bullet/.style={circle, fill, minimum size=4pt,               inner sep=0pt, outer sep=0pt},nodes=bullet]
				
				\node[bullet] at ($ (60:11) + (-.5,0) $) {};
				\node[bullet] at ($ (60:13) + (-.5,0) $) {};
				\node[bullet] at ($ (60:14) + (-.5,0) $) {};
				
				\node[bullet] at ($ (60:4) + (2.5,0) $) {};
				\node[bullet] at ($ (60:2) + (4.5,0) $) {};
				\node[bullet] at ($ (60:1) + (5.5,0) $) {};
				
				\draw[ultra thick] ($ (60:11) + (-.5,0) $) -- ($ (60:10) + (.5,0) $) -- ($ (60:7) + (.5,0) $) -- ($ (60:6) + (1.5,0) $) -- ($ (60:5) + (1.5,0) $) -- ($ (60:4) + (2.5,0) $);
				
				\draw[ultra thick] ($ (60:13) + (-.5,0) $) -- ($ (60:12) + (.5,0) $) -- ($ (60:11) + (.5,0) $) -- ($ (60:10) + (1.5,0) $) -- ($ (60:8) + (1.5,0) $) -- ($ (60:7) + (2.5,0) $) -- ($ (60:6) + (2.5,0) $) -- ($ (60:4) + (4.5,0) $) -- ($ (60:2) + (4.5,0) $);
				
				\draw[ultra thick] ($ (60:14) + (-.5,0) $) -- ($ (60:11) + (2.5,0) $) -- ($ (60:10) + (2.5,0) $) -- ($ (60:9) + (3.5,0) $) -- ($ (60:6) + (3.5,0) $) -- ($ (60:4) + (5.5,0) $) -- ($ (60:1) + (5.5,0) $);
				
				\draw[dotted,thick] ($ (60:8) + (0,0) $) -- ($ (60:14) + (0,0) $);
				\draw[dotted,thick] ($ (60:6) + (0,0) $) -- ($ (60:0) + (6,0) $);				
				
			\end{scope}
			
		\end{tikzpicture}
		$\longleftrightarrow$
		\ytableaushort{
			777633,
			\none65322,
			\none\none41
			}
		\caption{An example of the bijective correspondence between cyclically symmetric lozenge tilings of a holey hexagon and descending plane partitions. The lozenge tiling on the left side is the same as in Figure~\ref{fig:ExampleRhombusTiling} but rotated by $30^\circ$. The dotted lines mark a third of the lozenge tiling as the fundamental area.}
		\label{fig:DPP}
	\end{figure}
	
	While imposing an additional reflective symmetry on the cyclically symmetric lozenge tilings arises naturally from a geometric perspective, 
	characterizing the corresponding descending plane partitions yields a notion that is more obscure than enlightening. Details are provided in 
	\cite{MilRobRum83}, see in particular Conjecture~3S. For this reason, we choose to work with lozenge tilings rather than descending plane partitions.

\end{remark}

\subsection*{Outline of the paper} 
The paper is organized as follows: In Section~\ref{sec:AMT}, we introduce a formula for the generating function of arrowed monotone triangles.
We observe that for our particular bottom row, we can express the generating function with the help of a determinant.   
In Section~\ref{sec:first}, we provide the proof of Theorem~\ref{interpret1} and relate the families of lattice paths from Theorem~\ref{interpret1} to the holey symmetric lozenge tilings that are known to be equinumerous with VSASMs; this relation is also proved combinatorially in Section~\ref{ap:bijective_proofs}. In Section~\ref{sec:third}, we establish Theorem~\ref{interpret4} by relating the families of lattices paths from Theorems~\ref{interpret1} to the signless interpretation in terms of plane partitions in Theorem~\ref{interpret4}. Finally, in Section~\ref{sec:further_lattice_paths}, we establish two additional families of lattice paths with the same generating function as our particular arrowed monotone triangles.

In Appendix~\ref{ap:computational_proof}, we present a second proof of Theorem~\ref{interpret4}, which does not rely on Theorem~\ref{interpret1}.

\section{The generating function of arrowed monotone triangles}
\label{sec:AMT} 

We point out that there is a close analogy between arrowed monotone triangles and their generating function with respect to a fixed bottom row on one side and  semistandard Young tableaux and Schur polynomials on the other side: Recall that Gelfand--Tsetlin patterns are defined as monotone triangles with the condition of the strict increase along rows dropped. As demonstrated in the introduction, semistandard Young tableaux are in bijective correspondence with Gelfand--Tsetlin patterns. Schur polynomials are multivariate generating functions of semistandard Young tableaux with respect to the weight 
\begin{equation*}
X_1^{\text{$\#$ of $1$'s}} X_2^{\text{$\#$ of $2$'s}} \cdots X_n^{\text{$\#$ of $n$'s}}.
\end{equation*}
However, in terms of Gelfand--Tsetlin patterns, this weight can obviously be expressed as 
\begin{equation*}
\prod_{i=1}^n X_i^{\text{(sum of entries in row $i$)}-\text{(sum of entries in row $i-1$)}},
\end{equation*}
which is analogous to the weight we use for arrowed monotone triangles.
This analogy extends also to the following formula for the generating function of arrowed monotone triangles with given 
bottom row, which can be expressed in terms of (extended) Schur polynomials. We define 
\begin{equation*} 
s_{(k_n,\ldots,k_1)}(X_1,\ldots,X_n) \coloneq 
\frac{\det_{1 \le i,j \le n} \left( X_i^{k_j+j-1} \right)}{\prod_{1 \le i < j \le n} (X_j-X_i)}.
\end{equation*} 
Note that for $0 \le k_1 \le k_2 \le \ldots \le k_n$, $s_{(k_n,\ldots,k_1)}(X_1,\ldots,X_n)$ is the Schur polynomial of the partition $(k_n,k_{n-1},\ldots,k_1)$. In \cite[Theorem 2.2]{nASMDPP}, the following theorem has been established, which is one of our main ingredients for the present paper.

\begin{theorem} 
\label{robbins}
The generating function of arrowed monotone triangles with bottom row $k_1 <  k_2 < \ldots < k_n$ 
is 
\begin{equation*}
\prod_{i=1}^{n} (u X_i + v X_i^{-1} + w) 
\prod_{1 \le i < j \le n} \left( u  \e_{k_i} + v \e_{k_j}^{-1} + w \e_{k_i} \e_{k_j}^{-1}  \right) 
s_{(k_n,k_{n-1},\ldots,k_1)}(X_1,\ldots,X_n), 
\end{equation*}
where $\e_x$ denotes the \emph{shift operator}, defined as $\e_x p(x) \coloneq p(x+1)$. 
\end{theorem} 
Using the antisymmetrizer, defined as 
\begin{equation*}
\asym_{X_1,\ldots,X_n} \left[ f(X_{1},\ldots,X_{n}) \right] \coloneq \sum_{\sigma \in {\mathcal S}_n}  \sgn \sigma 
f(X_{\sigma(1)},\ldots,X_{\sigma(n)}), 
\end{equation*}
we rewrite the expression from Theorem~\ref{robbins} and obtain
\begin{multline}
\label{eq:antisymformula}
\prod_{i=1}^{n} (u  X_i + v  X_i^{-1}+w) \prod_{1 \le i < j \le n} \left( u  \e_{k_i} + v  \e_{k_j}^{-1} + w  \e_{k_i} \e_{k_j}^{-1} \right) 
s_{(k_n,k_{n-1},\ldots,k_1)}(X_1,\ldots,X_n) \\
= \prod_{i=1}^{n} (u  X_i + v  X_i^{-1}+w) \prod_{1 \le i < j \le n} \left( u  \e_{k_i} + v  \e_{k_j}^{-1} + w  \e_{k_i} \e_{k_j}^{-1} \right)
\frac{\asym_{X_1,\ldots,X_n} \left[ \prod_{i=1}^{n} X_i^{k_i+i-1} \right] }{\prod_{1 \le i < j \le n} (X_j - X_i)} \\ 
= \prod_{i=1}^{n} (u  X_i + v  X_i^{-1}+w)  \frac{\asym_{X_1,\ldots,X_n} \left[ \prod_{1 \le i < j \le n} \left( u  \e_{k_i} + v  \e_{k_j}^{-1} + w  \e_{k_i} \e_{k_j}^{-1} \right) \prod_{i=1}^{n} X_i^{k_i+i-1} \right] }{\prod_{1 \le i < j \le n} (X_j - X_i)} \\ 
= \prod_{i=1}^{n} (u  X_i + v  X_i^{-1}+w) \frac{\asym_{X_1,\ldots,X_n} \left[ \prod_{1 \le i <  j \le n} \left( u X_i  + v  X_j^{-1} + w  X_i X_j^{-1} \right) \prod_{i=1}^{n} X_i^{k_i+i-1} \right] }{\prod_{1 \le i < j \le n} (X_j - X_i)} \\
=  \frac{\asym_{X_1,\ldots,X_n} \left[ \prod_{1 \le i \le j \le n} \left( u  X_j + v  X_i^{-1}  + w  \right) \prod_{i=1}^{n} X_i^{k_i+n-i} \right] }{\prod_{1 \le i < j \le n} (X_j - X_i)}.
\end{multline} 
We set $k_i=2i-2$ for all $1\le i \le n$ to restrict to VSASMs. We also multiply with 
\begin{equation}
\label{Bfactor}  
\prod_{1 \le i \le j \le n} \left( u - v X_i^{-1} X_j^{-1}  \right),  
\end{equation} 
which is a symmetric function and can therefore be put inside the antisymmetrizer. Thus, we obtain
\begin{multline}\label{eq:genfun_AMT1} 
\frac{\asym_{X_1,\ldots,X_n} \left[ \prod_{1 \le i \le j \le n} \left( u  X_j + v  X_i^{-1}  + w  \right) 
\left(u -  v X_i^{-1} X_j^{-1}  \right) 
\prod_{i=1}^{n} X_i^{i+n-2} \right] }{\prod_{1 \le i < j \le n} (X_j - X_i)} \\
= \prod_{i=1}^{n} X_i^{n-2} \frac{\asym_{X_1,\ldots,X_n} \left[ \prod_{1 \le i \le j \le n} \left( u  X_j + v  X_i^{-1}  + w  \right) 
\left(u X_j -  v X_i^{-1}  \right)  \right] }{\prod_{1 \le i < j \le n} (X_j - X_i)}.
\end{multline}
The following lemma allows us to manipulate the previous expression even further. It first appeared in \cite[Lemma~3.1]{TSPPmany} with a computational proof based on algebraic manipulations, while in 
\cite[Lemma 4.1]{TSPP} also a combinatorial proof is provided. In this paper, we give a new proof which is shorter than the other two.

\begin{lemma}
	\label{general}
	Let $n \ge 1$, and $\mathbf{Y}=(Y_1,\ldots,Y_n), \mathbf{Z}=(Z_1,\ldots,Z_n)$ be two sets of indeterminants.
	Then
	\begin{equation}\label{eq:antisymlemma}
		\det_{1 \le i, j \le n} \left( Y_i^j - Z_i^j \right) 
		= \overline{\asym}  \left[  \prod_{1 \le i \le j \le n} (Y_j-Z_i)     \right],
	\end{equation}
	with 
	\begin{equation*}
	\overline{\asym} \left[f(\mathbf{Y};\mathbf{Z})\right] \coloneq \sum_{\sigma \in {\mathcal{S}_n}} \sgn \sigma f(Y_{\sigma(1)},\ldots,Y_{\sigma(n)};Z_{\sigma(1)},\ldots,Z_{\sigma(n)}).
	\end{equation*}
\end{lemma}

\begin{proof}  The proof is by induction on $n$. The result is obvious for $n=1$. 
	Let $A_n(\mathbf{Y};\mathbf{Z})$ denote the right-hand side of \eqref{eq:antisymlemma}, and 
	observe that we have 
	\begin{equation*}
	A_n(\mathbf{Y};\mathbf{Z})
	=  \sum_{i=1}^{n} (-1)^{i+1}  \left( \prod_{k=1}^n (Y_k-Z_i)  \right) 
	A_{n-1}(Y_1,\ldots,\widehat{Y_i},\ldots,Y_n;Z_1,\ldots,\widehat{Z_i},\ldots,Z_n),
	\end{equation*}
	where $\widehat{Y_i}$ and $\widehat{Z_i}$ means that $Y_i$ and $Z_i$ are omitted, respectively. Using the induction hypothesis, it follows that $A_n(\mathbf{Y};\mathbf{Z})$ is 
	\begin{equation*}
	(-1)^{n-1} \det_{1 \le i, j \le n} \left( \begin{cases} Y_i^j - Z_i^j,  & j < n, \\ 
		\prod_{k=1}^n (Y_k-Z_i), & j=n. 
	\end{cases}  \right)
	\end{equation*}
	For $j \in \{1,2,\ldots,n-1\}$, we multiply the $j$-th column with $(-1)^j e_{n-j}(Y_1,\ldots,Y_n)$ and add it to the last column, where $e_{n-j}$ denotes the \emph{elementary symmetric polynomial} of degree~$n-j$. The entry in the $i$-th row and $n$-th column is then
	\begin{multline*}
		\sum_{j=1}^{n-1} (-1)^j (Y_i^j - Z_i^j) e_{n-j}(Y_1,\ldots,Y_n) + \prod_{k=1}^n (Y_k-Z_i) \\ = 
		\sum_{j=1}^{n-1} (-1)^j (Y_i^j - Z_i^j) e_{n-j}(Y_1,\ldots,Y_n) + \sum_{j=0}^n (-1)^j Z_i^j e_{n-j}(Y_1,\ldots,Y_n) \\ = 
		\sum_{j=1}^{n-1} (-1)^j Y_i^j e_{n-j}(Y_1,\ldots,Y_n) + e_n(Y_1,\ldots,Y_n) + (-1)^n Z_i^n \\ =
		\prod_{k=1}^n (Y_k - Y_i) - e_n(Y_1,\ldots,Y_n) - (-1)^n Y_i^n + e_n(Y_1,\ldots,Y_n) + (-1)^n Z_i^n \\ = 
		(-1)^n \left( Z_i^n - Y_i^n \right), 
	\end{multline*}                            
	and this proves the assertion.                     
\end{proof}

The crucial observation is that 
\begin{equation*}
\left( u  X_j + v  X_i^{-1}  + w  \right) \left(u X_j -  v X_i^{-1}  \right) =
u^2 X_j^2 + u w X_j - v^2 X_i^{-2} - v w X_i^{-1},
\end{equation*}
and so the lemma is applicable to \eqref{eq:genfun_AMT1}; note that for this it was important to multiply with \eqref{Bfactor}: 
Setting $Y_j = u^2 X_j^2 + u w X_j$ and $Z_i = v^2 X_i^{-2} + v w X_i^{-1}$ in the lemma and applying it to \eqref{eq:genfun_AMT1} yields
\begin{equation*}
\prod_{i=1}^{n} X_i^{n-2} \frac{
	\det_{1 \le i,j \le n} \left( \left(u^2 X_i^2 + u w X_i\right)^j - \left(v^2 X_i^{-2} + v w X_i^{-1}\right)^j \right)}
{\prod_{1 \le i < j \le n} (X_j - X_i)}.
\end{equation*}
We divide by \eqref{Bfactor} and eventually obtain the following generating function of arrowed monotone triangles with bottom row $0,2,\ldots,2n-2$:
\begin{equation} 
	\label{bialternant} 
	\prod_{i=1}^{n} X_i^{n-2} (u-v X_i^{-2})^{-1}   
	\frac{\det_{1 \le i,j \le n} \left( \left(u^2 X_i^2 + u w X_i\right)^j - \left(v^2 X_i^{-2} + v w X_i^{-1}\right)^j \right)}{\prod_{1 \le i < j \le n} (X_j - X_i)\left( u - v X_i^{-1} X_j^{-1}  \right)}.
\end{equation} 

\section{\texorpdfstring{Proof of Theorem~\ref{interpret1}}{Proof of Theorem~2.4}}
\label{sec:first} 
The following lemma is an important tool to transform bialternant formulas into Jacobi--Trudi-type formulas; the proof can be found in \cite[Lemma 7.2]{nASMDPP}. The lemma involves 
\emph{complete homogeneous symmetric functions} $h_k$ also of negative degree~$k$, which are defined as follows:
\begin{equation*}
h_k(X_1,\ldots,X_n) \coloneq (-1)^{n+1} \sum_{l_1+\ldots+l_n=k,l_i<0} X_1^{l_1} X_2^{l_2} \cdots X_n^{l_n}.
\end{equation*}
Note that 
$h_{k}(X_1,\ldots,X_n)$ vanishes for $-n+1 \le k \le -1$.
For $k \le -n$, they can be expressed in terms of those with positive degree as 
\begin{multline} 
\label{negative} 
h_k(X_1,\ldots,X_n)
= (-1)^{n+1} (X_1 \cdots X_n)^{-1}  \sum_{l_1+\ldots+l_n=-k-n,l_i \ge 0} X_1^{-l_1} X_2^{-l_2} \cdots X_n^{-l_n} \\
= (-1)^{n+1} (X_1 \cdots X_n)^{-1} h_{-k-n}(X_1^{-1},\ldots,X_n^{-1}).
\end{multline} 
In fact, it is not hard to see that the relation 
\begin{equation*}
	h_k(X_1,\ldots,X_n) = (-1)^{n+1} (X_1 \cdots X_n)^{-1} h_{-k-n}(X_1^{-1},\ldots,X_n^{-1})
\end{equation*}	
is true for any integer $k$.

\begin{lemma}
\label{lem:det} 
Let $f_j(Y)$ be formal Laurent series for $1 \le j \le n$, and define 
\begin{equation*}
f_j[Y_1,\ldots,Y_i]=\sum_{k \in \mathbb{Z}} \langle Y^{k} \rangle f_j(Y) \cdot h_{k-i+1}(Y_1,\ldots,Y_i),
\end{equation*}
where $\langle Y^{k} \rangle f_j(Y)$ denotes the coefficient of $Y^{k}$ in $f_j(Y)$.
Then 
\begin{equation*}
\frac{\det_{1 \le i, j \le n} \left( f_j(Y_i) \right) }{\prod_{1 \le i < j \le n} (Y_j - Y_i)} = \det_{1 \le i, j \le n} \left( f_j[Y_1,\ldots,Y_i] \right).
\end{equation*}
\end{lemma}

We rewrite \eqref{bialternant} as follows:
\begin{equation}
\label{bialternatint1} 
\prod_{i=1}^{n} X_i^{n-1}  
\frac{\det\limits_{1 \le i,j \le n} \left( \frac{\left(u^2 X_i^2 + u w X_i\right)^j - \left(v^2 X_i^{-2} + v w X_i^{-1}\right)^j}{u X_i - v X_i^{-1}} \right)}{\prod\limits_{1 \le i < j \le n} \left( (u X_j + v X_j^{-1}) - (u X_i + v X_i^{-1})\right)}.
\end{equation} 
We aim at applying Lemma~\ref{lem:det} to this bialternant with $Y_i=u X_i + v X_i^{-1}$. For this purpose, we need to express the entries in the $i$-th row of the matrix that underlies the determinant in \eqref{bialternatint1} as formal Laurent series in $Y_i$, which is possible since these entries are invariant under $X_i \mapsto  u^{-1} v X_i^{-1}$. We rewrite the entries as follows:
\begin{equation}
\label{entry}
\frac{\left(u^2 X_i^2 + u w X_i\right)^j - \left(v^2 X_i^{-2} + v w X_i^{-1}\right)^j}{u X_i - v X_i^{-1}} =  
\sum_{k=0}^j \binom{j}{k} w^{j-k} 
\frac{u^{j+k} X_i^{j+k} - v^{j+k} X_i^{-(j+k)}}{u X_i - v X_i^{-1}}.
\end{equation}

The next lemma is useful to transform this further.

\begin{lemma} 
\label{basis} 
For $l \ge 0$, we have 
\begin{equation*}
\frac{u^{l} X^{l} - v^{l} X^{-l}}{u X - v X^{-1}} = \sum_{r=0}^{(l-1)/2} (-uv)^r \binom{l-r-1}{r} (u X + v X^{-1})^{l-2r-1}.
\end{equation*}
\end{lemma}

\begin{proof}
Using the geometric series expansion, the left-hand side can be written as  
$\sum_{k=0}^{l-1} u^k \allowbreak v^{-k+l-1}  X^{2k-l+1}$.
 The right-hand side is equal to 
\begin{multline*} 
 \sum_{r=0}^{(l-1)/2} \sum_{s \ge 0}  (-1)^r \binom{l-r-1}{r} \binom{l-2r-1}{s} u^{r+s}  v^{-r-s+l-1} X^{2r+2s-l+1}  \\ = 
 \sum_{k \le l-1}  u^{k}  v^{-k+l-1} X^{2k-l+1} \sum_{r=0}^{(l-1)/2} (-1)^r \binom{l-r-1}{r} \binom{l-2r-1}{k-r}. 
 \end{multline*} 
 Now, using the Chu--Vandermonde identity in the third step, the inner sum simplifies as follows and the proof is complete:
 \begin{multline*} 
 \sum_{r=0}^{l-k-1} (-1)^r \frac{(l-r-1)!}{r! (k-r)! (l-r-k-1)!}  = 
 \sum_{r=0}^{l-k-1} (-1)^r \binom{k}{r} \binom{l-r-1}{l-r-1-k}   \\ = 
 \sum_{r=0}^{l-k-1} (-1)^{l-1-k} \binom{k}{r} \binom{-k-1}{l-r-1-k} = (-1)^{l-1-k} \binom{-1}{l-1-k} = 
 \begin{cases} 1, & k \le l-1, \\ 0, & k \ge l. \end{cases}
 \end{multline*} 
\end{proof} 
  
We use Lemma~\ref{basis} to rewrite \eqref{entry} as follows:
\begin{multline*} 
\sum_{k=0}^{j} \sum_{r=0}^{(j+k-1)/2}  \binom{j}{k} w^{j-k} (-uv)^r \binom{j+k-r-1}{r} (u X_i + v X_i^{-1})^{j+k-2r-1} \\ = 
\sum_{k,r \ge 0} (-uv)^r w^{2j-k-2r-1} \binom{j}{k-j+2r+1} \binom{k+r}{r} (u X_i + v X_i^{-1})^{k}.
\end{multline*} 
Applying Lemma~\ref{lem:det} to \eqref{bialternatint1} with $Y_i=u X_i + v X_i^{-1}$, we have 
\begin{multline*}
f_j[u X_1+ v X_1^{-1},\ldots,u X_i+ v X_i^{-1}] \\
= 
\sum_{k,r \ge 0} (-uv)^r w^{2j-k-2r-1} \binom{k+r}{r} \binom{j}{2j-k-2r-1}  
h_{k-i+1}(u X_1+ v X_1^{-1},\ldots,u X_i+ v X_i^{-1}) \\
= \sum_{\substack{k,r \in \mathbb{Z}\\2j-k-2r-1 \ge 0}} 
(-u v)^r w^k \binom{2j-k-r-1}{r} \binom{j}{k} 
h_{2j-k-2r-i}(u X_1+ v X_1^{-1},\ldots,u X_i+ v X_i^{-1}).
\end{multline*} 
We set $p=2j-k$ and eliminate $k$. Then we let $q=p-2r$ and eliminate $r$. 
Also exchange the role of $i$ and $j$, which is possible since we are 
interested in the determinant. We obtain
\begin{equation} 
\label{LGV1} 
\sum_{p \ge 1} w^{2i-p} \binom{i}{2i-p} 
\sum_{q \ge 1, 2 \mid (p-q)} 
(-u v)^{(p-q)/2} \binom{(p+q)/2-1}{(p-q)/2}  
h_{q-j}(u X_1+ v X_1^{-1},\ldots,u X_j+ v X_j^{-1}),
\end{equation} 
which we denote by $a_{i,j} $.

Our combinatorial interpretations are based on the Lindstr\"om--Gessel--Viennot lemma \cite{Lin73,GesVie85,GesVie89}; we state a version that is useful for us. 

\begin{lemma} 
\label{LGV} 
Suppose $G$ is a finite directed acyclic edge-weighted graph with weight function~$\W$. For $i \in \{1,2,\ldots,n\}$, let $A_i, E_i$ be sets of vertices of $G$ and assume that weights are also assigned to the vertices in these sets, also indicated by $\W$. Further, let 
${\mathcal P}(A_i,E_j)$ be the generating function of paths from $A_i$ to $E_j$, that is,
\begin{equation*}
	{\mathcal P}(A_i,E_j) = \sum_{\substack{(x,y) \in A_i \times E_j\\\text{$P$ path $x \to y$}}} 
\W(x) \W(y) \prod_{\text{$e \in P$}} \W(e).
\end{equation*}
Then $\det_{1 \le i,j \le n} \left( {\mathcal P}(A_i,E_j) \right)$ is the signed generating function of families of $n$ pairwise nonintersecting paths $(P_1,\ldots,P_n)$ with the property that there exists a permutation $\sigma$ of $\{1,2,\ldots,\allowbreak n\}$ (not necessarily the same for all $n$-tuples) such that $P_i$ is a directed path from a vertex $x_i$ of $A_i$ to a vertex $y_i$ of $E_{\sigma(i)}$. The ``signed'' weight is 
\begin{equation*}
\sgn \sigma \prod_{i=1}^n \W(x_i) \W(y_i) \prod_{e \in P_i} \W(e). 
\end{equation*}
\end{lemma}  

We interpret \eqref{LGV1} as a certain generating function of lattice paths from $(-i,i)$ to 
$(j-1,-j+2)$ as follows, see Figure~\ref{fig:SinglePathFirstInterpretation}:

\begin{itemize}
\item In the region $\{(x,y) \mid x \le 0\}$, the step set is $\{(1,1),(1,0)\}$.
Assuming that $(0,p)$ is the first lattice point on $x=0$, there are $p-i$ steps of type  $(1,1)$ and $2i-p$ steps of type $(1,0)$, so that there are $\binom{i}{2i-p}$ such paths. Steps of type $(1,0)$ are equipped with the weight $w$. 
\item In the region $\{(x,y) \mid  x \ge 0, y \ge 1\}$, the step set $\{(1,-1),(0,-2)\}$, and we aim at reaching the line $y=1$. 
The steps of type $(0,-2)$ are equipped weight $-uv$. Assuming that $(q-1,1)$ is the first lattice point on $y=1$, there are $q-1$ steps of type $(1,-1)$ and $(p-q)/2$ steps of type $(0,-2)$, so there are $\binom{(p+q)/2-1}{(p-q)/2}$ such paths. Note that the special form of the step set implies $2 \mid (p-q)$.
\item 
In the region $\{(x,y) \mid  x \ge 0, y \le 1\}$, the step set is $\{(-1,0),(0,-1)\}$. Horizontal steps with distance $d$ from the line $y=2$ have weight $u X_d+v X_d^{-1}$. The generating function is 
$h_{q-j}(u X_1+ v X_1^{-1},\ldots,u X_j+ v X_j^{-1})$ as the number of $(-1,0)$ steps is $q-j$, while the number of $(0,-1)$ steps is $j-1$ . 
\item Paths starting in $(-i,i)$ are equipped with the additional weight $X_i^{n-1}$. 
\end{itemize} 

\begin{figure}[htb]
	\centering
	\begin{tikzpicture}[scale=.45,baseline=(current bounding box.center)]
		\fill [verylight-gray] (-6.75,-2.75) rectangle (0,9.75);
		\fill [light-gray] (0,-2.75) rectangle (8.75,1);
		
		\draw [help lines,step=1cm,dashed] (-6.75,-2.75) grid (8.75,9.75);
		
		\draw[->,thick] (-6.75,0)--(8.75,0) node[right]{$x$};
		\draw[->,thick] (0,-2.75)--(0,9.75) node[above]{$y$};
		
		\fill (-5,5) circle (5pt);
		\fill (2,-1) circle (5pt);
		\fill (0,8) circle (5pt);
		\fill (5,1) circle (5pt);
		
		\path[decoration=arrows, decorate] (-5,5) --++ (1,1) --++ (1,0) --++ (1,1) --++ (1,1) --++ (1,0) --++ (1,-1) --++ (1,-1) --++ (0,-2) --++ (1,-1) --++ (1,-1) --++ (1,-1) --++ (-1,0) --++ (-1,0)  --++ (0,-1) --++ (0,-1) --++ (-1,0);
		
		\node[below] at (-5,5) {\contour{verylight-gray}{\scriptsize$(-i,i)$}};
		\node[below] at (2.25,-1) {\contour{light-gray}{\scriptsize$(j-1,-j+2)$}};
		\node[above right] at (0,8) {\contour{white}{\scriptsize$(0,p)$}};
		\node[above right] at (5,1) {\contour{white}{\scriptsize$(q-1,1)$}};
		
	\end{tikzpicture}
	\caption{Example of a lattice path interpretation of \eqref{LGV1} with $i=5$, $j=3$, $p=8$ and $q=6$. The steps contribute the factor $-u v w^2 (u X_1+v X_1^{-1})^2 (u X_3 + v X_3^{-1})$ to the weight of the path. In the second region, the path is even.}
	\label{fig:SinglePathFirstInterpretation}
\end{figure}

Applying Lemma~\ref{LGV} shows that $\prod_{i=1}^{n} X_i^{n-1} \det_{1 \le i, j \le n} (a_{i,j})$ is the generating function of such lattice paths that start in $(-1,1),(-2,2),\ldots,(-n,n)$ and end in  $(0,1),(1,0),\ldots,(n-1,-n+2)$, and are nonintersecting in the first and in the third region. In the second region, we can distinguish between odd and even paths, depending on whether they only reach odd or even lattice points, respectively. In this region, Lemma~\ref{LGV} implies that the family of even paths is nonintersecting as well as the family of odd paths. This concludes the proof of Theorem~\ref{interpret1}. 

Our final goal in this section is to relate the nonintersecting lattice paths in Theorem~\ref{interpret1} to cyclically and vertically symmetric lozenge tilings of a hexagon with side lengths $2n+2, 2n, 2n+2, 2n, 2n+2, 2n$ and with a central triangular hole of size $2$, which are known to be equinumerous with 
$(2n+1) \times (2n+1)$ VSASMs.

Due to the symmetries of these tilings, we can decompose the hexagon into six fundamental areas, which can be seen in Figure~\ref{fig:ExampleRhombusTiling}. Figure~\ref{fig:FundamentalArea} shows such a fundamental area~$\mathcal{H}_n$ for $n=5$. Thus, it suffices to consider lozenge tilings of $\mathcal{H}_n$ which are in bijective correspondence with families of nonintersecting lattice paths with the step set $\{(-1,0),(0,1)\}$ starting in $A_i=(2i-1,i-1)$ and ending in $E_i=(i-1,2i-2)$ for $1 \leq i \leq n$. An example of this correspondence is illustrated in Figure~\ref{fig:ExampleRhombusTiling}.

\begin{figure}[htb]
	\centering
	\begin{tikzpicture}[scale=.45]
		
		\foreach \x in {-4,...,5}
		\foreach \y in {0,...,10}
		\coordinate (\x/\y) at ($ (60:\y) + (\x,0) $);
		
		
		\draw (0/0) -- (1/0) -- (1/1) -- (2/1) -- (2/2) -- (3/2) -- (3/3) -- (4/3) -- (4/4) -- (5/4) -- (5/5) -- (1/9) -- (-4/9) -- (-4/8) -- (-3/7) -- (-3/6) -- (-2/5) -- (-2/4) -- (-1/3) -- (-1/2) -- (0/1) -- cycle;
		
		\node at (-2/10) {$n$};
		\node at (5/7) {$n-1$};
		
		
		\draw (0/1) -- (2/1);
		\draw (-1/2) -- (3/2);	
		\draw (-1/3) -- (4/3);	
		\draw (-2/4) -- (5/4);	
		\draw (-2/5) -- (5/5);	
		\draw (-3/6) -- (4/6);	
		\draw (-3/7) -- (3/7);	
		\draw (-4/8) -- (2/8);	
		\draw (-4/9) -- (1/9);

		
		\draw (0/1) -- (1/0);
		\draw (-1/3) -- (1/1);
		\draw (-2/5) -- (2/1);
		\draw (-3/7) -- (2/2);
		\draw (-4/9) -- (3/2);
		\draw (-3/9) -- (3/3);
		\draw (-2/9) -- (4/3);
		\draw (-1/9) -- (4/4);
		\draw (0/9) -- (5/4);

		
		\draw (-3/7) -- (-3/9);
		\draw (-2/5) -- (-2/9);
		\draw (-1/3) -- (-1/9);
		\draw (0/1) -- (0/9);
		\draw (1/1) -- (1/9);
		\draw (2/2) -- (2/8);
		\draw (3/3) -- (3/7);
		\draw (4/4) -- (4/6);

	\end{tikzpicture}
	\caption{Fundamental area $\mathcal{H}_n$}
	\label{fig:FundamentalArea}
\end{figure}

The number of families of nonintersecting lattice paths between the two sets $(A_i)_{1 \leq i \leq n}$ and $(E_i)_{1 \leq i \leq n}$ can be calculated by means of the Lindstr\"om--Gessel--Viennot lemma (Lemma~\ref{LGV}):
\begin{equation}
	\label{unrefined}
	\det_{1 \le i,j \le n} \mathcal{P} (A_i \to E_j) = \det_{1 \le i,j \le n} \left( \binom{i+j-1}{2j-i-1} \right).
\end{equation}

Noting that 
\begin{equation}
	\label{hspecial}  
	\left. h_{k}(X_1,\ldots,X_n) \right|_{(X_1,\ldots,X_n)=X} = X^k \binom{n+k-1}{k},
\end{equation} 
we observe that, when setting $u=v=1, w=-1$ and $X_i=1$, for $i=1,2,\ldots,n$, then \eqref{LGV1} simplifies to 
\begin{equation}
	\label{specialized} 
	\sum_{p \ge 1} (-1)^{2i-p} \binom{i}{2i-p}  \sum_{q \ge 1,2 \mid (p-q)}  2^{q-j} (-1)^{(p-q)/2} \binom{(p+q)/2-1}{(p-q)/2} \binom{q-1}{j-1}. 
\end{equation}
We see that
\begin{equation}\label{eq:binomial_identity}
	\sum_{p \ge 1} (-1)^{2i-p} \binom{i}{2i-p}  \sum_{q \ge 1,2 \mid (p-q)}  2^{q-j} (-1)^{(p-q)/2} \binom{(p+q)/2-1}{(p-q)/2} \binom{q-1}{j-1}
	=  \binom{i+j-1}{2j-i+1},
\end{equation}
which equals the entries of the matrix underlying the determinant in \eqref{unrefined}.

In Section~\ref{ap:bijective_proofs}, we provide a combinatorial proof of \eqref{eq:binomial_identity} in two steps within our model of lattice paths.

\section{\texorpdfstring{Proof of Theorem~\ref{interpret4}}{Proof of Theorem~2.2}}
\label{sec:third} 

In this section, we see that the signed families of lattice paths from Theorem~\ref{interpret1} can be transformed, in a purely combinatorial matter, into a signless interpretation in terms of certain plane partitions. Thus, Theorem~\ref{interpret1} implies Theorem~\ref{interpret4}. In Appendix~\ref{ap:computational_proof}, we present a second proof exploiting another Jacobi--Trudi-type expression for \eqref{bialternant} which differs from the one used in Section~\ref{sec:first}.

A drawback of Theorem~\ref{interpret1} is that the lattice path interpretation also involves signs although the generating function has only non-negative  coefficients, since it is also the generating function of arrowed monotone triangles. We will define a sign-reversing involution on the families of lattice paths from Theorem~\ref{interpret1} to see directly that also their generating function has no negative coefficients. In fact, we will arrive again at the signless interpretation in Theorem~\ref{interpret4}.
To be more precise, we will first have to ``move back'' to possibly intersecting paths in the second and third region, and apply the sign-reversing involution on such paths. This sign-reversing involution will also involve a Gessel--Viennot-type sign-reversing involution in the third region. 

For this purpose, we will first show that \eqref{LGV1} has only non-negative coefficients. In fact, we define a sign-reversing involution that shows that the inner sum of \eqref{LGV1} has no negative coefficients.

Recall that we consider lattice paths from $(0,p)$ to 
$(j-1,-j+2)$, with step set $\{(1,-1),\allowbreak(0,-2)\}$ until we reach the line $y=1$ at $(q-1,1)$ such that $q \le p$, while on and below the line $y=1$ the step set is 
$\{(-1,0),(0,-1)\}$. We then equip steps of type $(0,-2)$ with the weight $-uv$, while steps of type $(-1,0)$ are colored in blue or red, and 
blue and red steps are equipped with the weight $u X_d$ and $v X_d^{-1}$, respectively, where $d$ is the distance from the line $y=2$. All other steps have weight $1$.

\begin{lemma} 
\label{signlessfirst} 
Let $p,j$ be positive integers, then 
\begin{equation*}
\sum_{q \ge 1, 2 \mid (p-q)} 
(-u v)^{(p-q)/2} \binom{(p+q)/2-1}{(p-q)/2}  
h_{q-j}(u X_1+ v X_1^{-1},\ldots,u X_j+ v X_j^{-1})
\end{equation*}
is the generating function of lattice paths from $(0,p)$ to 
$(j-1,-j+2)$ as given above, but without steps of type $(0,-2)$ and without consecutive pairs of horizontal steps with the first step being blue and the second step being red. 
\end{lemma} 

Clearly, forbidding steps of type $(0,-2)$ implies that above the line $y=1$ we go straight from $(0,p)$ to 
$(p-1,1)$, using only steps of type $(1,-1)$. 

\begin{proof}
We define the following sign-reversing involution on the lattice paths from $(0,p)$ to $(j-1,-j+2)$ that do not fall into the class as described in the lemma: 
Let $r$ be the number of $(1,-1)$-steps at the beginning of the lattice path (so that the next step is either $(0,-2)$ or we have already 
reached the line $y=1$). Since there are in total $q-1$ diagonal steps, we know that $r \in \{0,1,\ldots,q-1\}$. We also consider the first 
$r$ steps of the path after the first point on the line $y=1$. Such $r$ steps exist since there are precisely $q-1$ steps in the region $\{(x,y) \mid y \le 1\}$. They are all in the step set $\{(-1,0),(0,-1)\}$.  

We distinguish between two cases. If among the latter $r$ steps, there are two consecutive horizontal steps such that the first is blue and 
the second is red (when traversing the path from $(p,0)$ to $(j-1,-j+2)$), we choose the first pair of such horizontal steps. We delete both of them, and, assuming they were the steps at position $s$ and $s+1$ now counted from the first point on the line $y=1$, 
we replace the diagonal steps at position $s$ and 
$s+1$, counted from the start, by a single $(0,-2)$-step. We adjust the path so that it starts in $(p,0)$ and ends in 
$(j-1,-j+2)$. Note that, by this replacement, the first intersection point of the path with the line $y=1$ is shifted from $(q-1,1)$ to $(q-3,1)$.

If there is no such pair, then the path does not fall into the class as described in the lemma if and only if $r<q-1$, and, consequently, there is a $(0,-2)$-step 
right after the $r$-th diagonal step. In this case, we replace the $(0,-2)$-step by two steps of type $(1,-1)$ and insert a blue horizontal step 
at position $r+1$ and a red horizontal step at position $r+2$, counted from the first point on the line $y=1$. In this case, the first intersection point  with the line $y=1$ is moved from $(q-1,1)$ to $(q+1,1)$. 

This mapping is an involution and clearly sign-reversing: Two consecutive horizontal steps such that a blue step is followed by a red step have in total the weight $u X_d v X_d^{-1} = u v$, and, in a sense, they are replaced by a $(0,-2)$-step, which has weight $-u v$, or vice versa.
\end{proof}

Using Lemma~\ref{signlessfirst}, we are left with families of paths with step sets as described in Theorem~\ref{interpret1}  that 
\begin{itemize} 
\item are nonintersecting in the region $\{(x,y) \mid x \le 0\}$, 
\item have only diagonal steps $(1,-1)$ in the region $\{(x,y) \mid x \ge 0, y \ge 1\}$ (which implies automatically that the paths are nonintersecting in this region), and 
\item consecutive horizontals steps in the region $\{(x,y) \mid x \ge 0, y \le 1\}$ are divided into an initial section of red steps (which may be empty) followed by an ending section of blue steps (which may be empty as well). The paths may still be intersecting in this region at this point.
\end{itemize} 

The lattice point of a maximal section of consecutive horizontal steps that divides the section into two portions of red and blue steps is said to be the \emph{center} of the section. Note that the center can also be the left or right end point of these maximal sections if there are no blue or no red steps, respectively. Now we say that two paths in the region $\{(x,y) \mid  x \ge 0 , y \le 1\}$ touch if they intersect, but their intersections are fully contained in maximal sections of consecutive horizontal steps and none of these intersections contain centers of either paths. Two paths that are intersecting but not touching are said to \emph{intersect strongly}. 

We make the following important observation: Suppose $P_1,P_2$ are two touching paths, and let 
$(-i_1,i_1),(-i_2,i_2)$ be their starting points, respectively, and $(j_1-1,-j_1+2),(j_2-1,-j_2+2)$ their end points. Then $i_1 \le i_2$ if and only if 
$j_1 \le j_2$. This implies that for families where no two paths intersect strongly, the path starting in 
$(-i,i)$ ends in $(i-1,-i+2)$, so that the underlying permutation is the identity, which does not contribute to the sign.

Now we define a sign-reversing involution on such families of paths with at least one pair of paths that is 
intersecting strongly. Among those pairs we chose a canonical one as follows: The starting point of one path $P_1$ should be as rightmost as possible, and among all paths that are intersecting strongly with this path, choose again the path $P_2$ whose starting point is as rightmost as possible. 

We distinguish several cases: If $P_1 \cap P_2$ contains a vertical step, then we select the last such vertical step in $P_1$ and interchange the end portions of the two paths after this step. Thus we can assume that $P_1 \cap P_2$ contains only horizontal steps. Then the involution is provided by Figure~\ref{cases}, which is applied to the last connected component of the intersection. In all cases, the sign of the family of paths changes by a factor of $-1$.
An example of families of paths that do not cancel under this involution is given in Figure~\ref{fig:signless}.

\begin{figure}[htb]
Case~1:
\begin{tikzpicture}[scale=.45,baseline=(current bounding box.center)]
		
		\draw [help lines,step=1cm,dashed] (-0.75,-0.75) grid (13.75,4.75);
		
		\draw[ultra thick] (0,0) --++ (0,2.1) --++ (1,0);
		\draw[ultra thick,myblue] (0,2.1) --++ (1,0) --++ (1,0) --++ (1,0) --++ (1,0) --++ (1,0) --++ 
		(1,0);
		
		\draw[ultra thick, myred] (6,2.1) --++ (1,0) --++ (1,0) --++ (0,1);
		\draw[ultra thick] (8,2.1) --++ (0,1.9);
		
		\draw[ultra thick] (2,0) --++ (0,1.9) --++ (1,0);
		\draw[ultra thick,myblue] (2,1.9) --++ (1,0) --++ (1,0) --++ (1,0) --++ (1,0) --++ (1,0) --++ 
		(1,0) --++ (1,0) --++ (1,0);
		
		\draw[ultra thick, myred] (10,1.9) --++ (1,0) --++ (1,0) --++ (1,0) --++ (0,1);
		\draw[ultra thick] (13,1.9) --++ (0,2.1);
		
\end{tikzpicture} $\longleftrightarrow$
\begin{tikzpicture}[scale=.45,baseline=(current bounding box.center)]
		
		\draw [help lines,step=1cm,dashed] (-0.75,-0.75) grid (13.75,4.75);

		\draw[ultra thick] (0,0) --++ (0,2.1) --++ (1,0);
		\draw[ultra thick,myblue] (0,2.1) --++ (1,0) --++ (1,0) --++ (1,0) --++ (1,0) --++ (1,0) --++ 
		(1,0) --++ (1,0) --++ (1,0) --++ (1,0) --++ (1,0);
		
		\draw[ultra thick, myred] (10,2.1) --++ (1,0) --++ (1,0) --++ (1,0) --++ (0,1);
		\draw[ultra thick] (13,2.1) --++ (0,1.9);
		
		\draw[ultra thick] (2,0) --++ (0,1.9) --++ (1,0);
		\draw[ultra thick,myblue] (2,1.9) --++ (1,0) --++ 
		(1,0) --++ (1,0) --++ (1,0);
		
		\draw[ultra thick, myred] (6,1.9) --++ (1,0) --++ (1,0) --++ (0,1);
		\draw[ultra thick] (8,2.1) --++ (0,1.9);
		
\end{tikzpicture}

\bigskip

Case~2:
\begin{tikzpicture}[scale=.45,baseline=(current bounding box.center)]
		
		\draw [help lines,step=1cm,dashed] (-0.75,-0.75) grid (13.75,4.75);

		\draw[ultra thick] (0,0) --++ (0,2.1) --++ (1,0);
		\draw[ultra thick,myblue] (0,2.1) --++ (1,0) --++ (1,0) --++ (1,0) --++ (1,0) --++ (1,0) --++ 
		(1,0);
		
		\draw[ultra thick, myred] (6,2.1) --++ (1,0) --++ (1,0) --++ (0,1);
		\draw[ultra thick] (8,2.1) --++ (0,1.9);
		
		\draw[ultra thick] (2,0) --++ (0,1.9) --++ (1,0);
		\draw[ultra thick,myblue] (2,1.9) --++ (1,0) --++ (1,0) --++ 
		(1,0) --++ (1,0) --++ (1,0);
		
		\draw[ultra thick, myred] (7,1.9) --++ (1,0) --++ (1,0) --++ (1,0) --++ (1,0) --++ (1,0) --++ (1,0) --++ (0,1);
		\draw[ultra thick] (13,1.9) --++ (0,2.1);
		
\end{tikzpicture} $\longleftrightarrow$
\begin{tikzpicture}[scale=.45,baseline=(current bounding box.center)]
		
		\draw [help lines,step=1cm,dashed] (-0.75,-0.75) grid (13.75,4.75);

		\draw[ultra thick] (0,0) --++ (0,2.1) --++ (1,0);
		\draw[ultra thick,myblue] (0,2.1) --++ (1,0) --++ (1,0) --++ (1,0) --++ (1,0) --++ (1,0) --++ 
		(1,0);
		
		\draw[ultra thick, myred] (6,2.1) --++ (1,0) --++ (1,0) --++ (1,0) --++ (1,0) --++ (1,0) --++ (1,0) --++ (1,0) --++ (0,1);
		\draw[ultra thick] (13,2.1) --++ (0,1.9);
		
		\draw[ultra thick] (2,0) --++ (0,1.9) --++ (1,0);
		\draw[ultra thick,myblue] (2,1.9) --++ (1,0) --++ (1,0) --++ 
		(1,0) --++ (1,0) --++ (1,0);
		
		\draw[ultra thick, myred] (7,1.9) --++ (1,0) --++ (0,1);
		\draw[ultra thick] (8,1.9) --++ (0,2.1);
		
\end{tikzpicture}

\bigskip

Case~3:
\begin{tikzpicture}[scale=.45,baseline=(current bounding box.center)]
		
		\draw [help lines,step=1cm,dashed] (-0.75,-0.75) grid (13.75,4.75);

		\draw[ultra thick] (0,0) --++ (0,2.1) --++ (1,0);
		\draw[ultra thick,myblue] (0,2.1) --++ 
		(1,0);
		
		\draw[ultra thick, myred] (1,2.1) --++ (1,0) --++ (1,0)--++ (1,0) --++ (1,0) --++ (1,0) --++ (1,0) --++ (1,0) --++ (0,1);
		\draw[ultra thick] (8,2.1) --++ (0,1.9);
		
		\draw[ultra thick] (2,0) --++ (0,1.9) --++ (1,0);
		\draw[ultra thick,myblue] (2,1.9) --++ (1,0) --++ (1,0) --++ 
		(1,0) --++ (1,0) --++ (1,0);
		
		\draw[ultra thick, myred] (7,1.9) --++ (1,0) --++ (1,0) --++ (1,0) --++ (1,0) --++ (1,0) --++ (1,0) --++ (0,1);
		\draw[ultra thick] (13,1.9) --++ (0,2.1);
		
\end{tikzpicture} $\longleftrightarrow$
\begin{tikzpicture}[scale=.45,baseline=(current bounding box.center)]
		
		\draw [help lines,step=1cm,dashed] (-0.75,-0.75) grid (13.75,4.75);

		\draw[ultra thick] (0,0) --++ (0,2.1) --++ (1,0);
		\draw[ultra thick,myblue] (0,2.1) --++ (1,0);
		
		\draw[ultra thick, myred] (1,2.1) --++ (1,0) --++ (1,0) --++ (1,0) --++ (1,0) --++ (1,0)  --++ 
		(1,0) --++ (1,0) --++ (1,0) --++ (1,0) --++ (1,0) --++ (1,0) --++ (1,0) --++ (0,1);
		\draw[ultra thick] (13,2.1) --++ (0,1.9);
		
		\draw[ultra thick] (2,0) --++ (0,1.9) --++ (1,0);
		\draw[ultra thick,myblue] (2,1.9) --++ (1,0) --++ (1,0) --++ 
		(1,0) --++ (1,0) --++ (1,0);
		
		\draw[ultra thick, myred] (7,1.9) --++ (1,0) --++ (0,1);
		\draw[ultra thick] (8,1.9) --++ (0,2.1);
		
\end{tikzpicture}
\caption{The different cases of the sign-reversing involution on strongly intersecting paths.}
\label{cases}
\end{figure} 

\begin{figure}[htb]
	\centering
	\begin{tikzpicture}[scale=.45,baseline=(current bounding box.center)]
		\fill [verylight-gray] (-6.75,-4.75) rectangle (0,10.75);
		\fill [light-gray] (0,-4.75) rectangle (10.75,1);
		
		\draw [help lines,step=1cm,dashed] (-6.75,-4.75) grid (10.75,10.75);
		
		\draw[->,thick] (-6.75,0)--(10.75,0) node[right]{$x$};
		\draw[->,thick] (0,-4.75)--(0,10.75) node[above]{$y$};
		
		\fill (-1,1) circle (5pt);
		\fill (-2,2) circle (5pt);
		\fill (-3,3) circle (5pt);
		\fill (-4,4) circle (5pt);
		\fill (-5,5) circle (5pt);
		\fill (-6,6) circle (5pt);
		
		\fill (0,1) circle (5pt);
		\fill (1,0) circle (5pt);
		\fill (2,-1) circle (5pt);
		\fill (3,-2) circle (5pt);
		\fill (4,-3) circle (5pt);
		\fill (5,-4) circle (5pt);
		
		\path[decoration=arrows, decorate] (-1,1) --++ (1,1) --++ (1,-1) --++ (-1,0);
		
		\draw[ultra thick,dashed,myblue] (1,1) --++ (-1,0);
		
		\path[decoration=arrows, decorate] (-2,2) --++ (1,1) --++ (1,1) --++ (1,-1) --++ (1,-1) --++ (1,-1) --++ (-1,0) --++ (0,-1) --++ (-1,0);
		
		\draw[ultra thick,dashed,myblue] (3,1) --++ (-1,0);
		
		\draw[ultra thick,dotted,myred] (2,0) --++ (-1,0);
		
		\path[decoration=arrows, decorate] (-3,3) --++ (1,1) --++ (1,0) --++ (1,1) --++ (1,-1) --++ (1,-1) --++ (1,-1) --++ (1,-1) --++ (0,-1) 
		--++ (-1,0) --++ (-1,0) --++ (0,-1);
		
		\draw[ultra thick,dashed,myblue] (4,0) --++ (-1,0) --++ (-1,0);
		
		\path[decoration=arrows, decorate] (-4,4) --++ (1,1) --++ (1,0) --++ (1,1) --++ (1,1) --++ (1,-1) --++ (1,-1) --++ (1,-1) --++ (1,-1) --++    (1,-1) --++ (1,-1) --++ (0,-1) --++ (-1,0) --++ (0,-1) --++ (-1,0) --++ (-1,0) --++ (0,-1);
		
		\draw[ultra thick,dotted,myred] (6,0) --++ (-1,0);
		\draw[ultra thick,dotted,myred] (5,-1) --++ (-1,0);
		\draw[ultra thick,dashed,myblue] (4,-1) --++ (-1,0);
		
		\path[decoration=arrows, decorate] (-5,5) --++ (1,1) --++ (1,1) --++ (1,1) --++ (1,0) --++ (1,1) --++ (1,-1) --++ (1,-1) --++ (1,-1) --++ (1,-1) --++ (1,-1) --++ (1,-1) --++ (1,-1) --++ (1,-1) --++ (0,-1) --++ (-1,0) --++ (0,-1) --++ (-1,0) --++ (-1,0) --++ (0,-1) --++ (-1,0) --++ (0,-1);
		
		\draw[ultra thick,dashed,myblue] (8,0) --++ (-1,0);
		\draw[ultra thick,dashed,myblue] (7,-1) --++ (-1,0) --++ (-1,0);
		\draw[ultra thick,dashed,myblue] (5,-2) --++ (-1,0);
		
		\path[decoration=arrows, decorate] (-6,6) --++ (1,1) --++ (1,1) --++ (1,1) --++ (1,0) --++ (1,1) --++ (1,0) --++ (1,-1) --++ (1,-1) --++ (1,-1) --++ (1,-1) --++ (1,-1) --++ (1,-1) --++ (1,-1) --++ (1,-1) --++ (1,-1) --++ (0,-1) --++ (0,-1) --++ (-1,0) --++ (0,-1) --++ (-1,0) --++ (-1,0) --++ (0,-1) --++ (-1,0) --++ (0,-1);
		
		\draw[ultra thick,dotted,myred] (9,-1) --++ (-1,0);
		\draw[ultra thick,dotted,myred] (8,-2) --++ (-1,0) --++ (-1,0);
		\draw[ultra thick,dotted,myred] (6,-3) --++ (-1,0);
		
	\end{tikzpicture}
	\caption{\label{fig:signless} This family of lattice paths corresponds to the pair of \CSPP and \RSPP in Section~\ref{sec:MainResults}.}
\end{figure}

Such paths correspond to pairs $(P,Q')$ of \CSPPs of the same shape with at most $n$ rows with the following properties:
\begin{itemize}
	\item The entries of $P$ in the $i$-th row from the bottom are no greater than $2i$.
	\item The entries of $Q'$ in the $i$-th row from the bottom are no greater than $2i-1$ and no less than $\ell + i -1$ 
	if $\ell$ is the length of the $i$-th row.
\end{itemize} 
The entries of $P$ correspond to the heights of the horizontal steps of the paths on and below the line $y=1$. Each odd entry $2i-1$ corresponds to a red step at distance~$i$ from the line $y=2$ and contributes $u X_i$ multiplicatively to the weight, while each even entry $2i$ corresponds to a blue step at distance~$i$ from the line $y=2$ and contributes $v X_i^{-1}$.  The entries of $Q'$ correspond to the diagonal steps, more precisely their distance from the line $y=x+1$, and each of them contributes $w$ to the weight.

Now we subtract $i-1$ from all entries in the $i$-th row from the bottom of $Q'$ and then $j-1$ from all entries in the $j$-th column of $Q'$. We obtain an \RSPP $Q$ with positive entries such that the entries in the $i$-th row from the bottom are no greater than $i$. This establishes the proof of Theorem~\ref{interpret4}.

\section{\texorpdfstring{Bijective proof of Equation~\eqref{eq:binomial_identity}}{Bijective proof of Equation~(4.8)}}\label{ap:bijective_proofs}

In this section, we present a combinatorial proof of \eqref{eq:binomial_identity}:
\begin{equation*}
	\sum_{p \ge 1} (-1)^{2i-p} \binom{i}{2i-p}  \sum_{q \ge 1,2 \mid (p-q)}  2^{q-j} (-1)^{(p-q)/2} \binom{(p+q)/2-1}{(p-q)/2} \binom{q-1}{j-1}
	=  \binom{i+j-1}{2j-i+1}.
\end{equation*}
In particular, we provide combinatorial proofs within our model of lattice paths for the following two identities:
\begin{equation}
	\label{sum1} 
	\sum_{q \ge 1, 2 \mid (p-q)}  2^{q-j} (-1)^{(p-q)/2} \binom{(p+q)/2-1}{(p-q)/2} \binom{q-1}{j-1} = \binom{p+j-1}{p-j}
\end{equation} 
and
\begin{equation}
	\label{sum2} 
	\sum_{p \ge 1} (-1)^{2i-p} \binom{i}{2i-p} \binom{p+j-1}{p-j} = \binom{i+j-1}{2j-i+1}.
\end{equation}
Interestingly, we have to leave the regime of nonintersecting lattice paths for the proof of \eqref{sum2} and have to work with lattice paths that are possibly intersecting instead. This was not the case in \cite{nASMDPP} and it is an open problem whether we can avoid intersecting lattice paths here as well.
\subsection{\texorpdfstring{Combinatorial proof of \eqref{sum1}}{Combinatorial proof of (6.1)}} 
\label{innersection} 

Recall that in our combinatorial model in Section~\ref{sec:first}, the left-hand side of \eqref{sum1} is interpreted as 
\begin{itemize} 
	\item lattice paths from $(0,p)$ to $(j-1,-j+2)$, 
	\item with step set $\{(1,-1),(0,-2)\}$ until we reach the line $y=1$ (at $(q-1,1)$), 
	\item while on and below the line $y=1$, the step set is $\{(-1,0),(0,-1)\}$. 
\end{itemize} 
Steps of type $(0,-2)$ are equipped with the weight~$-1$, while steps of type $(-1,0)$ are equipped with the weight~$2$. We take the latter weight into account by coloring such steps either in blue or in red. All steps different from these steps are colored in green to distinguish them from the uncolored paths that will provide an interpretation for the right-hand side of \eqref{sum1}.
The set of these colored paths is denoted by $A(j,p)$, and we refer to them as the $A$-paths.

The combinatorial model for the right-hand side of \eqref{sum1} are paths from $(0,p)$ to $(2j-1,-j+1)$ with step set 
$\{(1,-1),(0,-1)\}$. 
The set of these uncolored paths is denoted by $B(j,p)$, and we refer to them as the $B$-paths.

{\bf From $\boldsymbol{B(j,p)}$ to $\boldsymbol{A(j,p)}$:} We explain how to transform paths from $B(j,p)$ into paths from $A(j,p)$. It is an inductive process. In each step, we will have a lattice path that starts in $(0,p)$. The initial section will be colored in green (using the step set $\{(1,-1),(0,-2)\}$ of the first region of $A$-paths), while the ending section will be colored in green, red and blue (using the step set $\{(0,-1),(-1,0)\}$ of the second region of $A$-paths). For the middle section, we use the uncolored step set $\{(1,-1),(0,-1)\}$ of $B$-paths.  

When traveling along the path starting at $(0,p)$, we look for the first pair of uncolored steps $(s_1,s_2)$ and perform the following transformations, see also Figure~\ref{fig:RulesBtoA}:

\begin{enumerate}
	\item If $(s_1,s_2)=((1,-1),(1,-1))$, then we color $s_1$ green, contract $s_2$ and add a green $(0,-1)$-step at the end. \label{list:BAfirst}
	\item If $(s_1,s_2)=((1,-1),(0,-1))$, then we color $s_1$ green, replace $s_2$ by an uncolored $(1,-1)$-step and add a blue $(-1,0)$-step at the end. \label{list:BAsecond}
	\item If $(s_1,s_2)=((0,-1),(1,-1))$, then we replace $s_1$ by a green $(1,-1)$-step and add a red $(-1,0)$-step at the end. \label{list:BAthird}
	\item If $(s_1,s_2)=((0,-1),(0,-1))$ we make three copies, two of them with the same sign and a third with the opposite sign. In the first two, we replace $s_1$ by a green $(1,-1)$-step. In both cases, we add a $(-1,0)$-step at the end; in one case it is blue and in the other case it is red. In the copy with the opposite sign, we replace $(s_1,s_2)$ by a $(0,-2)$-step. The sign is taken into account, since $(0,-2)$-steps have sign $-1$. \label{list:BAfourth}
\end{enumerate} 

\begin{figure}[htb]
	\centering
	\begin{minipage}{.33\textwidth}
		\eqref{list:BAfirst}\;\begin{tikzpicture}[scale=.45,baseline=(current bounding box.center)]
			\draw [help lines,step=1cm,dashed] (-.75,-.75) grid (2.75,2.75);
			\path[decoration=arrows, decorate] (0,2) --++ (1,-1) --++ (1,-1);
		\end{tikzpicture} $\mapsto$
		\begin{tikzpicture}[scale=.45,baseline=(current bounding box.center)]
			\draw [help lines,step=1cm,dashed] (-.75,-.75) grid (3.75,4.75);
			\draw[ultra thick, mygreen, ->] (0,4) --++ (1,-1);
			\draw[thick] (1,3) .. controls (2,2) and (3,3) .. (3,2) .. controls (3,1) and (2,2) .. (2,1);
			\draw[ultra thick, mygreen, ->] (2,1) --++ (0,-1);
		\end{tikzpicture}
	\end{minipage}\hfill\begin{minipage}{.33\textwidth}
		\eqref{list:BAsecond}\;\begin{tikzpicture}[scale=.45,baseline=(current bounding box.center)]
			\draw [help lines,step=1cm,dashed] (-.75,-.75) grid (1.75,2.75);
			\path[decoration=arrows, decorate] (0,2) --++ (1,-1) --++ (0,-1);
		\end{tikzpicture} $\mapsto$
		\begin{tikzpicture}[scale=.45,baseline=(current bounding box.center)]
			\draw [help lines,step=1cm,dashed] (-1.75,.25) grid (3.75,5.75);
			\draw[ultra thick, mygreen, ->] (-1,5) --++ (1,-1);
			\draw[ultra thick, ->] (0,4) --++ (1,-1);
			\draw[thick] (1,3) .. controls (2,2) and (3,3) .. (3,2) .. controls (3,1) and (2,2) .. (2,1);
			\draw[ultra thick, myblue, ->] (2,1) --++ (-1,0);
		\end{tikzpicture}
	\end{minipage}\hfill\begin{minipage}{.33\textwidth}
		\eqref{list:BAthird}\;\begin{tikzpicture}[scale=.45,baseline=(current bounding box.center)]
			\draw [help lines,step=1cm,dashed] (-.75,-.75) grid (2.75,2.75);
			\path[decoration=arrows, decorate] (0,2) --++ (0,-1) --++ (1,-1);
		\end{tikzpicture} $\mapsto$
		\begin{tikzpicture}[scale=.45,baseline=(current bounding box.center)]
			\draw [help lines,step=1cm,dashed] (-1.75,.25) grid (3.75,5.75);
			\draw[ultra thick, mygreen, ->] (-1,5) --++ (1,-1);
			\draw[ultra thick, ->] (0,4) --++ (1,-1);
			\draw[thick] (1,3) .. controls (2,2) and (3,3) .. (3,2) .. controls (3,1) and (2,2) .. (2,1);
			\draw[ultra thick, myred, ->] (2,1) --++ (-1,0);
		\end{tikzpicture}
	\end{minipage}\\\bigskip
	\eqref{list:BAfourth}\;\begin{tikzpicture}[scale=.45,baseline=(current bounding box.center)]
		\draw [help lines,step=1cm,dashed] (-.75,-.75) grid (.75,2.75);
		\path[decoration=arrows, decorate] (0,2) --++ (0,-1) --++ (0,-1);
	\end{tikzpicture} $\mapsto$
	\begin{tikzpicture}[scale=.45,baseline=(current bounding box.center)]
		\draw [help lines,step=1cm,dashed] (-.75,.25) grid (3.75,5.75);
		\draw[ultra thick, mygreen, ->] (0,5) --++ (1,-1);
		\draw[ultra thick, ->] (1,4) --++ (0,-1);
		\draw[thick] (1,3) .. controls (2,2) and (3,3) .. (3,2) .. controls (3,1) and (2,2) .. (2,1);
		\draw[ultra thick, myblue, ->] (2,1) --++ (-1,0);
	\end{tikzpicture}$+$
	\begin{tikzpicture}[scale=.45,baseline=(current bounding box.center)]
		\draw [help lines,step=1cm,dashed] (-.75,.25) grid (3.75,5.75);
		\draw[ultra thick, mygreen, ->] (0,5) --++ (1,-1);
		\draw[ultra thick, ->] (1,4) --++ (0,-1);
		\draw[thick] (1,3) .. controls (2,2) and (3,3) .. (3,2) .. controls (3,1) and (2,2) .. (2,1);
		\draw[ultra thick, myred, ->] (2,1) --++ (-1,0);
	\end{tikzpicture}$-$
	\begin{tikzpicture}[scale=.45,baseline=(current bounding box.center)]
		\draw [help lines,step=1cm,dashed] (-.75,-.75) grid (.75,2.75);
		\draw[ultra thick, mygreen, ->] (0,2) --++ (0,-2);
	\end{tikzpicture}
	\caption{Rules for transforming an uncolored $B$-path into a colored $A$-path.}
	\label{fig:RulesBtoA}
\end{figure}

If there is only one uncolored step left, delete it. This step is always a step of type $(1,-1)$ as can be seen as follows: Paths in $B(j,p)$ have $2j-1$ steps of this type, and only Rule~\eqref{list:BAfirst} transforms such steps. Among these $2j-1$ steps, there are $j-1$ pairs of such steps that are transformed using Rule~\eqref{list:BAfirst}, so that there is one such step left in the end. This also shows that we are always in the situation that there is one uncolored step left in the end.

{\bf The procedure is well-defined:} From the rules, it can also be seen that uncolored $(0,-1)$-steps are either transformed into green $(1,-1)$-steps or pairs of them are transformed into green $(0,-2)$-steps. Suppose that there are $2t$ of the second type so that there are $p-j-2t$ of the first type. It follows that there are $p-2t-1$ green steps of type $(1,-1)$ and $t$ green steps of type $(0,-2)$. With these steps we reach indeed the line $y=1$, since
\[
(0,p) + (p-2t-1) \cdot (1,-1) + t \cdot (0,-2) =  (p-2t-1,1). 
\]
Note that we set $t$ such that $p-2t$ corresponds to  $q$ in \eqref{sum1}.

Since none of the rules changes the height of the endpoint of our path except for the deletion of the final uncolored $(1,-1)$-step, the end point of our path lies on the line $y=-j+2$. We create (blue or red) steps of type $(-1,0)$ precisely when uncolored $(0,-1)$ steps are transformed, which happens in Rules~\eqref{list:BAsecond}, \eqref{list:BAthird}, and \eqref{list:BAfourth} using either of the first two options. There are $p-j-2t$ such applications so that the end point lies on the line $x=p-2t-1-(p-j-2t)=j-1$. Therefore, the end point is $(j-1,-j+2)$. An example is given in Figure~\ref{fig:ExampleBtoA}.

\begin{figure}[htb]
	\begin{multline*}
		\begin{tikzpicture}[scale=.45,baseline=(current bounding box.center)]
			\draw [help lines,step=1cm,dashed] (-.75,-1.75) grid (3.75,5.75);
			\path[decoration=arrows, decorate] (0,5) --++ (1,-1) --++ (0,-1) --++ (1,-1) --++ (0,-1) --++ (0,-1) --++ (1,-1);
		\end{tikzpicture}\overset{\eqref{list:BAsecond}}{\longmapsto}
		\begin{tikzpicture}[scale=.45,baseline=(current bounding box.center)]
			\draw [help lines,step=1cm,dashed] (-.75,-1.75) grid (4.75,5.75);
			\draw[ultra thick, ->, mygreen] (0,5) --++ (1,-1);
			\path[decoration=arrows, decorate] (1,4) --++ (1,-1) --++ (1,-1) --++ (0,-1) --++ (0,-1) --++ (1,-1);
			\draw[ultra thick, ->, myblue] (4,-1) --++ (-1,0);
		\end{tikzpicture}\overset{\eqref{list:BAfirst}}{\longmapsto}
		\begin{tikzpicture}[scale=.45,baseline=(current bounding box.center)]
			\draw [help lines,step=1cm,dashed] (-.75,-1.75) grid (3.75,5.75);
			\draw[ultra thick, ->, mygreen] (0,5) --++ (1,-1);
			\draw[ultra thick, ->, mygreen] (1,4) --++ (1,-1);
			\path[decoration=arrows, decorate] (2,3) --++ (0,-1) --++ (0,-1) --++ (1,-1);
			\draw[ultra thick, ->, myblue] (3,0) --++ (-1,0);
			\draw[ultra thick, ->, mygreen] (2,0) --++ (0,-1);
		\end{tikzpicture}\\
		\overset{\eqref{list:BAfourth}}{\longmapsto}
		\left\{ \begin{array}{l}
			\begin{tikzpicture}[scale=.45,baseline=(current bounding box.center)]
				\draw [help lines,step=1cm,dashed] (-.75,-1.75) grid (4.75,5.75);
				\draw[ultra thick, ->, mygreen] (0,5) --++ (1,-1);
				\draw[ultra thick, ->, mygreen] (1,4) --++ (1,-1);
				\draw[ultra thick, ->, mygreen] (2,3) --++ (1,-1);
				\path[decoration=arrows, decorate] (3,2) --++ (0,-1) --++ (1,-1);
				\draw[ultra thick, ->, myblue] (4,0) --++ (-1,0);
				\draw[ultra thick, ->, mygreen] (3,0) --++ (0,-1);
				\draw[ultra thick, ->, myblue] (3,-1) --++ (-1,0);
			\end{tikzpicture}\overset{\eqref{list:BAthird}}{\longmapsto}
			\begin{tikzpicture}[scale=.45,baseline=(current bounding box.center)]
				\draw [help lines,step=1cm,dashed] (-.75,-1.75) grid (5.75,5.75);
				\draw[ultra thick, ->, mygreen] (0,5) --++ (1,-1);
				\draw[ultra thick, ->, mygreen] (1,4) --++ (1,-1);
				\draw[ultra thick, ->, mygreen] (2,3) --++ (1,-1);
				\draw[ultra thick, ->, mygreen] (3,2) --++ (1,-1);
				\path[decoration=arrows, decorate] (4,1) --++ (1,-1);
				\draw[ultra thick, ->, myblue] (5,0) --++ (-1,0);
				\draw[ultra thick, ->, mygreen] (4,0) --++ (0,-1);
				\draw[ultra thick, ->, myblue] (4,-1) --++ (-1,0);
				\draw[ultra thick, ->, myred] (3,-1) --++ (-1,0);
			\end{tikzpicture}\longmapsto
			\begin{tikzpicture}[scale=.45,baseline=(current bounding box.center)]
				\draw [help lines,step=1cm,dashed] (-.75,-.75) grid (4.75,5.75);
				\draw[ultra thick, ->, mygreen] (0,5) --++ (1,-1);
				\draw[ultra thick, ->, mygreen] (1,4) --++ (1,-1);
				\draw[ultra thick, ->, mygreen] (2,3) --++ (1,-1);
				\draw[ultra thick, ->, mygreen] (3,2) --++ (1,-1);
				\draw[ultra thick, ->, myblue] (4,1) --++ (-1,0);
				\draw[ultra thick, ->, mygreen] (3,1) --++ (0,-1);
				\draw[ultra thick, ->, myblue] (3,0) --++ (-1,0);
				\draw[ultra thick, ->, myred] (2,0) --++ (-1,0);
			\end{tikzpicture}\\
			\begin{tikzpicture}[scale=.45,baseline=(current bounding box.center)]
				\draw [help lines,step=1cm,dashed] (-.75,-1.75) grid (4.75,5.75);
				\draw[ultra thick, ->, mygreen] (0,5) --++ (1,-1);
				\draw[ultra thick, ->, mygreen] (1,4) --++ (1,-1);
				\draw[ultra thick, ->, mygreen] (2,3) --++ (1,-1);
				\path[decoration=arrows, decorate] (3,2) --++ (0,-1) --++ (1,-1);
				\draw[ultra thick, ->, myblue] (4,0) --++ (-1,0);
				\draw[ultra thick, ->, mygreen] (3,0) --++ (0,-1);
				\draw[ultra thick, ->, myred] (3,-1) --++ (-1,0);
			\end{tikzpicture}\overset{\eqref{list:BAthird}}{\longmapsto}
			\begin{tikzpicture}[scale=.45,baseline=(current bounding box.center)]
				\draw [help lines,step=1cm,dashed] (-.75,-1.75) grid (5.75,5.75);
				\draw[ultra thick, ->, mygreen] (0,5) --++ (1,-1);
				\draw[ultra thick, ->, mygreen] (1,4) --++ (1,-1);
				\draw[ultra thick, ->, mygreen] (2,3) --++ (1,-1);
				\draw[ultra thick, ->, mygreen] (3,2) --++ (1,-1);
				\path[decoration=arrows, decorate] (4,1) --++ (1,-1);
				\draw[ultra thick, ->, myblue] (5,0) --++ (-1,0);
				\draw[ultra thick, ->, mygreen] (4,0) --++ (0,-1);
				\draw[ultra thick, ->, myred] (4,-1) --++ (-1,0);
				\draw[ultra thick, ->, myred] (3,-1) --++ (-1,0);
			\end{tikzpicture}\longmapsto
			\begin{tikzpicture}[scale=.45,baseline=(current bounding box.center)]
				\draw [help lines,step=1cm,dashed] (-.75,-.75) grid (4.75,5.75);
				\draw[ultra thick, ->, mygreen] (0,5) --++ (1,-1);
				\draw[ultra thick, ->, mygreen] (1,4) --++ (1,-1);
				\draw[ultra thick, ->, mygreen] (2,3) --++ (1,-1);
				\draw[ultra thick, ->, mygreen] (3,2) --++ (1,-1);
				\draw[ultra thick, ->, myblue] (4,1) --++ (-1,0);
				\draw[ultra thick, ->, mygreen] (3,1) --++ (0,-1);
				\draw[ultra thick, ->, myred] (3,0) --++ (-1,0);
				\draw[ultra thick, ->, myred] (2,0) --++ (-1,0);
			\end{tikzpicture}\\
			\begin{tikzpicture}[scale=.45,baseline=(current bounding box.center)]
				\draw [help lines,step=1cm,dashed] (-.75,-1.75) grid (3.75,5.75);
				\draw[ultra thick, ->, mygreen] (0,5) --++ (1,-1);
				\draw[ultra thick, ->, mygreen] (1,4) --++ (1,-1);
				\draw[ultra thick, ->, mygreen] (2,3) --++ (0,-2); 
				\path[decoration=arrows, decorate] (2,1) --++ (1,-1);
				\draw[ultra thick, ->, myblue] (3,0) --++ (-1,0);
				\draw[ultra thick, ->, mygreen] (2,0) --++ (0,-1);
			\end{tikzpicture}\longmapsto
			\begin{tikzpicture}[scale=.45,baseline=(current bounding box.center)]
				\draw [help lines,step=1cm,dashed] (-.75,-.75) grid (2.75,5.75);
				\draw[ultra thick, ->, mygreen] (0,5) --++ (1,-1);
				\draw[ultra thick, ->, mygreen] (1,4) --++ (1,-1);
				\draw[ultra thick, ->, mygreen] (2,3) --++ (0,-2); 
				\draw[ultra thick, ->, myblue] (2,1) --++ (-1,0);
				\draw[ultra thick, ->, mygreen] (1,1) --++ (0,-1);
			\end{tikzpicture}
		\end{array}\right.
	\end{multline*}
	\caption{The transformation of a $B$-path into a family of $A$-paths for $p=5$ and $j=2$.}
	\label{fig:ExampleBtoA}
\end{figure}

{\bf From $\boldsymbol{A(j,p)}$ to $\boldsymbol{B(j,p)}$:}
Conversely, paths from $A(j,p)$ are transformed into paths from $B(j,p)$ as follows: After the first lattice point on the line $y=1$ when coming from $(0,p)$, we add an uncolored step of type $(1,-1)$, so that the new end point is $(j,-j+1)$. From now on, it is again an inductive procedure. In each step, we have a path starting in $(0,p)$. The starting section 
uses the step set of $A(j,p)$ for above the line $y=1$, while the ending section uses the step set of $A(j,p)$ for below the line $y=1$. The middle part of the path uses the step set of $B(j,p)$. The rules are as follows (note that they reverse the procedure of how to transform a $B$-path into an $A$-path):

\begin{enumerate}
	\item If the last step of the beginning section is a $(0,-2)$-step, we replace it by two uncolored $(0,-1)$-steps. \label{list:ABfirst}
	\item If the last step of the beginning section is a $(1,-1)$-step and the last step of the path is a green $(0,-1)$-step, we delete this last step and replace the $(1,-1)$-step by two uncolored $(1,-1)$-steps. \label{list:ABsecond}
	\item If the last step of the beginning section is a $(1,-1)$-step and the last step of the path is a $(-1,0)$ step, we do the following: 
	\begin{enumerate}
		\item If the $(1,-1)$-step is followed by an uncolored $(0,-1)$ step, we replace the $(1,-1)$-step by an uncolored 
		$(0,-1)$ step and delete the last step of the path.\label{list:ABthirda}
		\item If the $(1,-1)$-step is followed by an uncolored $(1,-1)$-step, we distinguish according to the color of the last step (which is by assumption a $(-1,0)$ step):
		\begin{enumerate} \item If the step is red, we replace the green $(1,-1)$-step by an uncolored step of type $(0,-1)$, and delete also the red step. \label{list:ABthirdb1}
			\item If the step is blue, we replace the uncolored $(1,-1)$-step by an uncolored $(0,-1)$-step, uncolor the green $(1,-1)$-step, and delete the blue step. \label{list:ABthirdb2}
		\end{enumerate}
	\end{enumerate} 
\end{enumerate} 
An example is provided in Figure~\ref{fig:ExampleAtoB}.

\begin{figure}[htb]
	\centering
	\begin{tikzpicture}[scale=.45,baseline=(current bounding box.center)]
		\draw [help lines,step=1cm,dashed] (-.75,-1.75) grid (4.75,7.75);
		
		\draw[ultra thick, mygreen, ->] (0,7) --++ (1,-1);
		\draw[ultra thick, mygreen, ->] (1,6) --++ (0,-2);
		\draw[ultra thick, mygreen, ->] (1,4) --++ (1,-1);
		\draw[ultra thick, mygreen, ->] (2,3) --++ (1,-1);
		\draw[ultra thick, mygreen, ->] (3,2) --++ (1,-1);
		\draw[ultra thick, myblue, ->] (4,1) --++ (-1,0);
		\draw[ultra thick, mygreen, ->] (3,1) --++ (0,-1);
		\draw[ultra thick, myred, ->] (3,0) --++ (-1,0);
		\draw[ultra thick, myblue, ->] (2,0) --++ (-1,0);
	\end{tikzpicture}$\longmapsto$
	\begin{tikzpicture}[scale=.45,baseline=(current bounding box.center)]
		\draw [help lines,step=1cm,dashed] (-.75,-1.75) grid (5.75,7.75);
		\draw[ultra thick, mygreen, ->] (0,7) --++ (1,-1);
		\draw[ultra thick, mygreen, ->] (1,6) --++ (0,-2);
		\draw[ultra thick, mygreen, ->] (1,4) --++ (1,-1);
		\draw[ultra thick, mygreen, ->] (2,3) --++ (1,-1);
		\draw[ultra thick, mygreen, ->] (3,2) --++ (1,-1);
		\path[decoration=arrows, decorate] (4,1) --++ (1,-1);
		\draw[ultra thick, myblue, ->] (5,0) --++ (-1,0);
		\draw[ultra thick, mygreen, ->] (4,0) --++ (0,-1);
		\draw[ultra thick, myred, ->] (4,-1) --++ (-1,0);
		\draw[ultra thick, myblue, ->] (3,-1) --++ (-1,0);
	\end{tikzpicture}$\overset{\eqref{list:ABthirdb2}}{\longmapsto}$
	\begin{tikzpicture}[scale=.45,baseline=(current bounding box.center)]
		\draw [help lines,step=1cm,dashed] (-.75,-1.75) grid (4.75,7.75);
		\draw[ultra thick, mygreen, ->] (0,7) --++ (1,-1);
		\draw[ultra thick, mygreen, ->] (1,6) --++ (0,-2);
		\draw[ultra thick, mygreen, ->] (1,4) --++ (1,-1);
		\draw[ultra thick, mygreen, ->] (2,3) --++ (1,-1);
		\path[decoration=arrows, decorate] (3,2) --++ (1,-1) --++ (0,-1);
		\draw[ultra thick, myblue, ->] (4,0) --++ (-1,0);
		\draw[ultra thick, mygreen, ->] (3,0) --++ (0,-1);
		\draw[ultra thick, myred, ->] (3,-1) --++ (-1,0);
	\end{tikzpicture}$\overset{\eqref{list:ABthirdb1}}{\longmapsto}$
	\begin{tikzpicture}[scale=.45,baseline=(current bounding box.center)]
		\draw [help lines,step=1cm,dashed] (-.75,-1.75) grid (4.75,7.75);
		\draw[ultra thick, mygreen, ->] (0,7) --++ (1,-1);
		\draw[ultra thick, mygreen, ->] (1,6) --++ (0,-2);
		\draw[ultra thick, mygreen, ->] (1,4) --++ (1,-1);
		\path[decoration=arrows, decorate] (2,3) --++ (0,-1) --++ (1,-1) --++ (0,-1);
		\draw[ultra thick, myblue, ->] (3,0) --++ (-1,0);
		\draw[ultra thick, mygreen, ->] (2,0) --++ (0,-1);
	\end{tikzpicture}
	$\overset{\eqref{list:ABsecond}}{\longmapsto}$\begin{tikzpicture}[scale=.45,baseline=(current bounding box.center)]
		\draw [help lines,step=1cm,dashed] (-.75,-1.75) grid (4.75,7.75);
		\draw[ultra thick, mygreen, ->] (0,7) --++ (1,-1);
		\draw[ultra thick, mygreen, ->] (1,6) --++ (0,-2);
		\path[decoration=arrows, decorate] (1,4) --++ (1,-1) --++ (1,-1) --++ (0,-1)  --++ (1,-1) --++ (0,-1);
		\draw[ultra thick, myblue, ->] (4,-1) --++ (-1,0);
	\end{tikzpicture}$\overset{\eqref{list:ABfirst}}{\longmapsto}$
	\begin{tikzpicture}[scale=.45,baseline=(current bounding box.center)]
		\draw [help lines,step=1cm,dashed] (-.75,-1.75) grid (4.75,7.75);
		\draw[ultra thick, mygreen, ->] (0,7) --++ (1,-1);
		\path[decoration=arrows, decorate] (1,6) --++ (0,-1) --++ (0,-1) --++ (1,-1) --++ (1,-1) --++ (0,-1)  --++ (1,-1) --++ (0,-1);
		\draw[ultra thick, myblue, ->] (4,-1) --++ (-1,0);
	\end{tikzpicture}$\overset{\eqref{list:ABthirda}}{\longmapsto}$
	\begin{tikzpicture}[scale=.45,baseline=(current bounding box.center)]
		\draw [help lines,step=1cm,dashed] (-.75,-1.75) grid (3.75,7.75);
		
		\path[decoration=arrows, decorate] (0,7) --++ (0,-1) --++ (0,-1) --++ (0,-1) --++ (1,-1) --++ (1,-1) --++ (0,-1) --++ (1,-1) --++ (0,-1);
	\end{tikzpicture}
	\caption{The transformation of an $A$-path into a $B$-path for $p=7$ and $j=2$.}
	\label{fig:ExampleAtoB}
\end{figure}

{\bf The procedure is well-defined:}
It remains to show the following:
\begin{itemize}
	\item If the last step of the beginning section is of type $(1,-1)$ then the ending section is not empty.
	\item The end point of the final path is $(2j-1,-j+1)$. 
\end{itemize}  

To show the first assertion, note that in each step the number of green $(1,-1)$-steps is just the number of blue or red $(-1,0)$-step plus the number of green $(0,-1)$-steps. This is obviously true for paths in $A(j,p)$ and remains to be true in every step by the nature of the rules: Whenever a green $(1,-1)$-step is transformed, we delete precisely one step of the other type. 

Concerning the change of the end point, note that the end point does not change, except when we are in case~\eqref{list:ABsecond}, in which case the end point is pushed horizontally by one unit to the right. We have such a case~\eqref{list:ABsecond} for each green $(0,-1)$-step in the original path in $A(j,p)$. There are $j-1$ such steps, so that the end point is 
$(j,-j+1)+(j-1) \cdot (1,0)=(2j-1,-j+1)$ as required.

For this purpose, we will first show that \eqref{LGV1} has only non-negative coefficients. In fact, our combinatorial proof 
of \eqref{sum1} implies  a sign-reversing involution that shows that the inner sum of \eqref{LGV1} has no negative coefficients. To see this, we introduce essentially the weights from Theorem~\ref{interpret1} in the combinatorial interpretation of the left-hand side of \eqref{sum1} to obtain 
a combinatorial interpretation of the inner sum of \eqref{LGV1} 

\begin{remark}
	The combinatorial proof of \eqref{sum1} implies a sign-reversing involution that shows that the inner sum of \eqref{sum1} has no negative coefficients. The proof of Lemma~\ref{signlessfirst} provides this involution with more general weights. In fact, two $A$-paths that are paired off under this involution are mapped to the same $B$-path, but have opposite signs as $A$-paths. 
\end{remark}

\subsection{\texorpdfstring{Combinatorial proof of \eqref{sum2}}{Combinatorial proof of (6.2)}}
\label{outersection}

Recall that we interpret the left-hand side of \eqref{sum2} as paths from $(2j-1,-j+1)$ to $(-i,i)$. For $x \ge 0$, 
we allow steps of type $(-1,1)$ and $(0,1)$, and, for $x < 0$, we allow steps of type $(-1,-1)$ and $(-1,0)$. Steps of type $(-1,0)$ are equipped with the weight $-1$. Such a path must contain precisely one lattice point, say $(0,p)$, 
on the line $x=0$. We reflect the portion of the path that is to the left of the $y$-axis along this axis. 

This results in paths from $(2j-1,-j+1)$ to $(i,i)$ that touch the line $x=0$ precisely once, namely in $(0,p)$. 
From  $(2j-1,-j+1)$ to $(0,p)$, we allow steps of type $(-1,1)$ and $(0,1)$, and from $(0,p)$ to $(i,i)$, we allow steps of type $(1,-1)$ and $(1,0)$. Every step of type $(1,0)$ is equipped with the weight~$-1$. An example is provided in Figure~\ref{fig:ExamplePathForSignReversing}.

\begin{figure}[htb]
	\centering
	\begin{tikzpicture}[scale=.45,baseline=(current bounding box.center)]
		\draw [help lines,step=1cm,dashed] (0,-2.75) grid (5.75,7.75);
		
		\draw[->,thick] (0,0)--(5.75,0) node[right]{$x$};
		\draw[->,thick] (0,-2.75)--(0,7.75) node[above]{$y$};
		
		\fill (5,-2) circle (5pt);
		\fill (0,6) circle (5pt);
		\fill (5,5) circle (5pt);
		
		\path[decoration=arrows, decorate] (5,-2) --++ (-1,1) --++ (0,1) --++ (0,1) --++ (-1,1) --++ (-1,1) --++ (-1,1) --++ (0,1) --++ (-1,1) --++ (1,0) --++ (1,0) --++ (1,-1) --++ (1,0) --++ (1,0);
	\end{tikzpicture}
	
	\caption{Example of the lattice path interpretation of the left-hand side of \eqref{sum2} for $i=5$, $j=3$ and $p=6$. The weight of the path is $(-1)^4=1$.}
	\label{fig:ExamplePathForSignReversing}
\end{figure}

Next, we construct a sign-reversing involution on a subset of these paths. For this purpose, observe that the point $(0,p)$ divides the path into two subpaths, an upper path, that ends in $(i,i)$, and a lower path, that starts in $(2j-1,-j+1)$. Seen as lattice paths, these two paths might coincide in some diagonal 
steps (in the initial section of the upper path and the ending section of the lower path): Let $k\in\{0,\dots,i\}$ be the largest integer such that the vertices $(0,p)$, $(1,p-1)$, \dots, $(k,p-k)$ lie both on the upper and the lower path. The point $P=(k,p-k)$ is then the rightmost intersection point of these two paths. If $k < i$, we perform a transformation on the paths which we depict as follows:

\begin{center}
	\begin{tikzpicture}[scale=.45,baseline=(current bounding box.center)]
		\draw [help lines,step=1cm,dashed] (0.25,0.25) grid (3.75,4.75);
		
		\fill (2,3) circle (5pt);
		\node[above right] at (2,3) {\contour{white}{$P$}};
		
		\draw[ultra thick] (1,4) -- (2,3) -- (3,2);	
		\draw[ultra thick] (2,3) -- (2,2);	
		\draw[ultra thick,dashed] (2,1) -- (2,2) -- (3,1);	
	\end{tikzpicture}
	$\quad\longleftrightarrow\quad$
	\begin{tikzpicture}[scale=.45,baseline=(current bounding box.center)]
		\draw [help lines,step=1cm,dashed] (0.25,0.25) grid (3.75,4.75);
		
		\fill (2,2) circle (5pt);
		\node[above right] at (2,2) {\contour{white}{$P$}};
		
		\draw[ultra thick] (1,3) -- (2,2) -- (3,2);	
		\draw[ultra thick,dashed] (2,1) -- (2,2) -- (3,1);	
	\end{tikzpicture}
\end{center}

Depending on whether the upper path has a diagonal or a horizontal step after it has reached $P$, we delete the vertical step of the lower path incident to $P$ and replace the diagonal step of the upper 
path by the horizontal step
or insert a vertical step to the lower path and replace the horizontal step of the upper path by a diagonal step, respectively. This transformation changes the height of the intersection point~$P$ and therefore alters the point where the path touches the line $x=0$. The upper path either gains or loses a horizontal step which changes the weight of the path. Figure~\ref{fig:SignReversing} illustrates one instance of this sign-reversing involution for $i=4$ and $j=3$.

\begin{figure}[htb]
	\centering
	\begin{tikzpicture}[scale=.45,baseline=(current bounding box.center)]
		\draw [help lines,step=1cm,dashed] (0,-2.75) grid (5.75,7.75);
		
		\draw[->,thick] (0,0)--(5.75,0) node[right]{$x$};
		\draw[->,thick] (0,-2.75)--(0,7.75) node[above]{$y$};
		
		\fill (5,-2) circle (5pt);
		\fill (0,7) circle (5pt);
		\fill (4,4) circle (5pt);
		\fill (2,5) circle (5pt);
		\node[above right] at (2,5) {\contour{white}{$P$}};
		\node[left] at (0,7) {\contour{white}{$p+1$}};
		
		\draw[ultra thick] (5,-2) --++ (0,1) --++ (-2,2) --++ (0,2) --++ (-1,1) --++ (0,1) --++ (-2,2);
		\draw[ultra thick] (2,5) --++ (1,-1) --++ (1,0);
	\end{tikzpicture}
	$\quad\longleftrightarrow\quad$
	\begin{tikzpicture}[scale=.45,baseline=(current bounding box.center)]
		\draw [help lines,step=1cm,dashed] (0,-2.75) grid (5.75,7.75);
		
		\draw[->,thick] (0,0)--(5.75,0) node[right]{$x$};
		\draw[->,thick] (0,-2.75)--(0,7.75) node[above]{$y$};
		
		\fill (5,-2) circle (5pt);
		\fill (0,6) circle (5pt);
		\fill (4,4) circle (5pt);
		\fill (2,4) circle (5pt);
		\node[above right] at (2,4) {\contour{white}{$P$}};
		\node[left] at (0,6) {\contour{white}{$p$}};
		
		\draw[ultra thick] (5,-2) --++ (0,1) --++ (-2,2) --++ (0,2) --++ (-3,3);
		\draw[ultra thick] (2,4) --++ (2,0);
	\end{tikzpicture}
	
	\caption{Sign-reversing involution for $i=4$ and $j=3$.}
	\label{fig:SignReversing}
\end{figure}

We leave the path unchanged only if the rightmost intersection point~$P$ is $(i,i)$. In this case, the path touches the line $x=0$ in the point $(0,2i)$ and the upper path as well as the lower path coincide in the $i$ diagonal steps from $(0,2i)$ to $(i,i)$. We delete those diagonal steps, which leaves us paths from $(2j-1,-j+1)$ to $(i,i)$ with steps of type $(-1,1)$ and $(0,1)$. These paths are enumerated by $\binom{i+j-1}{2j-i+1}$.

\section{Further types of families of lattice paths}
\label{sec:further_lattice_paths}
Theorem~\ref{interpret1} presents a family of lattice paths with the same generating function as arrowed monotone triangles with bottom row $0,2,\ldots,2n-2$. In this section, we provide two other families of lattice paths with the same property, the second of which comprises even nonintersecting lattice paths.

\subsection{The second type of families of lattice paths} 

\begin{theorem} 
	\label{interpret2} 
	For $n \ge 1$, the generating function of arrowed monotone triangles with bottom row 
	$0,2,\ldots,2n-2$ is equal to the signed generating function of $n$ lattice paths with starting points
	$(1,1),(2,2),\ldots,(n,n)$ and end points $(2,1),(2,0),\ldots,(2,-n+2)$ as follows:
	\begin{itemize} 
		\item In the region $\{(x,y) \mid y \ge 1\}$, the step set is $\{(1,0),(0,-1)\}$ and horizontal steps with distance $d$ from the $x$-axis are equipped with the weight $u X_d + w + v X_d^{-1}$. 
		\item In the region $\{(x,y) \mid y \le 1\}$, the step set is $\{(-1,-1),(-2,-2),(-2,-1)\}$. Steps of type $(-1,-1)$ are equipped with the weight $-w$, while steps of type $(-2,-2)$ are equipped with the weight $-u v$.  
	\end{itemize} 
	The paths are nonintersecting in the region $\{(x,y) \mid y \ge 1\}$. In the region $\{(x,y) \mid y \le 1\}$, there may be intersections, however, no two steps of different paths may have a common end point. 
	
	The weight of a family is $\prod_{i=1}^{n} X_i^{n-1}$ multiplied by the product of the weights of all its steps where the weight of a step is $1$ if it has not been specified. 
	Let $\sigma$ be the permutation so that the $i$-th starting point is connected to the $\sigma(i)$-th end point, then the sign of the family is $\sgn \sigma$.
\end{theorem}

\begin{figure}[htb]
	\centering
	\begin{tikzpicture}[scale=.45,baseline=(current bounding box.center)]
		\fill [light-gray] (0,-4.75) rectangle (8.75,1);
		
		\draw [help lines,step=1cm,dashed] (0,-4.75) grid (8.75,6.75);
		
		\draw[->,thick] (0,0)--(8.75,0) node[right]{$x$};
		\draw[->,thick] (0,-4.75)--(0,6.75) node[above]{$y$};
		
		\fill (2,1) circle (5pt);
		\fill (2,0) circle (5pt);
		\fill (2,-1) circle (5pt);
		\fill (2,-2) circle (5pt);
		\fill (2,-3) circle (5pt);
		\fill (2,-4) circle (5pt);
		
		\fill (1,1) circle (5pt);
		\fill (2,2) circle (5pt);
		\fill (3,3) circle (5pt);
		\fill (4,4) circle (5pt);
		\fill (5,5) circle (5pt);
		\fill (6,6) circle (5pt);

		\path[decoration=arrows, decorate] (1,1) --++ (1,0);
		
		\path[decoration=arrows, decorate] (2,2) --++ (1,0) --++ (0,-1) --++ (-1,-1);
		
		\path[decoration=arrows, decorate] (3,3) --++ (1,0) --++ (0,-1) --++ (1,0) --++ (0,-1) --++ (-2,-2) --++ (-1,-1);
		
		\path[decoration=arrows, decorate] (4,4) --++ (1,0) --++ (0,-1) --++ (1,0) --++ (0,-1) --++ (0,-1) --++ (-2,-1) --++ (-2,-1);
		
		\path[decoration=arrows, decorate]  (5,5) --++ (1,0) --++ (1,0) --++ (0,-1) --++ (0,-1) --++ (0,-1) --++ (0,-1) --++ (-1,-1);
		\draw[ultra thick,->]  (5.95,0.05) --++ (-2,-2);
		\path[decoration=arrows, decorate]  (4,-2) --++ (-2,-1);
		
		\path[decoration=arrows, decorate]  (6,6) --++ (1,0) --++ (1,0) --++ (0,-1) --++ (0,-1) --++ (0,-1) --++ (0,-1) --++ (0,-1) --++ (-1,-1) --++ (-2,-1);
		\draw[ultra thick,->] (5.05,-1.05) --++ (-2,-2);
		\path[decoration=arrows, decorate] (3,-3) --++ (-1,-1);
		
		\draw[ultra thick,dashed,myred] (1,1) --++ (1,0);
		
		\draw[ultra thick,dotted,mygreen] (2,2) --++ (1,0) --++ (0,-1) --++ (-1,-1);
		
		\draw[ultra thick,dashed,myred] (3,3) --++ (1,0) --++ (0,-1) --++ (1,0) --++ (0,-1) --++ (-2,-2) --++ (-1,-1);
		
		\draw[ultra thick,dotted,mygreen] (4,4) --++ (1,0) --++ (0,-1) --++ (1,0) --++ (0,-1) --++ (0,-1) --++ (-2,-1) --++ (-2,-1);
		
		\draw[ultra thick,dashed,myred]  (5,5) --++ (1,0) --++ (1,0) --++ (0,-1) --++ (0,-1) --++ (0,-1) --++ (0,-1) --++ (-1,-1);
		\draw[ultra thick,dashed,myred]  (5.95,0.05) --++ (-2,-2);
		\draw[ultra thick,dashed,myred]  (4,-2) --++ (-2,-1);
		
		\draw[ultra thick,dotted,mygreen]  (6,6) --++ (1,0) --++ (1,0) --++ (0,-1) --++ (0,-1) --++ (0,-1) --++ (0,-1) --++ (0,-1) --++ (-1,-1) --++ (-2,-1);
		\draw[ultra thick,dotted,mygreen] (5.05,-1.05) --++ (-2,-2);
		\draw[ultra thick,dotted,mygreen] (3,-3) --++ (-1,-1);

	\end{tikzpicture}
	
	\caption{Example of the lattice paths in Theorem~\ref{interpret2} for $n=6$. The paths are drawn in alternating colors for a better distinction.}
	\label{fig:ExampleSecondInterpretation}
\end{figure}

An example is provided in Figure~\ref{fig:ExampleSecondInterpretation}. Its weight is 
\begin{multline*}
	\sgn \sigma (-u v)^3 (-w)^5  X_1^5 X_2^5 X_3^5 X_4^5 X_5^5 X_6^5 (u X_1+ w+ v X_1^{-1}) (u X_2+ w+ v X_2^{-1})^2\\
	\times (u X_3+ w+ v X_3^{-1})^2 (u X_4+ w+ v X_4^{-1}) (u X_5+ w+ v X_5^{-1})^2 (u X_6 + v X_6^{-1})^2
\end{multline*}
with the associated permutation $\sigma = (1\;2\;4\;3\;5\;6)$.

Again we list all families of paths from Theorem~\ref{interpret2} for the case $n=2$ along with their weights up to the overall factor $X_1 X_2$.

\begin{center} 
	\begin{tikzpicture}[scale=.45,baseline=(current bounding box.center)]
		\fill [light-gray] (0,-0.75) rectangle (4.75,1);
		
		\draw [help lines,step=1cm,dashed] (0,-0.75) grid (4.75,3.75);
		
		\draw[->,thick] (0,0)--(4.75,0) node[right]{\tiny$x$};
		\draw[->,thick] (0,-0.75)--(0,3.75) node[above]{\tiny$y$};
		
		\fill (2,1) circle (5pt);
		\fill (2,0) circle (5pt);
		
		\fill (1,1) circle (5pt);
		\fill (2,2) circle (5pt);
		
		\path[decoration=arrows, decorate] (1,1) --++ (1,0);
		
		\path[decoration=arrows, decorate] (2,2) --++ (1,0) --++ (0,-1) --++ (-1,-1);

		\draw[ultra thick,dashed,myred] (1,1) --++ (1,0);
		
		\draw[ultra thick,dotted,mygreen] (2,2) --++ (1,0) --++ (0,-1) --++ (-1,-1);
		
		\draw (2,-0.75) node[below]{\tiny $- w (u X_1 + w + v X_1^{-1})(u X_2 + w + v X_2^{-1})$};
		
	\end{tikzpicture}
	\begin{tikzpicture}[scale=.45,baseline=(current bounding box.center)]
		\fill [light-gray] (0,-0.75) rectangle (4.75,1);
		
		\draw [help lines,step=1cm,dashed] (0,-0.75) grid (4.75,3.75);
		
		\draw[->,thick] (0,0)--(4.75,0) node[right]{\tiny$x$};
		\draw[->,thick] (0,-0.75)--(0,3.75) node[above]{\tiny$y$};
		
		\fill (2,1) circle (5pt);
		\fill (2,0) circle (5pt);
		
		\fill (1,1) circle (5pt);
		\fill (2,2) circle (5pt);
		
		\path[decoration=arrows, decorate] (1,1) --++ (1,0);
		
		\path[decoration=arrows, decorate] (2,2) --++ (1,0) --++ (0,-1) --++ (1,0) --++ (-2,-1);

		\draw[ultra thick,dashed,myred] (1,1) --++ (1,0);
		
		\draw[ultra thick,dotted,mygreen] (2,2) --++ (1,0) --++ (0,-1) --++ (1,0) --++ (-2,-1);
		
		\draw (2,-0.75) node[below]{\tiny $(u X_1 + w + v X_1^{-1})^2(u X_2 + w + v X_2^{-1})$};
		
	\end{tikzpicture}
	\begin{tikzpicture}[scale=.45,baseline=(current bounding box.center)]
		\fill [light-gray] (0,-0.75) rectangle (4.75,1);
		
		\draw [help lines,step=1cm,dashed] (0,-0.75) grid (4.75,3.75);
		
		\draw[->,thick] (0,0)--(4.75,0) node[right]{\tiny$x$};
		\draw[->,thick] (0,-0.75)--(0,3.75) node[above]{\tiny$y$};
		
		\fill (2,1) circle (5pt);
		\fill (2,0) circle (5pt);
		
		\fill (1,1) circle (5pt);
		\fill (2,2) circle (5pt);
		
		\path[decoration=arrows, decorate] (1,1) --++ (1,0);
		
		\path[decoration=arrows, decorate] (2,2) --++ (1,0) --++ (1,0) --++ (0,-1) --++ (-2,-1);

		\draw[ultra thick,dashed,myred] (1,1) --++ (1,0);
		
		\draw[ultra thick,dotted,mygreen] (2,2) --++ (1,0) --++ (1,0) --++ (0,-1) --++ (-2,-1);
		
		\draw (2,-0.75) node[below]{\tiny $(u X_1 + w + v X_1^{-1})(u X_2 + w + v X_2^{-1})^2$};
	\end{tikzpicture}
\end{center}

\begin{proof}[Proof of Theorem~\ref{interpret2}]
	In the determinant in \eqref{bialternant}, we add $t$ times the $(n-j)$-th column to the $(n-j+1)$-st column, for $j=1,\ldots,n-1$, in this order. We repeat this for $j=1,\ldots,n-2$, then for 
	$j=1,\ldots,n-3$, etc. We obtain 
	\begin{multline*} 
		\prod_{i=1}^{n} X_i^{n-2} (u-v X_i^{-2})^{-1}\\
		\times
		\frac{\det\limits_{1 \le i,j \le n} \left( \left(u^2 X_i^2 + u w X_i+t \right)^{j-1} \left(u^2 X_i^2 + u w X_i \right) - \left(v^2 X_i^{-2} + v w X_i^{-1}+t\right)^{j-1} \left(v^2 X_i^{-2} + v w X_i^{-1}\right) \right)}{\prod\limits_{1 \le i < j \le n} (X_j - X_i)\left( u - v X_i^{-1} X_j^{-1}  \right)}.
	\end{multline*}
	Now, for $j=1$, we can replace $\left(u^2 X_i^2 + u w X_i \right)$ and $\left(v^2 X_i^{-2} + v w X_i^{-1}\right)$ by $\left(u^2 X_i^2 + u w X_i+t \right)$ and $\left(v^2 X_i^{-2} + v w X_i^{-1}+t\right)$, respectively, as the $t$ cancels in this case. Then we add $t$ times the $j$-th column to the 
	$(j+1)$-st column, for $j=1,2,\ldots,n-1$, in this order. The result is
	\begin{equation*}
	\prod_{i=1}^{n} X_i^{n-2} (u-v X_i^{-2})^{-1}  
	\frac{\det_{1 \le i,j \le n} \left( \left(u^2 X_i^2 + u w X_i+t \right)^{j}  - \left(v^2 X_i^{-2} + v w X_i^{-1}+t\right)^{j} \right)}{\prod_{1 \le i < j \le n} (X_j - X_i)\left( u - v X_i^{-1} X_j^{-1}  \right)}.
	\end{equation*}
	Setting $t=u v$, we obtain after some further modifications 
	\begin{equation*}
	\prod_{i=1}^{n} X_i^{n-1} 
	\frac{\det_{1 \le i,j \le n} \left( (u X_i + w + v X_i^{-1})^{j} 
		\frac{u^j X_i^j- v^j X_i^{-j}}{u X_i - v X_i^{-1}} \right)}{\prod_{1 \le i < j \le n} (X_j - X_i)\left( u - v X_i^{-1} X_j^{-1}  \right)}.
	\end{equation*}
	We will apply Lemma~\ref{lem:det} with $Y_i=u X_i +w + v X_i^{-1}$. We need to express the entries in the $i$-th row in terms of formal Laurent series in $Y_i$. By Lemma~\ref{basis}, this entry is 
	\begin{multline*} 
		(u X_i + w + v X_i^{-1})^{j} \sum_{r=0}^{(j-1)/2} (-uv)^r \binom{j-r-1}{r} (u X_i + v X_i^{-1})^{j-2r-1} \\
		=(u X_i + w + v X_i^{-1})^{j} \sum_{r=0}^{(j-1)/2} (-uv)^r \binom{j-r-1}{r} (u X_i + v X_i^{-1}+w-w)^{j-2r-1} 
		\\
		=\sum_{r=0}^{(j-1)/2} \sum_{s \ge 0} (-uv)^r (-w)^{j-2r-1-s} 
		\binom{j-r-1}{r} \binom{j-2r-1}{s} (u X_i +w + v X_i^{-1})^{j+s}.
	\end{multline*}  
	Using Lemma~\ref{lem:det} and setting $s=t-j-1$, the expression in \eqref{bialternant} is equal to  
	\begin{equation}
		\label{second}
		\prod_{i=1}^n X_i^{n-1} 
		\det_{1 \le i,j \le n} \left( \sum_{t\in\mathbb{Z}} 
		c_{t,j}(u,v,w) 
		h_{t-i}(u X_1 + w + v X_1^{-1},\ldots,u X_i + w + v X_i^{-1}) \right),
	\end{equation} 
	with $c_{t,j}(u,v,w)$ defined as
	\begin{multline} 
		\label{def_ctj}
		\sum_{r=0}^{(j-1)/2}  (-uv)^r (-w)^{2j-2r-t} 
		\binom{j-r-1}{r} \binom{j-2r-1}{t-j-1} \\
		= \sum_{r=0}^{(j-1)/2} (-uv)^r (-w)^{2j-2r-t} \binom{j-r-1}{r,2j-2r-t,t-j-1}.
	\end{multline}
	We claim that $c_{t,j}(u,v,w)$ is the generating function of lattice paths from $(0,0)$ to $(t-2,j-1)$ with 
	step set $\{(1,1),(2,2),(2,1)\}$ with respect to the weight 
	\begin{equation*}
	(-uv)^{\# (2,2)\text{-steps}}(-w)^{\# (1,1)\text{-steps}}.
	\end{equation*}
	Indeed, suppose that there are $r$ steps of type $(2,2)$, $x$ steps of type $(1,1)$ and $y$ steps of type $(2,1)$. Then 
	\begin{equation*}
	r \cdot (2,2) + x \cdot (1,1) + y \cdot (2,1) = (t-2,j-1),
	\end{equation*}
	which implies $x=2j-2r-t$ and $y=t-j-1$. Therefore, the number of such paths is indeed $\binom{j-r-1}{r,2j-2r-t,t-j-1}$. From Lemma~\ref{LGV}, it follows that the matrix underlying the determinant in \eqref{second} has the following combinatorial interpretation, see Figure~\ref{fig:ExampleSecondInterpretation}: We consider lattice paths from $(1,1),(2,2),\ldots,(n,n)$ to $(2,1),(2,0),\ldots,(2,-n+2)$ as follows:
	\begin{itemize} 
		\item In the region $\{(x,y) \mid y \ge 1\}$, the step set is $\{(1,0),(0,-1)\}$, and horizontal steps with distance $d$ from the $x$-axis are equipped with the weight $u X_d + w + v X_d^{-1}$. Suppose $(i,i)$ is the starting point and $(t,1)$ is last lattice point in this region, then 
		\begin{equation*}
		h_{t-i}(u X_1 + w + v X_1^{-1},\ldots,u X_i + w + v X_i^{-1})
		\end{equation*}
		is the generating function of such paths.
		\item In the region $\{(x,y) \mid y \le 1\}$, the step set is $\{(-1,-1),(-2,-2),(-2,-1)\}$. Steps of type $(-1,-1)$ are equipped with the weight $-w$, while steps of type $(-2,-2)$ are equipped with the weight $-u v$.  Suppose a path goes from 
		$(t,1)$ to $(2,-j+2)$, then the generating function of such paths is $c_{t,j}(u,v,w)$.
		\item Paths starting in $(i,i)$ are equipped with the additional weight $X_i^{n-1}$. 
	\end{itemize} 
	The paths are nonintersecting in the region $\{(x,y) \mid y \ge 1\}$. In the region $\{(x,y) \mid y \le 1\}$, there may be intersections, however, no two steps of different paths may have a common end point. This concludes the proof of Theorem~\ref{interpret2}. 
\end{proof}

\subsection*{Towards a signless version of Theorem~\ref{interpret2}?} 

It seems less clear how to show that the generating function of lattice paths from 
Theorem~\ref{interpret2} has no negative coefficients. It is tempting to proceed as in the case of the first interpretation since Lemma~\ref{signlessfirst} does have the following counterpart (Lemma~\ref{counterpart}). However, it 
is unclear what replaces the Gessel--Viennot involution from the previous section.

\begin{lemma} 
	\label{counterpart} 
	Let $i,j$ be positive integers, then 
	\begin{equation*}
	\sum_{t\in\mathbb{Z}} 
	c_{t,j}(u,v,w) 
	h_{t-i}(u X_1 + w + v X_1^{-1},\ldots,u X_i + w + v X_i^{-1})
	\end{equation*}
	is the generating function of lattice paths from $(i,i)$ to $(2j,1)$ using the step set $\{(1,0),\allowbreak (0,-1)\}$, with three types of $(1,0)$-steps, colored in red, blue and green, and with weights $u X_d$, $v X_d^{-1}$ and $w$, respectively, where $d$ is the distance from the $x$-axis and $c_{t,j}(u,v,w)$ is given as in \eqref{def_ctj}, and right of the line $y=x-j$ there is no green horizontal step and also no pair of consecutive 
	steps where the first is blue and the second is red.
\end{lemma} 

\begin{proof}
	Recall that the sum in the lemma is the generating function of lattice paths from $(i,i)$ to 
	$(2,-j+2)$ using the step set and weights as given in the lemma in the region 
	$\{(x,y) \mid y \ge 1\}$, while below $y=1$, the step set is $\{(-1,-1),(-2,-2),(-2,-1)\}$, where the steps of type $(-1,-1)$  and $(-2,-2)$ are equipped with the weight $-w$ and $-u v$, respectively. 
	
	We will construct a sign-reversing involution on such paths where one of the following is satisfied:
	\begin{itemize}
		\item The path contains a step of type $(-1,-1)$.
		\item The path contains a step of type $(-2,-2)$.
		\item Right of the line $y=x-j$, the path contains a green horizontal step.
		\item Right of the line $y=x-j$, the path contains a pair of consecutive horizontal steps where the first is blue and the second is red.
	\end{itemize} 
	A path that does not fall into that class has only steps of type $(-2,-1)$ below the line $y=1$, and, since its end point is $(2,-j+2)$, the last point on the line $y=1$ is $(2j,1)$. Consequently, we obtain a path as described in the lemma after deleting the (fixed) portion below the line $y=1$.  
	
	In order to construct the sign-reversing involution, we observe the following: Assuming that $(t,1)$ is the last point on the line $y=1$, the number of $(-2,-1)$-steps is $t-j-1$, while the total number of 
	steps from $(i,i)$ to $(t,1)$ is $t-1$. Let $r$ be the maximal number of consecutive $(-2,-1)$-steps right after 
	$(t,1)$. Now consider the $r$ steps before $(t,1)$ is reached. Since $r \le t-j-1 < t-1$, these steps exist.
	
	Among these, we consider occurrences of green horizontal steps and occurrences of pairs of consecutive horizontals steps  where the first is blue and the second is red, and, if there is such an occurrence, we consider the one that is closest to $(t,1)$.  Suppose it is a green horizontal step at position~$s$ before 
	$(t,1)$, then we delete this step and replace the $s$-th $(-2,-1)$-step after $(t,1)$ by a $(-1,-1)$-step. Clearly, this shifts the last point on the line $y=1$ from $(t,1)$ to $(t-1,1)$. If we have a pair of consecutive horizontal steps where the first is blue and the second is red, and they are in position $s,s+1$ before $(t,1)$, then we delete the blue and the red step and replace the two $(-2,-1)$-steps at positions $s$ and $s+1$ after $(t,1)$ by $(-2,-2)$. In this case, the last point on the line $y=1$ is shifted by $2$ to the left.
	
	If there is neither a green horizontal step nor a pair of consecutive horizontal steps where the first is blue and the second is red, then there has to be either a $(-1,-1)$-step or a $(-2,-2)$-step after the $r$-th 
	$(-2,-1)$-step. In the first case, we replace it by a $(-2,-1)$-step and insert a green horizontal step in 
	position $r+1$ before $(t,1)$. This shifts the last point on the line $y=1$ from $(t,1)$ to $(t+1,1)$. In the second case, we replace the $(-2,-2)$-step by two $(-2,-1)$-steps and insert a pair of consecutive horizontal steps where the first is blue and the second is red in positions $r+2$ and $r+1$ before $(t,1)$, which shifts the last point on the line $y=1$ from $(t,1)$ to $(t+2,1)$.
	
	This involution is clearly sign-reversing.\qedhere
	
\end{proof} 

\subsection{The third type of families of lattice paths} 
\label{third}

\begin{theorem} 
	\label{interpret3} 
	For $n \ge 1$, the generating function of arrowed monotone triangles with bottom row 
	$0,2,\ldots,2n-2$ is equal to the generating function of families of $n$ nonintersecting lattice paths with starting points in $A_i=\{(n-i+1,2n),(n+i-1,2n)\}$, $i=1,2,\ldots,n$, to $E_j=(j,-j+1)$, $j=1,2,\ldots,n$ with the following properties:
	\begin{itemize}
		\item  In the region $\{(x,y) \mid y \ge 1\}$, the step set is $\{(1,0),(0,-1)\}$. Horizontal steps at height $1,2,3,4,\ldots$ above the $x$-axis have weight 
		$u X_1, v X_1^{-1}, u X_2, v X_2^{-1},\ldots$, respectively.
		\item In the region $\{(x,y) \mid y \le 1\}$, the step set is $\{(-1,-1),(0,-1)\}$, where steps of type $(0,-1)$ are equipped with the weight $w$.
	\end{itemize} 
	Let $1 < i_1 < i_2 < \ldots < i_k \le n$ be precisely the indices for which we choose 
	$(n+i_k-1,2n) \in A_{i_k}$ as starting points (and not $(n-i_k+1,2n) \in A_{i_k}$). 
	The weight of a family is 
	\begin{equation*}
		\prod_{l=1}^k (-uv)^{i_l-1}  \prod_{i=1}^n X_i^{n-1}
	\end{equation*}
	multiplied by the product of the weights of all its steps where the 
	weight of a step is $1$ if it has not been specified. 
\end{theorem}

Note that $A_1=\{(n,2n)\}$ consists of a single point. Note as well that in this interpretation, the sign has been incorporated into the weight.
An example for the case $n=4$ is provided in Figure~\ref{fig:ExampleFirstOfThirdInterpretation}. For this interpretation, 
we omit to showcase the example for $n=2$ since 
it already involves $57$ configurations. 

\begin{figure}[htb]
	\centering
	\begin{tikzpicture}[scale=.45,baseline=(current bounding box.center)]
		\fill [light-gray] (0,-3.75) rectangle (7.75,1);
		
		\draw [help lines,step=1cm,dashed] (0,-3.75) grid (7.75,8.75);
		
		\draw[->,thick] (0,0)--(7.75,0) node[right]{$x$};
		\draw[->,thick] (0,-3.75)--(0,8.75) node[above]{$y$};
		
		\fill (1,8) circle (5pt);
		\fill (2,8) circle (5pt);
		\fill (3,8) circle (5pt);
		\fill (4,8) circle (5pt);
		\fill (5,8) circle (5pt);
		\fill (6,8) circle (5pt);
		\fill (7,8) circle (5pt);
		
		\fill (1,0) circle (5pt);
		\fill (2,-1) circle (5pt);
		\fill (3,-2) circle (5pt);
		\fill (4,-3) circle (5pt);
		
		\path[decoration=arrows, decorate] (1,8) --++ (0,-1) --++ (0,-1) --++ (1,0) --++ (0,-1) --++ (0,-1) --++ (0,-1) --++ (0,-1) --++ (0,-1) --++ (-1,-1);
		
		\path[decoration=arrows, decorate] (3,8) --++ (0,-1) --++ (0,-1) --++ (0,-1) --++ (0,-1) --++ (0,-1) --++ (1,0) --++ (0,-1) --++ (0,-1) --++ (-1,-1) --++ (-1,-1);
		
		\path[decoration=arrows, decorate] (4,8) --++ (0,-1) --++ (0,-1) --++ (0,-1) --++ (1,0) --++ (0,-1) --++ (0,-1) --++ (0,-1) --++ (0,-1) --++ (-1,-1) --++ (0,-1) --++ (-1,-1);
		
		\path[decoration=arrows, decorate] (6,8) --++ (0,-1) --++ (0,-1) --++ (0,-1) --++ (0,-1) --++ (0,-1) --++ (0,-1) --++ (0,-1) --++ (1,0) --++ (-1,-1) --++ (-1,-1) --++ (-1,-1) --++ (0,-1);

	\end{tikzpicture}
	
	\caption{Example of the lattice paths in Theorem~\ref{interpret3} for $n=4$. The weight of this family of nonintersecting lattice paths is $(-uv)^{3-1} w^2 \allowbreak (u X_1) (u X_2) (u X_{3}) (v X_{3}^{-1}) X_1^3 X_2^3 X_3^3 X_4^3$, which is equal to $u^5 v^3 w^2 X_1^4 X_2^4 X_3^3 X_4^3$.}
	\label{fig:ExampleFirstOfThirdInterpretation}
\end{figure}

\begin{proof}[Proof of Theorem~\ref{interpret3}]
	
	We interpret the Jacobi--Trudi-type expression~\eqref{jacobitrudi3} derived in Appendix~\ref{ap:computational_proof} in terms of the following families of $n$ nonintersecting lattice paths: We consider lattice paths from $A_i=\{(n-i+1,2n),(n+i-1,2n)\}$ to $E_j=(j,-j+1)$ for $1 \le i,j \le n$, see Figure~\ref{fig:ExampleFirstOfThirdInterpretation}. The step set as well as the edge weights depend on whether or not we are below the line $y=1$. 
	\begin{itemize} 
		\item Above and on the line $y=1$, the step set is $\{(1,0),(0,-1)\}$. Horizontal steps at height~$d$ above the $x$-axis have weight $u X_d$ if $d \le n$ and the weight $v X_{d-n}^{-1}$ if $d>n$. Assuming that $(k,1)$ is the last lattice point on the line $y=1$, the generating function of lattice paths from 
		$(n-i+1,2n)$ to $(k,1)$ is 
		\begin{equation*}
			h_{k+i-n-1}(u X_1,\ldots,u X_n,v X_1^{-1},\ldots,v X_n^{-1}),
		\end{equation*} 
		while the generating function of those lattice paths from $(n+i-1,2n)$ to $(k,1)$ is 
		\begin{equation*}
			h_{k-i-n+1}(u X_1,\ldots,u X_n,v X_1^{-1},\ldots,v X_n^{-1}).
		\end{equation*}
		Paths that start in $(n+i-1,2n)$ get an additional weight $(-uv)^{i-1}$, where the sign will be explained below.
		\item Below the line $y=1$, the step set is $\{(-1,-1),(0,-1)\}$, and since we want to reach $(j,-j+1)$, there are $k-j$ steps of type $(-1,-1)$ and $2j-k$ steps of type $(0,-1)$, which gives in total $\binom{j}{k-j}$ choices. The latter steps carry the weight $w$. 
		\item We have to multiply with an overall factor of $\prod_{i=1}^{n} X_i^{n-1}$.
	\end{itemize} 
	Note that the factor $\frac{1}{2}$ is taken into account by the fact that $A_1$ consists of a single point.
	
	The nonintersecting property does not force that $A_i$ is connected to $E_i$ by a path, so we still need to take the sign into account: Suppose we have the starting points $(n-i+1,2n)$ precisely for $n \ge i_p > i_{p-1} > \ldots > i_1=1$ and starting points $(n+i-1,2n)$ for $1 < j_1 < j_2 < \ldots < j_{n-p} \le n$. Then the only permutation for which the $n$-tuple of paths can be nonintersecting is
	\begin{equation*}
		(i_p,i_{p-1},\ldots,i_1,j_1,\ldots,j_{n-p})^{-1},
	\end{equation*}
	the sign of which is 
	$(-1)^{(i_1-1)+(i_2-1)+\ldots+(i_p-1)}$. Thus, taking the factor $(-1)^{\binom{n}{2}}$ into account, each path with $(n+i-1,2n)$ as starting point contributes $(-1)^{i-1}$ to the sign.
\end{proof}

\section{Acknowledgment} 

We thank Florian Schreier-Aigner for insightful discussions and the two anonymous referees for their helpful comments and suggestions.

\appendix

\section{\texorpdfstring{Alternative proof of Theorem~\ref{interpret4}}{Alternative proof of Theorem~2.2}}\label{ap:computational_proof}

The proof of Theorem~\ref{interpret4} in Section~\ref{sec:third} deduces the 
result from Theorem~\ref{interpret1}. We provide a direct proof in this appendix. To this end, we first derive a Jacobi--Trudi-type expression of \eqref{bialternant} which is different from the one in Section~\ref{sec:third}. We then interpret this expression as families of nonintersecting lattice paths.

\subsection{Another Jacobi--Trudi-type expression} 

\begin{lemma} 
	\label{jacobitrudi3lem}
	In the following identities, the argument of all complete homogeneous symmetric functions~$h$ is $(u X_1,\ldots,u X_n,v X_1^{-1},\ldots,v X_n^{-1})$:
	\begin{enumerate}
		\item\label{lem:jacobitrudi3lem1} For $n \ge 1$, the following identity holds: 
		\begin{multline*} 
				\det\limits_{1 \le i,j \le n} \left( \left(u^2 X_i^2 + u w X_i\right)^j + \left(v^2 X_i^{-2} + v w X_i^{-1}\right)^j \right)
				\det\limits_{1 \le i,j \le n} \left( \left(u^2 X_i^2 + u w X_i\right)^j - \left(v^2 X_i^{-2} + v w X_i^{-1}\right)^j \right)\\
			\times \prod_{1 \le i < j \le n} (X_j-X_i)^{-1} (u^{-1} v X_j^{-1}-u^{-1}v X_i^{-1})^{-1} 
				\prod_{i,j=1}^{n}  (u^{-1} v X_j^{-1}-X_i)^{-1} \\
			= \frac{(-1)^n u^{n^2+\binom{n}{2}} v^{-\binom{n}{2}}}{2} 
			\det_{1 \le i, j \le n } \left( \sum_{k=j}^{2j} \binom{j}{k-j} w^{2j-k}  \left( h_{k-i+1} - u^{-i+1+n} v^{-i+1+n} h_{k+i-2n-1} \right) \right) 
			\\
			\times \det_{1 \le i,j \le n} \left(  \sum_{k=j}^{2j} \binom{j}{k-j} w^{2j-k} 
			(h_{k+i-n-1} + u^{i-1} v^{i-1} h_{k-i-n+1} )  \right).
		\end{multline*} 
		\item\label{lem:jacobitrudi3lem2} For $n \ge 1$, the following identity holds:
		\begin{multline*}
			\det\limits_{1 \le i,j \le n} \left( \left(u^2 X_i^2 + u w X_i\right)^{j-1} + \left(v^2 X_i^{-2} + v w X_i^{-1}\right)^{j-1} \right) \\
			\times\det\limits_{1 \le i,j \le n} \left( \left(u^2 X_i^2 + u w X_i\right)^j - \left(v^2 X_i^{-2} + v w X_i^{-1}\right)^j \right) \\
			\times\prod_{1 \le i < j \le n} (X_j-X_i)^{-1} (u^{-1} v X_j^{-1}-u^{-1}v X_i^{-1})^{-1} 
			\prod_{i,j=1}^{n}  (u^{-1} v X_j^{-1}-X_i)^{-1}\\
			= (-1)^n u^{n^2+\binom{n}{2}} v^{-\binom{n}{2}} 
			\det_{1 \le i, j \le n-1 } \left( \sum_{k=j}^{2j} \binom{j}{k-j} w^{2j-k}  \left( h_{k-i} - u^{-i+n} v^{-i+n} 
			h_{k+i-2n} \right) \right) 
			\\
			\times 
			\det_{1 \le i,j \le n} \left(  \sum_{k=j}^{2j} \binom{j}{k-j} w^{2j-k} 
			(h_{k+i-n-1}+ u^{i-1} v^{i-1}  h_{k-i-n+1} \right).
		\end{multline*}
	\end{enumerate} 
\end{lemma} 

Note that we set the evaluation of the ``empty'' determinant to equal $1$ in the case $n=1$.

\begin{proof}
	{\it Re \eqref{lem:jacobitrudi3lem1}.}
	We consider the product 
	\begin{equation*}
	\det_{1 \le i,j \le n} \left( \left(u^2 X_i^2 + u w X_i\right)^j - \left(v^2 X_i^{-2} + v w X_i^{-1}\right)^j \right)
	\det_{1 \le i,j \le n} \left( \left(u^2 X_i^2 + u w X_i\right)^j + \left(v^2 X_i^{-2} + v w X_i^{-1}\right)^j \right).
	\end{equation*}
	For square matrices $A,B$ of the same size, we have 
	\begin{equation*}
	\renewcommand*{\arraystretch}{1.2}
	\det(A-B) \det(A+B) = 
	\det \left( \begin{array}{@{}c|c@{}} A-B & B \\ \hline  0 & A+B \end{array} \right) =
	\det \left( \begin{array}{@{}c|c@{}} A-B & B \\ \hline  B-A & A \end{array} \right) =
	\det \left( \begin{array}{@{}c|c@{}} A & B \\ \hline  B & A \end{array} \right),
	\end{equation*}
	and this implies that our product is equal to 
	\begin{equation*}
	\det \left( \begin{array}{@{}c@{}|c@{}} 
		\left( \left(u^2 X_i^2 + u w X_i\right)^j \right)_{1 \le i, j \le n} & \left( \left(v^2 X_i^{-2} + v w X_i^{-1}\right)^j \right)_{1 \le i,j \le n} 
		\xmathstrut{0.4} \\ \hline  
		\left( \left(v^2 X_i^{-2} + v w X_i^{-1}\right)^j \right)_{1 \le i,j \le n} & \left( \left(u^2 X_i^2 + u w X_i\right)^j \right)_{1 \le i,j \le n} \xmathstrut{0.4}
	\end{array} \right). 
	\end{equation*}
	Setting $X_{n+i} = u^{-1} v X_i^{-1}$, we can also write this as
	\begin{equation*}
	\det \left( \begin{array}{@{}c@{}|c@{}} 
		\left( \left(u^2 X_i^2 + u w X_i\right)^j \right)_{\substack{1 \le i \le 2n \\ 1 \le j \le n}} & \left( \left(v^2 X_i^{-2} + v w X_i^{-1}\right)^j \right)_{\substack{1 \le i \le 2n \\ 1 \le j \le n}} \xmathstrut{0.4}
	\end{array} \right). 
	\end{equation*}
	We apply Lemma~\ref{lem:det} with $Y_i=X_i$ to
	\begin{equation} 
		\label{applyto} 
		\frac{\det \left( \begin{array}{@{}c@{}|c@{}} 
				\left( \left(u^2 X_i^2 + u w X_i\right)^j \right)_{\substack{1 \le i \le 2n \\ 1 \le j \le n}} & \left( \left(v^2 X_i^{-2} + v w X_i^{-1}\right)^j \right)_{\substack{1 \le i \le 2n \\ 1 \le j \le n}} \xmathstrut{0.4}
			\end{array} \right)}{\prod_{1 \le i < j \le 2n} (X_j-X_i)}
	\end{equation} 
	and obtain 
	\begin{equation*}  
		\det \left( \begin{array}{@{}c@{}|c@{}} 
			\left( \sum\limits_{k \in \mathbb{Z}} \binom{j}{k-j} u^{k} w^{2j-k}  h_{k-i+1}(X_1,\ldots,X_i) \right)_{\substack{1 \le i \le 2n \\ 1 \le j \le n}} & 
			\left(  \sum\limits_{k \in \mathbb{Z}} \binom{j}{-k-j} v^{-k} w^{2j+k} h_{k-i+1}(X_1,\ldots,X_i) \right)_{\substack{1 \le i \le 2n \\ 1 \le j \le n}} \xmathstrut{0.8}
		\end{array} \right).
	\end{equation*} 
	We multiply from the left with the following matrix 
	\begin{equation*}
	(h_{j-i}(X_j,X_{j+1},\ldots,X_{2n}))_{1 \le i,j \le 2n}
	\end{equation*}
	with determinant $1$. For this purpose, note that  
	\begin{equation*}
	\sum_{l=1}^{2n} h_{l-i}(X_l,X_{l+1},\ldots,X_{2n}) h_{k-l+1}(X_1,\ldots,X_l)=h_{k-i+1}(X_1,\ldots,X_{2n}), 
	\end{equation*}
	and, therefore, the multiplication results in 
	\begin{equation*} 
			\det \left( \begin{array}{@{}c@{}|c@{}} 
				\left( \sum\limits_{k \in \mathbb{Z}} \binom{j}{k-j} u^{k} w^{2j-k}  h_{k-i+1}  \right)_{\substack{1 \le i \le 2n \\ 1 \le j \le n}} & 
				\left(  \sum\limits_{k \in \mathbb{Z}} \binom{j}{k-j} v^{k} w^{2j-k} h_{-k-i+1} \right)_{\substack{1 \le i \le 2n \\ 1 \le j \le n}} \xmathstrut{0.8}
			\end{array} \right),
	\end{equation*}
	where we omitted the arguments~$(X_1,\ldots,X_{2n})$ of the complete homogeneous symmetric functions~$h$.
		
	Using \eqref{negative}  and specializing $X_{n+i} = u^{-1} v X_i^{-1}$ for $i=1,2,\ldots,n$ gives
		\begin{equation*}
			(-1)^n
			\det \left( 
			\begin{array}{@{}c@{}} 
				\left( \sum\limits_{k=j}^{2j} \binom{j}{k-j} u^{i-1}  w^{2j-k}   
				h_{k-i+1}(u X_1,\ldots,u X_{n},v X_1^{-1},\ldots,v X_n^{-1})  \right)_{\substack{1 \le i \le 2n \\ 1 \le j \le n}}  \xmathstrut{1.2} \\\hline
				\left(  \sum\limits_{k=j}^{2j} \binom{j}{k-j} u^n v^{-i+1+n} w^{2j-k} h_{k+i-2n-1}(v X_1^{-1},\ldots,v X_n^{-1},u X_1\ldots,u X_{n}) 
				\right)_{\substack{1 \le i \le 2n \\ 1 \le j \le n}} \xmathstrut{1.2}
			\end{array} 
			\right)^T.
		\end{equation*} 
		Omitting again the arguments of the complete homogeneous symmetric functions~$h$ and some manipulations give 
		\begin{equation*} 
			(-1)^n
			\det \left( 
			\begin{array}{@{}c@{}|c@{}} 
				\left( \sum\limits_{k=j}^{2j} \binom{j}{k-j} u^{i-1+n}  w^{2j-k}   
				h_{k-i+1}  \right)_{\substack{1 \le i \le 2n \\ 1 \le j \le n}}  &
				\left(  \sum\limits_{k=j}^{2j} \binom{j}{k-j}  v^{-i+1+n} w^{2j-k} h_{k+i-2n-1}
				\right)_{\substack{1 \le i \le 2n \\ 1 \le j \le n}} \xmathstrut{1.2}
			\end{array} 
			\right).
		\end{equation*} 
		For $j=1,2,\ldots,n$, we subtract the $(n+j)$-th column multiplied by $u^{2n}$ from the $j$-th column.
		Then, for $i=n+2,n+3,\ldots,2n$, multiply the $(2n+2-i)$-th row with $u^{i-n-1} v^{-i+n+1}$ and add it to the $i$-th row. 
		After this, the entries in the first $n$ columns that are in row $n+1$ or below vanish. 
		This implies that the determinant is the product of the determinants of the upper left block and the lower right block, that is,
		\begin{multline*} 
			(-1)^n 
			\det_{1 \le i, j \le n } \left( \sum_{k=j}^{2j} \binom{j}{k-j} w^{2j-k}  \left(   u^{i-1+n} h_{k-i+1} - u^{2n} v^{-i+1+n} h_{k+i-2n-1} \right) \right) 
			\\
			\times 
			\det_{1 \le i,j \le n} \left(  \sum_{k=j}^{2j} \binom{j}{k-j} w^{2j-k} 
			(v^{-i+1}  h_{k+i-n-1} + [i \not=1] u^{i-1} h_{k-i-n+1} )  \right),
		\end{multline*} 
		where we make use of the \emph{Iverson bracket}: For a logical statement $S$, we set $\left[S\right]=1$ if $S$ is true and $\left[S\right]=0$ otherwise.
		
		The assertion then follows by taking the numerator of \eqref{applyto} as well as the specialization and some further manipulations into account.

		{\it Re \eqref{lem:jacobitrudi3lem2}.} We consider now the following product:
		\begin{multline*}
			\det_{1 \le i,j \le n} \left( 
			\left(u^2 X_i^2 + u w X_i\right)^{j-1} + \left(v^2 X_i^{-2} + v w X_i^{-1}\right)^{j-1}
			\right)
			\\
			\times
			\det_{1 \le i,j \le n} \left(  
			\left(u^2 X_i^2 + u w X_i\right)^j - \left(v^2 X_i^{-2} + v w X_i^{-1}\right)^j
			\right).
		\end{multline*}
		Let $A,B$ be $n \times (n-1)$ matrices and $c,d,e$ be column-vectors of length $n$, then we will use
		\begin{multline*}
			\renewcommand*{\arraystretch}{1.2}
			\det\left(\begin{array}{@{}c|c@{}}
				c & A+B
			\end{array}\right) \det\left(\begin{array}{@{}c|c@{}}
				A-B & e-d
			\end{array}\right) = 
			\det \left( \begin{array}{@{}c|c|c|c@{}} c & A+B  & B &d   \\ \hline  0 & 0  & A-B & e-d    \end{array} \right)  \\ \renewcommand*{\arraystretch}{1.2}
			= \det \left( \begin{array}{@{}c|c|c|c@{}} c & A+B  & B &d   \\ \hline  c & A+B  & A & e    \end{array} \right) =
			\det \left( \begin{array}{@{}c|c|c|c@{}} c & A & B &d   \\ \hline  c & B  & A & e    \end{array} \right),
		\end{multline*}
		and this implies that our product is equal to 
		\begin{equation*}
		2 \cdot \det \left( \begin{array}{@{}c@{}|c@{}} 
			\left( \left(u^2 X_i^2 + u w X_i\right)^{j-1} \right)_{1 \le i,j \le n} &  
			\left( \left(v^2 X_i^{-2} + v w X_i^{-1}\right)^j \right)_{1 \le i,j \le n} \xmathstrut{0.4} 
			\\ \hline  
			\left( \left(v^2 X_i^{-2} + v w X_i^{-1}\right)^{j-1} \right)_{1 \le i,j \le n} & 
			\left( \left(u^2 X_i^2 + u w X_i\right)^j \right)_{1 \le i,j \le n} \xmathstrut{0.4}
		\end{array} \right),  
		\end{equation*}
		where the $2$ compensates for $c=2$ since the column is now $1$. 
		Setting $X_{n+i} = u^{-1} v X_i^{-1}$, we can also write this is as
		\begin{equation*}
		2 \cdot \det \left( \begin{array}{@{}c@{}|c@{}} 
			\left( \left(u^2 X_i^2 + u w X_i\right)^{j-1} \right)_{\substack{1 \le i \le 2n \\ 1 \le j \le n}} &  
			\left( \left(v^2 X_i^{-2} + v w X_i^{-1}\right)^j \right)_{\substack{1 \le i \le 2n \\ 1 \le j \le n}}
			\xmathstrut{0.4} \end{array}  
		\right). 
		\end{equation*}
		We apply Lemma~\ref{lem:det} to
		\begin{equation} 
			\label{applyto2} 
			\frac{\det \left( \begin{array}{@{}c@{}|c@{}} 
					\left( \left(u^2 X_i^2 + u w X_i\right)^{j-1} \right)_{\substack{1 \le i \le 2n \\ 1 \le j \le n}} &  
					\left( \left(v^2 X_i^{-2} + v w X_i^{-1}\right)^j \right)_{\substack{1 \le i \le 2n \\ 1 \le j \le n}}
					\xmathstrut{0.4} \end{array}  
				\right)}{\prod_{1 \le i < j \le 2n} (X_j-X_i)}.
		\end{equation} 
		In a similar manner as above, after specializing again and omitting the arguments of the $h$'s, we obtain
		\begin{equation*} 
			(-1)^n
			\det \left( 
			\begin{array}{@{}c@{}|c@{}} 
				\left( \sum\limits_{k=j-1}^{2j-2} \binom{j-1}{k-j+1} u^{i-1+n}  w^{2j-2-k}   
				h_{k-i+1}  \right)_{\substack{1 \le i \le 2n \\ 1 \le j \le n}}  &
				\left(  \sum\limits_{k=j}^{2j} \binom{j}{k-j}  v^{-i+1+n} w^{2j-k} h_{k+i-2n-1}
				\right)_{\substack{1 \le i \le 2n \\ 1 \le j \le n}} \xmathstrut{1.2}
			\end{array} 
			\right),
		\end{equation*}
		which is equal to
		\begin{multline*} 
			(-1)^n
			\det \left( 
			\begin{array}{@{}c|c@{}|c@{}} 
				u^n \delta_{i,1} &
				\left( \sum\limits_{k=j}^{2j} \binom{j}{k-j} u^{i-1+n}  w^{2j-k}   
				h_{k-i+1}  \right)_{\substack{1 \le i \le 2n \\ 1 \le j \le n-1}}  &
				\left(  \sum\limits_{k=j}^{2j} \binom{j}{k-j}  v^{-i+1+n} w^{2j-k} h_{k+i-2n-1}
				\right)_{\substack{1 \le i \le 2n \\ 1 \le j \le n}} \xmathstrut{1.2}
			\end{array} 
			\right) \\
			= (-u)^n
			\det \left( 
			\begin{array}{@{}c@{}|c@{}} 
				\left( \sum\limits_{k=j}^{2j} \binom{j}{k-j} u^{i+n}  w^{2j-k}   
				h_{k-i}  \right)_{\substack{1 \le i \le 2n-1 \\ 1 \le j \le n-1}}  &
				\left(  \sum\limits_{k=j}^{2j} \binom{j}{k-j}  v^{-i+n} w^{2j-k} h_{k+i-2n}
				\right)_{\substack{1 \le i \le 2n-1 \\ 1 \le j \le n}} \xmathstrut{1.2}
			\end{array}
			\right),
		\end{multline*}
		where $\delta_{i,j} \coloneq \left[i=j\right]$ denotes the classical \emph{Kronecker delta}.  
		For $j=1,2,\ldots,n-1$, we subtract the $(n-1+j)$-th column multiplied by $u^{2n}$ from the $j$-th column.
		Then, for $i=n+1,n+2,\ldots,2n-1$, multiply the $(2n-i)$-th row with $u^{i-n} v^{-i+n}$ and add it to the $i$-th row. 
		After this, the entries in the first $n-1$ columns that are in row $n$ or below vanish. 
		This implies that the determinant is the product of the determinants of the upper left block and the lower right block, that is,
		\begin{multline*} 
			(-u)^n 
			\det_{1 \le i, j \le n-1 } \left( \sum_{k=j}^{2j} \binom{j}{k-j} w^{2j-k}  \left( u^{i+n}  h_{k-i} - u^{2n} v^{-i+n} 
			h_{k+i-2n} \right) \right) 
			\\
			\times 
			\det_{1 \le i,j \le n} \left(  \sum_{k=j}^{2j} \binom{j}{k-j} w^{2j-k} 
			(v^{-i+1} h_{k+i-n-1}+ u^{i-1}  h_{k-i-n+1} \right). 
		\end{multline*} 
		Taking the numerator of \eqref{applyto2} as well as the specialization into account and some further manipulations, the assertion follows.
	\end{proof}
	
	The following lemma is of use in the next corollary.	
	\begin{lemma} 
		\label{prefactor} 
		Let $n \ge 1$ and $0 \le m \le n-1$. Then we have 
		\begin{equation*}
		\sum_{l=1}^n  \frac{(-u X_l - v X_l^{-1})^{m}}{\prod_{j \not= l} (X_j-X_l)(u-v X_l^{-1} X_j^{-1})} =  
		\begin{cases}1, & m=n-1,\\0, & m<n-1. \end{cases}
		\end{equation*} 
	\end{lemma} 
	
	\begin{proof} As $(X_j-X_l)(u-v X_l^{-1} X_j^{-1}) = (u X_j + v X_j^{-1}) - (u X_l + v X_l^{-1})$, it suffices to show 
		\begin{equation*}
		\sum_{l=1}^n  \frac{(-Y_l)^{m}}{\prod_{j \not= l} (Y_j-Y_l)} =  
		\begin{cases}1, & m=n-1,\\0, & m<n-1. \end{cases}
		\end{equation*}
		Multiplying with $\prod_{1 \le i < j \le n} (Y_j-Y_i)$, this is equivalent to
		\begin{equation*}
		\sum_{l=1}^n  (-1)^{l+1} (-Y_l)^{m} \prod_{\substack{1 \le i < j \le n \\ i,j \not=l}} (Y_j-Y_i)  =  
		\begin{cases} \prod_{1 \le i < j \le n} (Y_j-Y_i),& m=n-1,\\0, & m<n-1. \end{cases}
		\end{equation*}
		Now, by the Vandermonde determinant evaluation, the left-hand side of this equation is 
		\begin{equation*}
		\det_{1 \le i,j \le n} \left( \begin{cases} Y_i^{j-1}, & j<n, \\  (-1)^{n-1} (-Y_i)^m, & j=n, \end{cases} \right)
		\end{equation*}
		and the assertion follows.
	\end{proof}

	\begin{cor} 
		\label{jacobitrudi3cor}
		In the following identities, the argument of all complete homogeneous symmetric functions~$h$ is $(u X_1,\ldots,u X_n,v X_1^{-1},\ldots,v X_n^{-1}).$
		\begin{enumerate}
			\item\label{cor:jacobitrudi3cor1} For $n \ge 1$, the following identity holds: 
			\begin{multline*}
				\frac{
					\det_{1 \le i,j \le n} \left( \left(u^2 X_i^2 + u w X_i\right)^j - \left(v^2 X_i^{-2} + v w X_i^{-1}\right)^j \right)}
				{\prod_{1 \le i < j \le n} (X_j-X_i)(v X_i^{-1} X_j^{-1}-u) \prod_{i=1}^{n} (u X_i -v X_i^{-1})} \\
				=
				\frac{1}{2} \det_{1 \le i,j \le n} \left(  \sum_{k=j}^{2j} \binom{j}{k-j} w^{2j-k} 
				(h_{k+i-n-1} + u^{i-1} v^{i-1} h_{k-i-n+1} )  \right). 
			\end{multline*} 
			\item\label{cor:jacobitrudi3cor2}For $n \ge 1$, the following identity holds:
			\begin{multline*}
				\frac{
					\det_{1 \le i,j \le n} \left( \left(u^2 X_i^2 + u w X_i\right)^j + \left(v^2 X_i^{-2} + v w X_i^{-1}\right)^j \right)}
				{\prod_{1 \le i < j \le n} (X_j-X_i)(u- v X_i^{-1} X_j^{-1})} \\ 
				=
				\det_{1 \le i, j \le n } \left( \sum_{k=j}^{2j} \binom{j}{k-j} w^{2j-k}  \left(  h_{k-i+1} - u^{-i+1+n} v^{-i+1+n} h_{k+i-2n-1} \right) \right).
			\end{multline*} 
			\item\label{cor:jacobitrudi3cor3} For $n \ge 1$, the following identity holds:
			\begin{multline*}
				\frac{1}{2}  \frac{\det_{1 \le i,j \le n} \left( \left(u^2 X_i^2 + u w X_i\right)^{j-1} + \left(v^2 X_i^{-2} + v w X_i^{-1}\right)^{j-1} \right)}
				{\prod_{1 \le i < j \le n} (X_j-X_i)(u - v X_i^{-1} X_j^{-1})} \\
				=
				\det_{1 \le i, j \le n-1} \left( \sum_{k=j}^{2j} \binom{j}{k-j} w^{2j-k}  \left( h_{k-i} - u^{-i+n} v^{-i+n} h_{k+i-2n} \right) \right).
			\end{multline*} 
		\end{enumerate} 
	\end{cor} 
	
	\begin{proof} 
		The proof is by induction on $n$. For small $n$, the identities can be shown by direct computations. We assume that 
		all identities are proved for $n-1$. 
		We start by proving the third identity. Now observe that, by expanding the determinant with respect to the first column and the induction 
		hypothesis applied to the case $n-1$ in \eqref{cor:jacobitrudi3cor2}, we obtain
		\begin{multline} 
			\label{step}
			\frac{1}{2}  \frac{\det_{1 \le i,j \le n} \left( \left(u^2 X_i^2 + u w X_i\right)^{j-1} + \left(v^2 X_i^{-2} + v w X_i^{-1}\right)^{j-1} \right)}
			{\prod_{1 \le i < j \le n} (X_j-X_i)(u - v X_i^{-1} X_j^{-1})} \\
			= \sum_{l=1}^n (-1)^{l+1} 
			\frac{\det_{\substack{1 \le i \le  n, i \not=l  \\ 1 \le j \le n-1}}
				\left( \left(u^2 X_i^2 + u w X_i\right)^{j} + \left(v^2 X_i^{-2} + v w X_i^{-1}\right)^{j} \right)}
			{\prod_{1 \le i < j \le n} (X_j-X_i)(u - v X_i^{-1} X_j^{-1})} \\
			= \sum_{l=1}^n \frac{1}{\prod_{j \not= l} (X_j-X_l)(u-v X_l^{-1} X_j^{-1})} 
			\det_{1 \le i, j \le n-1} \left( \sum_{k=j}^{2j} \binom{j}{k-j} w^{2j-k}  \left( h^{(l)}_{k-i+1} - u^{-i+n} v^{-i+n} h^{(l)}_{k+i-2n+1} \right) \right), 
		\end{multline} 
		where the argument of the complete homogeneous symmetric functions in the $l$-th summand is 
		\begin{equation}
			\label{argument}
			(u X_1,\ldots,\widehat{u X_l},\ldots,u X_n,v X_1^{-1},\ldots,\widehat{v X_l^{-1}},\ldots,v X_n^{-1}),
		\end{equation} 
		which is indicated by $h^{(l)}$. We will use 
		\begin{multline*} 
			h_k(X_1,\ldots,\widehat{X_p},\ldots,\widehat{X_q},\ldots,X_n) = h_k(X_1,\ldots,\widehat{X_q},\ldots,X_n) - X_p \, h_{k-1}(X_1,\ldots,\widehat{X_q},\ldots,X_n) \\
			= h_k(X_1,\ldots,X_n) - (X_p+X_q) h_{k-1}(X_1,\ldots,X_n) + X_p X_q h_{k-2}(X_1,\ldots,X_n).
		\end{multline*} 
		Applied to our argument in \eqref{argument}, we see that 
		\begin{equation*}
		h_k^{(l)}= h_k - (u X_l + v X_l^{-1}) h_{k-1} + u v\,  h_{k-2},
		\end{equation*}
		and, therefore, \eqref{step} is further equal to 
		\begin{multline}
			\label{eq:splittedhsum}
			\sum_{l=1}^n \frac{1}{\prod_{j \not= l} (X_j-X_l)(u-v X_l^{-1} X_j^{-1})} 
			\det_{1 \le i, j \le n-1} \left( \sum_{k=j}^{2j} \binom{j}{k-j} w^{2j-k} \right. \\
			\times ( h_{k-i+1} - u^{-i+n+1} v^{-i+n+1} h_{k+i-2n-1} 
			- (u X_l + v X_l^{-1}) (h_{k-i} - u^{-i+n} v^{-i+n} h_{k+i-2n})  \\
			\left. \phantom{\sum \binom{n}{k}} \left. + u v (h_{k-i-1} - u^{-i+n-1} v^{-i+n-1} h_{k+i-2n+1}) \right) \right).
		\end{multline} 
		Setting 
		\begin{equation*}
		a_{i,j} = \sum_{k=j}^{2j} \binom{j}{k-j} w^{2j-k} (h_{k-i} - u^{-i+n} v^{-i+n} h_{k+i-2n}),
		\end{equation*}
		the determinant is equal to
		\begin{equation*}
		\det_{1 \le i,j \le n-1} \left( a_{i-1,j} - (u X_l + v X_l^{-1}) a_{i,j} + u v a_{i+1,j} \right).
		\end{equation*}
		By the linearity in the rows, we can write each of the $n$ determinants in \eqref{eq:splittedhsum} as a sum of $3^{n-1}$ determinants, by choosing in a particular row $i$ either of the three terms, namely $a_{i-1,j}$, $-(u X_l + v X_l^{-1}) a_{i,j}$ or $u v a_{i+1,j}$. We pull out $u X_l + v X_l^{-1}$ whenever we choose the second term, so that the determinant is independent of $l$. Across the different choices of $l$, we combine the determinants where in each row we have made precisely the same choices among the three terms (all these determinants are of course equal since we pulled out the appropriate power of $u X_l + v X_l^{-1}$). If we have chosen $m$ times the second term, then we have a prefactor that is equal to the left-hand side of the identity in Lemma~\ref{prefactor}. By the lemma, the prefactor is only nonzero if $m=n-1$. This establishes \eqref{cor:jacobitrudi3cor3} for $n$.
		
		Now \eqref{cor:jacobitrudi3cor1} follows from \eqref{cor:jacobitrudi3cor3} and Lemma~\ref{jacobitrudi3lem}~\eqref{lem:jacobitrudi3lem2}, whereas \eqref{cor:jacobitrudi3cor2} finally follows from \eqref{cor:jacobitrudi3cor1} and Lemma~\ref{jacobitrudi3lem}~\eqref{lem:jacobitrudi3lem1}.
	\end{proof} 

\subsection{\texorpdfstring{Second proof of Theorem~\ref{interpret4}}{Second proof of Theorem 2.2}}
\label{sec:Prooffourthinterpretation}

Corollary~\ref{jacobitrudi3cor}~\eqref{cor:jacobitrudi3cor1} implies that the generating function in \eqref{bialternant} is equal to
\begin{equation}
	\label{jacobitrudi3} 
	\frac{(-1)^{\binom{n}{2}}}{2} \prod_{i=1}^n X_i^{n-1} \det_{1 \le i,j \le n} \left(  \sum_{k=j}^{2j} \binom{j}{k-j} w^{2j-k} 
	(h_{k+i-n-1} + u^{i-1} v^{i-1} h_{k-i-n+1} )  \right),
\end{equation} 
where the arguments $(u X_1,\ldots,u X_n,v X_1^{-1},\ldots,v X_n^{-1})$ of the complete homogeneous symmetric functions are omitted. We will multiply the matrix underlying the determinant in \eqref{jacobitrudi3} from the left with 
the following matrix of determinant $\frac{1}{2}$:
\begin{equation}
	\label{eq:onehalfdet}
	\left( (-1)^{i+j} \frac{1}{1+[j=1]} e_{i-j}(u X_{n-i+2},\ldots,u X_n,v X_{n-i+2}^{-1},\ldots,v X_n^{-1}) \right)_{1 \le i,j \le n}.
\end{equation}
For this purpose, we need the following lemma.
\begin{lemma} For $n \ge 1, m+n \ge 0$ and $1 \le i \le n$, we have 
	\begin{multline*} 
		\sum_{l=1}^n (-1)^{i+l} \frac{1}{1+[l=1]} e_{i-l}(u X_{n-i+2},v X_{n-i+2}^{-1},\ldots,u X_n,v X_n^{-1}) \\
		\times \left[ h_{m+l-1}(u X_1,v X_1^{-1},\ldots,u X_n,v X_n^{-1}) 
		+ (u v)^{l-1} h_{m-l+1}(u X_1,v X_1^{-1},\ldots,u X_n,v X_n^{-1}) \right]  \\
		= h_{m+i-1}(u X_1,v X_1^{-1},\ldots,u X_{n-i+1},v X_{n-i+1}^{-1}).
	\end{multline*} 
\end{lemma} 

\begin{proof} The left-hand side can be written as 
	\begin{multline*} 
		\sum_{l=0}^{i-1} (-1)^{i+l-1} e_{i-1-l}(u X_{n-i+2},v X_{n-i+2}^{-1},\ldots,u X_n,v X_n^{-1}) h_{m+l}(u X_1,v X_1^{-1},\ldots,u X_n,v X_n^{-1}) \\
		+ \sum_{l=1}^{i-1} (-1)^{i+l-1} (u v)^{l} e_{i-1-l}(u X_{n-i+2},v X_{n-i+2}^{-1},\ldots,u X_n,v X_n^{-1})  h_{m-l}(u X_1,v X_1^{-1},\ldots,u X_n,v X_n^{-1}). 
	\end{multline*} 
	The result is an immediate consequence of two identities which state that
	\begin{multline} 
		\label{id1}
		h_{m+i-1}(u X_1,v X_1^{-1},\ldots,u X_{n-i+1},v X_{n-i+1}^{-1})\\
		= \sum_{l=-i+1}^{i-1} (-1)^{i+l-1} e_{i-1-l}(u X_{n-i+2},v X_{n-i+2}^{-1},\ldots,u X_n,v X_n^{-1}) h_{m+l}(u X_1,v X_1^{-1},\ldots,u X_n,v X_n^{-1})
	\end{multline}
	and that the following vanishes
	\begin{multline}
		\label{id2}
		\sum_{l=-i+1}^{-1} (-1)^{i+l} e_{i-1-l}(u X_{n-i+2},v X_{n-i+2}^{-1},\ldots,u X_n,v X_n^{-1}) h_{m+l}(u X_1,v X_1^{-1},\ldots,u X_n,v X_n^{-1}) \\
		+ \sum_{l=1}^{i-1} (-1)^{i+l-1} (u v)^{l} e_{i-1-l}(u X_{n-i+2},v X_{n-i+2}^{-1},\ldots,u X_n,v X_n^{-1})  h_{m-l}(u X_1,v X_1^{-1},\ldots,u X_n,v X_n^{-1}).
	\end{multline}
	In order to show the first identity, observe that the right-hand side of \eqref{id1} can be written as 
	\begin{multline*} 
		\sum_{l=-i+1}^{i-1} \langle t^{i-l-1} \rangle \left[ 
		\prod_{p=n-i+2}^n (1-t u X_p)(1-t v X_p^{-1}) \right]  
		\langle t^{l+m} \rangle \left[ \prod_{p=1}^n (1-t u X_p)^{-1}(1-t v X_p^{-1})^{-1} \right] \\
		= \langle t^{i+m-1} \rangle \left[ \prod_{p=1}^{n-i+1} (1-t u X_p)^{-1}(1-t vX_p^{-1})^{-1} \right] \\ = h_{i+m-1}(u X_1,\ldots,u X_{n-i+1},v X_1^{-1},\ldots,v X_{n-i+1}^{-1}).
	\end{multline*} 
	The second identity follows after showing that the $j$-th summand in the first sum in 
	\eqref{id2} is the negative of the $(i-j)$-th summand  in the second sum, that is,
	\begin{multline*}
		(-1)^{j} e_{2i-j-1}(u X_{n-i+2},v X_{n-i+2}^{-1},\ldots,u X_n,v X_n^{-1}) h_{m-i+j}(u X_1,v X_1^{-1},\ldots,u X_n,v X_n^{-1})= \\
		-[(-1)^{j+1} (u v)^{i-j} e_{j-1}(u X_{n-i+2},u X_{n-i+2}^{-1},\ldots,v X_n,v X_n^{-1})  h_{m-i+j}(u X_1,v X_1^{-1},\ldots,u X_n,v X_n^{-1})], 
	\end{multline*} 
	for $j=1,2,\ldots,i-1$, which is equivalent to 
	\begin{equation*}
	e_{2i-j-1}(u Y_1,v Y_1^{-1},\ldots,u Y_{i-1},v Y_{i-1}^{-1}) 
	= (u v)^{i-j} e_{j-1}(u Y_1,v Y_1^{-1},\ldots,u Y_{i-1},v Y_{i-1}^{-1})
	\end{equation*} 
	by setting $Y_k=X_{k+n-i+1}$ for $k=1,2,\dots,i-1$.
	
	The identity follows as, in each monomial of the expression on the left-hand side, we choose at least 
	$2i-j-1-(i-1)=i-j$ pairs of variables $(u Y_k, v Y_k^{-1})$ and each such pair contributes a factor $u v$. 
\end{proof} 

The previous lemma implies that \eqref{jacobitrudi3} multiplied with \eqref{eq:onehalfdet} is equal to 
\begin{multline}
	\label{jacobitrudi4} 
	(-1)^{\binom{n}{2}}\prod_{i=1}^n X_i^{n-1} \det_{1 \le i,j \le n} \left(  \sum_{k=j}^{2j} \binom{j}{k-j} w^{2j-k} 
	h_{k+i-n-1}(u X_{1},v X_{1}^{-1},\ldots,u X_{n+1-i},v X_{n+1-i}^{-1}) \right) \\
	= \prod_{i=1}^n X_i^{n-1} \det_{1 \le i,j \le n} \left(  \sum_{k=j}^{2j} \binom{j}{k-j} w^{2j-k} 
	h_{k-i}(u X_{1},v X_{1}^{-1},\ldots,u X_{i},v X_{i}^{-1}) \right).
\end{multline} 

We consider families of $n$ nonintersecting lattice paths from $A_i=(i,2i)$ to $E_j=(j,-j+1)$ for $1 \le i,j \le n$, see Figure~\ref{fig:ExampleSecondOfThirdInterpretation}. The step set as well as the edge weights depend on whether or not we are below the line $y=1$. 
\begin{itemize} 
	\item Above and on the line $y=1$, the step set is $\{(1,0),(0,-1)\}$. Horizontal steps at height 
	$1,2,3,4,\ldots,2n$ have weight $u X_1, v X_1^{-1},u X_{2}, v X_{2}^{-1},\ldots, u X_n, 
	v X_n^{-1}$, respectively.  Assuming that $(k,1)$ is the last lattice point on the line $y=1$, the generating function of such lattice paths from 
	$(i,2i)$ to $(k,1)$ is 
	\begin{equation*}
		h_{k-i}(u X_{1},v X_{1}^{-1},\ldots,u X_i^{-1},v X_i^{-1}),
	\end{equation*} 
	\item Below the line $y=1$, the step set is $\{(-1,-1),(0,-1)\}$, and since we want to reach $(j,-j+1)$, there are $k-j$ steps of type $(-1,-1)$ and $2j-k$ steps of type $(0,-1)$, which gives in total $\binom{j}{k-j}$ choices. The latter steps carry the weight $w$. Equivalently, we can also choose that the $(-1,-1)$-steps are equipped with the weight $w^{-1}$ and that there is an overall factor of $w^{\binom{n+1}{2}}$.
	\item Again, we have to multiply with the overall factor $\prod_{i=1}^n X_i^{n-1}$.  
\end{itemize} 

\begin{figure}[htb]
	\centering
	\begin{tikzpicture}[scale=.45,baseline=(current bounding box.center)]
		\fill [light-gray] (0,-5.75) rectangle (10.75,1);
		
		\draw [help lines,step=1cm,dashed] (0,-5.75) grid (10.75,12.75);
		
		\fill (1,2) circle (5pt);
		\fill (2,4) circle (5pt);
		\fill (3,6) circle (5pt);
		\fill (4,8) circle (5pt);
		\fill (5,10) circle (5pt);
		\fill (6,12) circle (5pt);
		
		\fill (1,0) circle (5pt);
		\fill (2,-1) circle (5pt);
		\fill (3,-2) circle (5pt);
		\fill (4,-3) circle (5pt);
		\fill (5,-4) circle (5pt);
		\fill (6,-5) circle (5pt);
		
		\draw[->,thick] (0,0)--(10.75,0) node[right]{$x$};
		\draw[->,thick] (0,-5.75)--(0,12.75) node[above]{$y$};
		
		\path[decoration=arrows, decorate] (1,2) --++ (1,0) --++ (0,-1) --++ (-1,-1);
		
		\path[decoration=arrows, decorate] (2,4) --++ (0,-1) --++ (1,0) --++ (0,-1) --++ (1,0) --++ (0,-1) --++ (-1,-1) --++ (-1,-1);
		
		\path[decoration=arrows, decorate] (3,6) --++ (0,-1) --++ (0,-1) --++ (1,0) --++ (1,0) --++ (0,-1) --++ (0,-1) --++ (0,-1) --++ (-1,-1) --++ (0,-1) --++ (-1,-1);
		
		\path[decoration=arrows, decorate] (4,8) --++ (0,-1) --++ (0,-1) --++ (1,0) --++ (0,-1) --++ (1,0) --++ (0,-1) --++ (0,-1) --++ (1,0) --++ (0,-1) --++ (0,-1) --++ (-1,-1) --++ (-1,-1) --++ (0,-1) --++ (-1,-1);
		
		\path[decoration=arrows, decorate] (5,10) --++ (0,-1) --++ (0,-1) --++ (1,0) --++ (0,-1) --++ (0,-1) --++ (1,0) --++ (1,0) --++ (0,-1) --++ (0,-1) --++ (1,0) --++ (0,-1) --++ (0,-1) --++ (0,-1) --++ (-1,-1) --++ (0,-1) --++ (-1,-1) --++ (-1,-1) --++ (-1,-1);
		
		\path[decoration=arrows, decorate] (6,12) --++ (0,-1) --++ (0,-1) --++ (0,-1) --++ (1,0) --++ (0,-1) --++ (0,-1) --++ (1,0) --++ (1,0) --++ (0,-1) --++ (0,-1) --++ (1,0) --++ (0,-1) --++ (0,-1) --++ (0,-1) --++ (0,-1) --++ (0,-1) --++ (-1,-1) --++ (0,-1) --++ (-1,-1) --++ (-1,-1) --++ (-1,-1);
		
	\end{tikzpicture}
	
	\caption{Example of the lattice path interpretation of \eqref{jacobitrudi4} for $n=6$. The weight of these paths is $u^7 v^9 w^5 X_1^{-2} X_2^{-1} X_3^{-1} X_4 X_5$. Note that this family of nonintersecting lattice paths corresponds to the pair of \CSPP and \RSPP in Section~\ref{sec:MainResults}.}
	\label{fig:ExampleSecondOfThirdInterpretation}
\end{figure}

To conclude this proof of Theorem~\ref{interpret4}, we argue somewhat similar as in Section~\ref{sec:third} where we derived the plane partition representation. The horizontal steps on and above the line $y=1$ correspond to the entries in the plane partition $P$. Each odd entry $2i-1$ contributes $u X_i$ multiplicatively to the weight, while each even entry $2i$ contributes $v X_i^{-1}$. The entries of $Q$ correspond to the diagonal steps left of the $y$-axis as follows: Similar to before, their distances from the line $y=x$ yield a \CSPP $Q'$. By the same manipulations as in the end of Section~\ref{sec:third}, that is, first subtracting $i-1$ from the entries in the $i$-th row from the bottom of $Q'$ and then subtracting $j-1$ from all entries in the $j$-th column of $Q'$, we finally obtain $Q$.

The pair $(P,Q)$ of the nonintersecting lattice paths from Figure~\ref{fig:ExampleSecondOfThirdInterpretation} is given in Section~\ref{sec:MainResults}. 

Note that when setting $u=v=1$, $w=-1$ and $X_i=1$ for $1 \le i \le n$ in \eqref{jacobitrudi4}, we obtain $\binom{i+j-1}{2j-i}=\binom{i+j-1}{2i-j-1}$ by \eqref{hspecial}  as entries of the determinant's underlying matrix and 
thus the unrefined count \eqref{unrefined} of cyclically and vertically symmetric lozenge tilings of a hexagon with a central triangular hole of size $2$.  

\bibliographystyle{alphaurl}
\bibliography{asmpp}

\end{document}